 \documentclass[a4paper,10pt,intlimits,sumlimits,final]{amsart}

\usepackage[utf8]{inputenc} 

\usepackage{
amsmath ,mathtools 
,amssymb ,amsfonts 
,amstext 
,xfrac 
,mathrsfs 
,bbm ,bm 
}
\usepackage{url}
\usepackage[bookmarks,colorlinks,breaklinks,hypertexnames=false]{hyperref}
\usepackage[hyphenbreaks]{breakurl}
\hypersetup{linkcolor=blue,citecolor=blue,filecolor=blue,urlcolor=blue}
\usepackage{cleveref} 
\usepackage{
tikz 
}
\usetikzlibrary{ 
trees
}

\newtheorem{theorem}{Theorem}
\newtheorem{theorem*}{Theorem}
\newtheorem{definition}[theorem]{Definition}
\newtheorem{definition*}{Definition}
\newtheorem{lemma}[theorem]{Lemma}
\newtheorem{lemma*}{Lemma}
\newtheorem{corollary}[theorem]{Corollary}
\newtheorem{corollary*}{Corollary}
\newtheorem{proposition}[theorem]{Proposition}
\newtheorem{proposition*}{Proposition}
\theoremstyle{remark}
\newtheorem{remark}[theorem]{Remark}
\newtheorem{remark*}[theorem]{Remark}

\numberwithin{theorem}{section} 
\numberwithin{equation}{section} 
\newcommand\eqnum{\stepcounter{equation}\tag{\theequation}}

\usepackage{enumerate} 

\usepackage[notcite, notref]{showkeys} 
\usepackage[normalem]{ulem} 
\usepackage[obeyFinal]{todonotes} 

\usepackage[citestyle=alphabetic,bibstyle=alphabetic,maxalphanames=5,backend=bibtex]{biblatex} 
\newcommand{\mc}[1]{\mathcal{#1}} 
\newcommand{\mf}[1]{\mathfrak{#1}} 
\newcommand{\mbb}[1]{\mathbb{#1}} 
\newcommand{\mb}[1]{\mathbf{#1}} 
\newcommand{\mrm}[1]{\mathrm{#1}} 
\newcommand{\ms}[1]{\mathscr{#1}} 
\newcommand{\N}{\mathbb{N}} 
\newcommand{\Z}{\mathbb{Z}} 
\newcommand{\R}{\mathbb{R}} 
\newcommand{\C}{\mathbb{C}} 
\newcommand{\linspan}{\mathrm{span}} 
\newcommand{\Id}{\mathop{\kern0pt \mrm{Id}} \mathopen{}} 
\newcommand{\diam}{\mathop{\kern0pt \mrm{diam}} \mathopen{}} 
\newcommand{\spt}{\mathop{\kern0pt \mrm{spt}} \mathopen{}} 
\newcommand{\opclo}{\mathop{\kern0pt \mrm{clo}} \mathopen{}} 
\newcommand{\opint}{\mathop{\kern0pt \mrm{int}} \mathopen{}} 
 \newcommand{\eqd}{\coloneqq} 
\newcommand{\dd}{\mathop{\kern0pt \mathrm{d}} \mathopen{}} 
\newcommand{\st}{\colon} 
\newcommand{\Fourier}{\mathop{\kern0pt \mc{F}} \mathopen{}} 
\newcommand{\FT}[1]{\widehat{#1}} 
\newcommand{\E}{\mathop{\kern0pt \mathbb{E}} \mathopen{}} 
\newcommand{\1}{\mathbbm{1}} 
\newcommand{\mcl}{\mathclap}
\newcommand{\mrl}{\mathrlap}

\newcommand{\nquad}{\hspace{-1em}} 
\newcommand{\nqquad}{\hspace{-2em}} 
\newcommand{\LHS}[1]{\ensuremath{\mathop{\kern0pt \mathrm{LHS}} \mathopen{}#1}}
\newcommand{\RHS}[1]{\ensuremath{\mathop{\kern0pt \mathrm{RHS}} \mathopen{}#1}}
\renewcommand{\bar}[1]{\overline{#1}} 
\renewcommand{\tilde}[1]{\widetilde{#1}} 
\renewcommand{\phi}{\varphi} 
\renewcommand{\epsilon}{\varepsilon} 
\newcommand{\Lap}{\mathop{\kern0pt \ms{L}}\mathopen{}} 
\newcommand{\xs}{\mf{s}} 
\newcommand{\LP}{\mathop{\kern0pt \mathrm{P}}\hspace{-0.10em}\mathopen{}} 
\newcommand{\QP}{\mathop{\kern0pt \mathrm{Q}}\mathopen{}} 
\newcommand{\UP}{\mathop{\kern0pt \mathrm{U}}\hspace{-0.10em}\mathopen{}} 
\newcommand{\rf}{\mf{f}} 
\newcommand{\rz}{\mf{z}} 
\newcommand{\TT}{\mbb{T}} 

\begin{document}
\title{Higher order expansion for the probabilistic local well-posedness
  theory for a cubic nonlinear Schrödinger equation}
\renewcommand{\shorttitle}{Local well-posedness for cubic NLS}


\author{Jean-Baptiste Casteras}
\address{CMAFcIO, Faculdade de Ciências da Universidade de Lisboa,
  Edificio C6, Piso 1, Campo Grande 1749-016 Lisboa, Portugal}
\email{jeanbaptiste.casteras@gmail.com}
\thanks{J.-B.C. supported by FCT — Fundação para a Ciência e a
  Tecnologia, under the project: UIDB/04561/2020}
\providecommand{\todoJB}[1]{\todo[color=blue!40]{#1}} 

\author{Juraj Földes}
\address{Dept.\ of Mathematics, University of Virginia, Kerchof Hall,
  Charlottesville, VA 22904-4137}
\email{foldes@virginia.edu}
\thanks{J. F. was partially supported by the grant NSF-DMS-1816408}
\providecommand{\todoJF}[1]{\todo[color=green!40]{#1}} 

\author{Gennady Uraltsev}
\address{Dept.\ of Mathematical Sciences, University of Arkansas,
  Fayetteville, AA 72701}
\email{gennady.uraltsev@gmail.com}
\providecommand{\todoGU}[1]{\todo[color=red!40]{#1}} 

\begin{abstract}
In this paper, we study the probabilistic local well-posedness of the
cubic Schrödinger equation (cubic NLS):
\[
(i\partial_{t}  + \Delta) u = \pm |u|^{2} u \text{ on } [0,T) \times \R^{d},
\]
with initial data being a Wiener randomization at unit scale of a
given function $f$.  We prove that a solution exists almost-surely
locally in time provided \(f\in H^{S}_{x}(\R^{d})\) with
\(S>\max\big(\frac{d-3}{4},\frac{d-4}{2}\big)\) for \(d\geq 3\). In particular, we
establish that the local well-posedness holds for any \(S>0\) when
\(d=3\). We also show that, under appropriate smallness conditions for
the initial data, the solutions are global in time and scatter.

The solutions are constructed as a sum of an explicit multilinear
expansion of the flow in terms of the random initial data and of an
additional smoother remainder term with deterministically subcritical
regularity. This construction allows us to introduce a
new and refined notion of graded scattering.

We develop the framework of directional space-time norms to control
the (probabilistic) multilinear expansion and the (deterministic)
remainder term and to obtain improved bilinear
probabilistic-deterministic Strichartz estimates.
\end{abstract}

{\maketitle}

\section{Introduction}\label{sec:introduction}

In this paper we study the almost-sure local well-posedness of the
cubic nonlinear Schrödinger equation:
\[\eqnum\label{eq:NLS-deterministic}
\begin{dcases}
(i\partial_t + \Delta) u = \pm |u|^2 u \qquad \text{on} \quad [0,T)\times \R^d ,
\\
u(0)=f \in H_x^S (\R^d) , 
\end{dcases}
\]
in one temporal and \(d \geq 3\) spatial dimensions.  We assume that the
initial datum \(f\) belongs to the Sobolev space $H^S_x(\R^d)$ with
\(S \geq 0\) (see \eqref{eq:Sobolev-spaces} below for the definition) and the
solution $u$ is a function in
\(C^{0}\big([0,T);H^{S}_{x}(\R^{d})\big)\) for some \(T>0\) that solves
\eqref{eq:NLS-deterministic} in a mild sense (see the discussion
above \eqref{eq:duhamel}). The equation is invariant
under the transformation
\[\eqnum\label{eq:scaling}
u_\lambda (t,x)=\lambda u(\lambda^2 t ,\lambda x),\quad \lambda>0,
\]
which also conserves the homogeneous Sobolev
norm $\dot{H}^{\xs_{c}}_x(\R^d)$, $\xs_{c}:= \frac{d-2}{2}$, of the initial condition \(u_\lambda(0,x)=\lambda f(\lambda x)\):
\[
\|u_\lambda (0)\|_{\dot{H}^{\xs_{c}}_{x}(\R^{d})}:=\Big(\int_{\R^N} \big|\Delta^{\xs_{c}/2}(\lambda f(\lambda x)) \big|^2 dx\Big)^{1/2} = \|f\|_{\dot{H}^{\xs_{c}}_{x}(\R^{d})}.
\]
This suggests that the regularity exponent \(\xs_{c}\) is critical for
the well-posedness of the equation.  Indeed, it is known that when
\(S\geq\xs_{c}\eqd \frac{d-2}{2}\) a solution to \eqref{eq:NLS-deterministic}
exists at least locally in time, that is,
\eqref{eq:NLS-deterministic} is locally well-posed, (see for instance
\cite{cazenaveSemilinearSchrodingerEquations2003,cazenaveCauchyProblemCritical1990,collianderGlobalWellposednessScattering2008,ryckmanGlobalWellposednessScattering2007,pausaderGlobalWellposednessEnergy2007,pausaderMasscriticalFourthorderSchrodinger2010}).
On the other hand, Christ, Colliander, and Tao in \cite{christIllposednessNonlinearSchrodinger2003}
established that \eqref{eq:NLS-deterministic} is locally ill-posed if
$S < \xs_c$. More precisely, there are solutions with an arbitrary small
initial datum in $H^S_x(\R^d)$ which are unbounded on arbitrarily small
time intervals.

In the present paper, we show that the ill-posedness is
probabilistically exceptional when \(S> S_{\mrm{min}}\) with
\(S_{\mrm{min}} < \xs_{c}\), given by
\[ \eqnum\label{eq:intro-Smin}
\begin{aligned}
S_{\mrm{min}}:= \begin{cases}
0 & \text{if } d=3,
\\
\frac{1}{4} & \text{if } d=4,
\\
\frac{d-4}{2} & \text{if } d\geq 5,
\end{cases}
\end{aligned}
\qquad
\begin{aligned}
\xs_{c}\eqd \frac{d-2}{2}.
\end{aligned}
\]
In other words, if we randomly choose the initial condition
\(f\in H^{S}_{x}(\R^{d})\), then with probability one, problem
\eqref{eq:NLS-deterministic} is locally well-posed.

The deterministic nonlinear Schrödinger equation
\eqref{eq:NLS-deterministic} has been the subject of extensive study
since the late 1970s due to its physical relevance as a model for
dispersive wave-like systems, especially in their asymptotic regimes,
for example many interacting Bosonic particle systems (Bose-Einstein
condensates), non-linear optics, and small water waves (see
e.g. \cite{erdosDerivationNonlinearSchrodinger2001,craigNonlinearModulationGravity1992}). The
equation \eqref{eq:NLS-deterministic} is also a model problem for a large
class of ``dispersive'' PDEs, and it is a standard example of an
infinite dimensional Hamiltonian system
\cite{mendelsonRigorousDerivationHamiltonian2020}.

Bourgain's seminal work
\cite{bourgainPeriodicNonlinearSchrodinger1994} attracted interest to
probabilistic aspects of \eqref{eq:NLS-deterministic}, and revealed
new connections to other fields. Specifically, there is the desire to
study local and global dynamics of \eqref{eq:NLS-deterministic} with
initial data given by a probability distribution derived from
constructive quantum field theory, such as the \(\Phi^{4}_{3}\) model. This
naturally leads to study \eqref{eq:NLS-deterministic} with
low-regularity, random initial conditions. Significant results were
obtained for \eqref{eq:NLS-deterministic} on a torus $\TT^d$, by Deng,
Nahmod, Yue, who proved in
\cite{dengInvariantGibbsMeasure2021,dengRandomTensorsPropagation2022}
well-posedness and, for \(\TT^{3}\), the invariance of the associated
Gibbs measures. Analogous results for the cubic wave equation on
$\TT^3$ were obtained by the same authors together with Bringmann in
\cite{bringmannInvariantGibbsMeasures2022}, providing a wave equation
analogue to Hairer's (later with Matetski) celebrated result on
regularity structures
\cite{hairerTheoryRegularityStructures2014,hairerDiscretisationsRoughStochastic2018}
for parabolic equations. On \(\R^{d}\), the understanding of randomized
local and global dynamics of \eqref{eq:NLS-deterministic} is more
limited, with important partial progress obtained in
\cite{benyiProbabilisticCauchyTheory2015,pocovnicuLptheoryAlmostSure2018,burqRandomDataCauchy2008,spitzAlmostSureLocal2021,shenAlmostSureScattering2021,shenAlmostSureWellposedness2021,campsScatteringCubicSchrodinger2023,dodsonAlmostSureLocal2019}.
In one-dimension ($d = 1$) Burq, Thomann, and Tzvetkov in
\cite{burqLongTimeDynamics2013} proved the existence and invariance of
the Gibbs measure via probabilistic methods for a Nonlinear
Schrödinger equation with an additional coercive potential.

Our present work improves the local well-posedness theory of all the
above works on $\R^d$, $d \geq 3$. For a more in-depth overview of the
history and motivation, we direct the reader to \Cref{sec:history}.

In this manuscript and widely accepted in literature, a solution of
\eqref{eq:NLS-deterministic} is a function
\(u \in C^{0}\big([0,T);H^{S}_{x}(\R^{d})\big)\) that is a fixed point of the
Duhamel iteration map \(u\mapsto\mbb{I}_{f}(|u|^{2}u)\), where
\[\eqnum\label{eq:duhamel}
\mbb{I}_{f}(h)\eqd e^{ i t \Delta} f(x) \mp i\int_{0}^{t} e^{i (t - s) \Delta} h(s,x) \dd s\,,
\]
with the sign in front of the integral being opposite to the sign of
the right-hand side of \eqref{eq:NLS-deterministic}.  Next, let us
state our first main result.

\begin{theorem}[Local
well-posedness]\label{thm:random-local-wellposedness} Fix $f\in H_{x}^{S} (\R^{d})$ with $S_{\mrm{min}}<S<\xs_{c}$ (defined in
\eqref{eq:intro-Smin}). Then, with probability \(1\) the equation
\[\eqnum\label{eq:NLS-randomized}
\begin{dcases}
(i\partial_{t} + \Delta) u = \pm |u|^{2} u & \text{on } [0,T) \times \R^{d} ,
\\
u(0)=\rf.
\end{dcases}    
\]
admits a solution \(u\in C^{0}\big([0,T),H^{S}_{x}(\R^{d})\big)\), for a random
time \(T>0\) satisfying for appropriate \(C,c>0\) the bound
\[\eqnum\label{eq:random-time-probability-distribution}
\mbb{P}(T>t)\geq 1- C \exp\Big(-\frac{t^{-c}}{C \|f\|_{H^{S}_{x}(\R^{d})}^{2}}\Big)
\]
for any $t \in (0, 1)$.

The random initial datum \(\rf \in H^S_x(\R^d)\) is the unit-scale Wiener
randomization of the function $f \in H^S_x(\R^d)$ given by
\eqref{eq:initial-data-randomization} below (see
\Cref{sec:randomization} for details).
\end{theorem}

\Cref{thm:random-local-wellposedness} improves on known results for
probabilistic well-posedness on \(\R^{d}\) for all dimensions
$d \geq 3$.  In the case of physical relevance $d = 3$, we obtain
the optimal result except, possibly, for the endpoint
$S_{\mrm{min}} = 0$. Indeed, if \(f\) belongs to $H^{S}_{x}(\R^{d})$ with
$S < 0$, then a solution $u$ to \eqref{eq:NLS-deterministic} would
reside in $C^{0}\big([0;T);H^{S}_{x}(\R^{d})\big)$. Consequently,
$u(t)$ for any $t\in[0,T)$ would, a priori, be merely a distribution,
and the nonlinearity $|u(t)|^2u(t)$ would not be defined. A
probabilistic scaling argument (see \cite[Section
1.2]{dengRandomTensorsPropagation2022}) suggests that the natural
regularity endpoint for probabilistic local well-posedness results is
\(\xs_{\mrm{pr}}=-\sfrac{1}{2}\) in all dimensions. However, the above
discussion shows that for \(S<0\) the equation
\eqref{eq:NLS-deterministic} needs to be appropriately renormalized,
and, unlike on compact spatial domains, this process on \(\R^{d}\) is not
well understood.

If $S>S_{\mrm{min}}$ as in \Cref{thm:random-local-wellposedness}, we
also obtain global well-posedness and scattering if \(\rf\) is
sufficiently small, as detailed in \Cref{thm:random-scattering}
below.  We say a solution $u$ scatters in the space $H_{x}^{S}(\R^d)$
if, asymptotically, $u$ behaves like the solution to the linear
Schrödinger equation, that is, there exists a function
$\mf{w}:\R^d \to \R$ such that
\[\eqnum\label{eq:basic-scattering-defn}
\lim_{t \to \infty}\|u(t) - e^{it\Delta} \mf{w}\|_{H^{S}_x(\R^d)} = 0 \,.
\]
Global well-posedness and scattering for large initial data requires different techniques and assumptions that we do not address in this manuscript.

\begin{theorem}[Global well-posedness and scattering for small initial data] \label{thm:random-scattering} Under the assumptions of
\Cref{thm:random-local-wellposedness},  there exists a set
\(\Omega_{\mrm{glob}}\) with
\[\eqnum\label{eq:probabilistic-global-probability}
\mbb{P}(\Omega_{\mrm{glob}})>1 - C\exp\Bigg(- \frac{1}{c\|f\|_{H^{S+2\epsilon}_{x} (\R^{d})}^{2}}\Bigg)
\]
with $C, c > 0$ depending only on $S$ and $\xs_c$, such that
\(T=+\infty\) on $\Omega_{\mrm{glob}}$, where $T$ is the existence time of the
solution from \Cref{thm:random-local-wellposedness}. In addition, such
a solution \(u\in C^{0}\big([0,+\infty),H^{S}_{x}(\R^{d})\big)\), scatters in
\(H^{S}_{x}(\R^{d})\), that is, almost surely on $\Omega_{\mrm{glob}}$ there
exists \(\mf{w}\in H^{S}_{x}(\R^{d})\) such that
\eqref{eq:basic-scattering-defn} holds. Furthermore, scattering
\eqref{eq:basic-scattering-defn} holds in higher regularity
norms upon subtracting from the solution $u$ an explicit time dependent
function, as detailed in \Cref{cor:graded-scatter}. 
\end{theorem}

Our proof of Theorem \ref{thm:random-local-wellposedness} and
\ref{thm:random-scattering} relies on a multilinear expansion of the
solution \(u\in C^{0}\big([0,T);H^{S}_{x}(\R^{d})\big)\) to
\eqref{eq:NLS-randomized} of the form
\[\eqnum\label{eq:solution-decomposition}
u(t)=\sum_{k\leq M} \rz_{k}(t) + u^{\#}_{M}(t)\quad\text{with the remainder}\quad u^{\#}_{M}(t)\eqd
u(t)-\sum_{k\leq M} \rz_{k}(t),
\]
for a fixed, sufficiently large, \(M\in\N\). Each \(\rz_{k}\) can be
explicitly expressed as a $k$-linear expression of \(\rf\) by using
point-wise products and compositions with the linear
evolution. Although these formulas are explicit, they quickly become
complicated for \(k>1\). Specifically, \(\rz_{k}\) are defined inductively
by setting $\rz_0=0$, $\rz_{1}(t)\eqd e^{it\Delta}\rf$, and
\[\eqnum\label{eq:def:zk-randomized}
\begin{aligned}
\rz_{k+1}(t)\eqd \mp i\sum_{\mcl{\substack{ k_{1}+k_{2}+k_{3}=k+1\\0\leq k_{1},k_{2},k_{3}\leq k }} }
\hspace{2.5em}\int_{0}^{t} e^{i(t-s)\Delta} \rz_{k_{1}} \bar{\rz_{k_{2}}} \rz_{k_{3}}(s) \dd s.
\end{aligned}
\]
Substituting the expansion \eqref{eq:solution-decomposition} into
\eqref{eq:NLS-randomized} gives us an equation for the remainder
\(u^{\#}_{M}(t)\) and we \emph{require} that for an appropriate $M$, we
have \(u^{\#}_{M}\in C^{0}\big([0,T),H^{\xs}_{x}(\R^{d})\big)\) for some
\(\xs>\xs_{c}\). If $M = 1$, the solution has the form
$u = \rz_1 + u^{\#}_{1}$; in the literature this is called Bourgain’s
trick \cite{bourgainPeriodicNonlinearSchrodinger1994} or Da
Prato-Debussche trick
\cite{dapratoTwoDimensionalNavierStokes2002}. The decomposition
$u = \rz_1 + u^{\#}_{1}$ is natural, because the linear Schrödinger
evolution does not have a smoothing effect: there is no gain in
regularity for \(\rz_{1}\) compared to $\rf$. In fact, we show that the
remainder $u^{\#}_{1}$ term and the higher order terms \(\rz_{k}\),
\(k>1\), have better regularity than \(\rf\) and $\rz_1$, especially
when \(S\ll\xs_{c}\), that is, when \(\rf\) is rough.

Requiring that $u$ has the form \eqref{eq:solution-decomposition} with
\(u^{\#}_{M}\in C^{0}\big([0,T),H^{\xs}_{x}(\R^{d})\big)\) restricts the notion
of the solution. Indeed, \Cref{thm:random-local-wellposedness} does
not provide uniqueness of solutions in the class
\(C^{0}\big([0;T);H^{S}_{x}(\R^{d})\big)\), while uniqueness does hold for
solutions of the form \eqref{eq:solution-decomposition}.

The next theorem summarizes the regularity and scattering results for
the multilinear correction terms \(\rz_{k}\), \(k\in\N\).

\begin{theorem}[Regularity and scattering for $\rz_{k}$]\label{thm:random-multilinear-classical-bounds}
Fix \(f \in H^{S}_{x}(\R^{d})\) with \(S>0\), let
\(\rf \in H^S_x(\R^d)\) be the unit-scale Wiener randomization of the
function $f \in H^S_x(\R^d)$ given by
\eqref{eq:initial-data-randomization}, and for each $k \in \N$ define
\[ \eqnum\label{eq:def:multilinear-regularity}
\mu(k, S)\eqd\min \Big( kS, 2S + \frac{1}{2}, S+1\Big).
\]
The following properties hold almost-surely for the multilinear
expansion terms \(\rz_{k}\), defined inductively by the relation
\eqref{eq:def:zk-randomized}:
\begin{description}
\item[Regularity] For any \(\tilde{S}<\mu(k,S)\) it holds that \(\rz_{k}\in C^{0}(\R; H^{\tilde{S}}_{x}(\R^{d}))\) and there exists \(C=C(\tilde{S},k,S)\) such that
\[
\mbb{P}\Big(\|\rz_{k}\|_{C^{0}_{t}(\R;H^{\tilde{S}}_{x}(\R^{d}))}>\lambda\Big)\leq
C \exp\Bigg(- \frac{ \lambda^{\frac{2}{k}}}{C\|f\|_{H^{S}_{x} (\R^{d})}^{2}}\Bigg)\,.
\]
\item[Scattering]  For any \(\tilde{S}<\mu(k,S)\) the term \(\rz_{k}\) scatters in \(H^{\tilde{S}}_{x}(\R^{d})\), that is, there exists random \(\mf{w}_{k}\in\bigcap_{\tilde{S}<\mu(k,S)} H^{\tilde{S}}_{x}(\R^{d})\)
and a deterministic constant \(C=C(k,\tilde{S},S)\) such that
\[\eqnum\label{eq:random-multilinear-scattering}
\lim_{t\to+\infty}\big\|\rz_{k}(t)-e^{it\Delta}\mf{w}_{k}\big\|_{H^{\tilde{S}}_{x}(\R^{d})}=0,
\]
and 
\[
\mbb{P}\Big(\big\|\mf{w}_{k}\big\|_{H^{\tilde{S}}_{x}(\R^{d})}>\lambda\Big)\leq C \exp\Bigg(- \frac{ \lambda^{\frac{2}{k}}}{C\|f\|_{H^{S}_{x} (\R^{d})}^{2}}\Bigg).
\]
\end{description}
\end{theorem}

We believe that if \(\mu(k,S)=k S\), then
\Cref{thm:random-multilinear-classical-bounds} is optimal except for
the endpoint case. For example \(\mu(3,S)\leq 3S\), as has already observed in
\cite{shenAlmostSureWellposedness2021}: if
$\FT{f}(\xi) = |\ell|^{-S}\1_{B_{1/2}(\ell)}(\xi)$ for
\(\ell \in\Z^{d}\) large, then
$\sup_{t \in \R}\|\rz_3(t)\|_{H^{3S}_{x}(\R^d)}<+\infty $ only if
$f \in H^{S}_{x}(\R^d)$.  Concerning the endpoint, it is reasonable to
expect that \(\rz_{k}\in C^{0}\big(\R;H^{\mu(k,S)}_{x}(\R^{d})\big)\) but such
result would not improve our main results. On the other hand, proving
that \(\rz_{k}\in C^{0}\big(\R;H^{\tilde{S}}_{x}(\R^{d})\big)\) for some
\(\tilde{S}>\mu(k,S)\) when \(kS>\min\big(2S+\sfrac{1}{2},S+1\big)\) would lead to
improvements on \(S_{\mrm{min}}\) for \(d\geq4\) in our main theorems.

The minimal initial regularity \(S_{\mrm{min}}\) is chosen such that
\(\min\big(2S_{\mrm{min}}+\sfrac{1}{2},S_{\mrm{min}}+1\big) = \xs_c\), and
therefore since
\(\lim_{k\to+\infty}\mu(k,S)=\min\big(2S+\sfrac{1}{2},S+1\big)\),
\Cref{thm:random-multilinear-classical-bounds} implies that if
\(f\in H^{S}_{x}(\R^{d})\) with \(S>S_{\mrm{min}}\) then
\(\rz_{k}\in C^{0}\big(\R;H^{\xs}_{x}(\R^{d})\big)\) with
\(\xs>\xs_{c}\) for all sufficiently large \(k\in\N\).  The following result
refines \Cref{thm:random-local-wellposedness} by asserting that the
local in time solution \(u\) for \eqref{eq:NLS-randomized} of the form
\eqref{eq:solution-decomposition} with
\(u^{\#}_{M}\in C^{0}\big([0,T),H^{\xs}_{x}(\R^{d})\big)\) for some
\(\xs>\xs_{c}\) is unique, exists for a maximal time interval
\(T_{\mrm{max}}\), and scatters in \(H^{\xs_{c}}_{x}(\R^{d})\) if the initial
data is small enough.

\begin{theorem}[Regularity and scattering of $u^{\#}_{M}$]\label{thm:random-remainder-regularity}~
Fix \(S\) with \(S_{\mrm{min}}<S<\xs_{c}\) and let \(M\in\N\) be such that
\(\mu(M+1,S)>\xs_{c}\). 
Then for any $\xs < \mu(M+1,S)$ the following properties hold almost-surely. 
\begin{itemize}
\item There exists a random existence time \(T\in(0,+\infty]\) such that \eqref{eq:NLS-randomized} admits a solution $u\in C^{0}\big([0,T);H^{S}_{x}(\R^{d})\big)$ of the form \eqref{eq:solution-decomposition} with
\(u^{\#}_{M}\in C^{0}\big([0,T);H^{\xs}_{x}(\R^{d})\big)\).
\item It holds that \(u^{\#}_{M}\in C^{0}([0,T); H^{\xs}_{x}(\R^{d}))\).
\item Any two solutions with \(u^{\#}_{M}\in C^{0}([0,T); H^{\xs}_{x}(\R^{d}))\) coincide on \([0,T)\).
\item The time of existence satisfies \(T=+\infty\) on \(\Omega_{\mrm{glob}}\) as in \Cref{thm:random-scattering}. Moreover, 
\(u^{\#}_{M}\) scatters in \(H^{\xs}_{x}(\R^{d})\), that is, there exists
\(\mf{w}_{M}^{\#}\in\bigcap_{\xs<\mu(M+1,S)} H^{\xs}_{x}(\R^{d})\) such that
\[\eqnum\label{eq:random-remainder-scattering}
\lim_{t\to+\infty}\big\|u^{\#}_{M}-e^{it\Delta}\mf{w}_{M}^{\#}\big\|_{H^{\xs}_{x}(\R^{d})}=0\,,
\]
and
\[
\mbb{P}\Big(\big\|\mf{w}_{M}^{\#}\big\|_{H^{\xs}_{x}(\R^{d})}>\lambda\Big)\leq C \exp\Bigg(- \frac{ \lambda^{\frac{2}{k}}}{C\|f\|_{H^{S}_{x} (\R^{d})}^{2}}\Bigg).
\]
\end{itemize}
\end{theorem}

If we combine \Cref{thm:random-multilinear-classical-bounds} and
\Cref{thm:random-remainder-regularity}, then we obtain, information
about the behavior of $u$ as $t \to \infty$ that generalizes scattering from
\eqref{eq:basic-scattering-defn}. In particular, in the next corollary
we summarize the scattering results obtained above and we deduce scattering
in more regular spaces $H^{\mu(M+1,S)}_x(\R^d)$ by removing explicit
higher order correction terms \(\rz_{k}\), \(k\leq M\) from the solution. The
non-explicit remainder term \(u^{\#}_{M}\) scatters in deterministically
subcritical regularity \(H^{\xs}_{x}(\R^{d})\) with \(\xs>\xs_{c}\).

\begin{corollary}[Graded scattering]\label{cor:graded-scatter}
Under the assumptions of \Cref{thm:random-remainder-regularity}, on
$\Omega_{\mrm{glob}}$ there exists a unique global solution
\(u\in C^{0}([0,+\infty); H^{S}_{d}(\R^{d}))\) to \eqref{eq:NLS-randomized} of
the form
\[
u= \sum_{k\leq M} \rz_{k}(t) +u^{\#}_{M}
\]
and the terms \(\rz_{k}\) and \(u^{\#}_{M}\) scatter, as described by
\eqref{eq:random-multilinear-scattering} and \eqref{eq:random-remainder-scattering},
respectively. It follows that for any for any \(M'\in\N\) there exists
\(\mf{w}_{>M'}^{\#}\in\bigcap_{\tilde{S}<\mu(M'+1,S)} H^{\tilde{S}}_{x}(\R^{d})\) such that
\[
\begin{aligned}[t]
&
\lim_{t \to \infty} \Big\|\big(u-\sum_{k=1}^{M'}\rz_{k}(t)\big) - e^{it\Delta}\mf{w}_{>M'}^{\#}\Big\|_{H^{\tilde{S}}_x(\R^d)} = 0 \qquad \text{for all} \quad \tilde{S}<\mu(M'+1,S)\,.
\end{aligned}
\]
\end{corollary}

\Cref{thm:random-remainder-regularity} summarizes our results in
classical spaces \(C^{0}\big([0,T),H^{S}_{x}(\R^{d})\big)\). The crucial
novelty of our paper is the introduction of new family of directional
space-time norms, defined in \eqref{eq:def:directional-norm}, that
prove to be more efficient at capturing dispersion and regularity
properties of \(\rz_{k}\). Indeed, once the directional space-time
spaces \(X^{\xs}([0,T))\) and \(Y^{S}([0,T))\) are defined (see
\eqref{eq:def:X-norm}, \eqref{eq:def:Y-norm}), Theorems
\ref{thm:random-local-wellposedness} --
\ref{thm:random-remainder-regularity} can be reformulated more
precisely and deterministically.

Specifically, given a deterministic initial condition
\(f \in H^{S}_{x}(\R^{d})\) we define the multilinear data
\(\vec{z}_{M}=(z_{k})_{k\in\{1,\ldots,M\}}\) of order \(M\) associated to
\(f\) inductively, by setting $z_0 = 0$, $z_{1}(t)\eqd e^{it\Delta}f $ and
\[\eqnum\label{eq:def:zk-deterministic}
\begin{aligned}
 z_{k+1}(t)\eqd \mp i\sum_{\mcl{\substack{ k_{1}+k_{2}+k_{3}=k+1\\ 0 \leq k_{1},k_{2},k_{3}\leq k }} }
\hspace{2.2em}\int_{0}^{t} e^{i(t-s)\Delta}(z_{k_{1}} \bar{z_{k_{2}}} z_{k_{3}}) (s) \dd s \,, \quad k \geq 2\,.
\end{aligned}
\]
Notice that if \(f\in H^{S}_{x}(\R^{d})\) with \(S<\xs_{c}\), then
$z_{k}$ for \(k\geq2\) may not be well defined in general. However, we
assume that $z_{k}$ are well defined, as for example, in the case when
\(f\in H^{S}_{x}\) with \(S>\xs_{c}\). Then the multilinear data
\(\vec{z}_{M}\) can be viewed as an enrichment of the initial condition
\(f\), in the spirit of Lyon's theory of rough paths
\cite{lyonsDifferentialEquationsDriven1998,
  lyonsDifferentialEquationsDriven2007} or Hairer's regularity
structures, see \cite{frizCourseRoughPaths2020,
  hairerTheoryRegularityStructures2014}, and references therein.  The
central assertion of the following
\Cref{thm:intro-deterministic-controlled-well-posedness} then becomes
the continuity of the map $(f, \vec{z}_{M}) \mapsto u$ in appropriate
low-regularity spaces. We emphasize that $\vec{z}_{M}$ depends on $f$. In
addition, we provide of lower bound on the existence time of $u$ in
terms of suitable norms of $f$ and $\vec{z}_{M}$.

Below \Cref{thm:intro-deterministic-controlled-well-posedness}, we
introduce Wiener randomization of the initial data and we show that
almost surely the functions $z_k$ are well defined and bounded in the
required spaces.

The next two theorems contain precise formulation of the described
heuristics, and they are the main results of our paper from which
Theorems \ref{thm:random-local-wellposedness},
\ref{thm:random-scattering},
\ref{thm:random-multilinear-classical-bounds}, and
\ref{thm:random-remainder-regularity} are deduced. For the formulation
of the statements we used the notation $\lesssim$ defined in
\Cref{sec:notation} and the norms of the $Y^S$ spaces defined in
\eqref{eq:def:Y-norm}.

\begin{theorem}[Deterministic well-posedness for rough data]\label{thm:intro-deterministic-controlled-well-posedness}
Fix \(S>S_{\mrm{min}}\), \(M\in\N\) such that \(\mu(M+1,S)>\xs_c\), and then  
\(\xs_{c}<\xs<\min\big(\mu(M+1,S),\xs_{c}+\sfrac{1}{2}\big)\). Fix any
\(f \in H^{S+\epsilon}_{x}(\R^{d})\) and any \(T_{0}\in[0,+\infty]\) for which the
associated multilinear data
\(\vec{z}_{M}=(z_{k})_{k\in{1,\ldots,M}}\), defined a-priori by
\eqref{eq:def:zk-deterministic}, satisfy
\begin{equation}\label{eq:asoaz}
\|z_{k}\|_{Y^{\mu(k,S)+\epsilon}([0,T_{0}))}<\infty\qquad \text{ for all }  k\in\{1,\ldots,M\}\,.   
\end{equation}
For any $T>0$ and any function \(u\) with the domain
\([0,T)\times\R^{d}\) we define \(u^{\#}_{M}\) as
\[\eqnum\label{eq:solution-decomposition--deterministic}
u=z_{\leq M}+u^{\#}_{M},\quad z_{\leq M}\eqd\sum_{k\leq M}z_{k}.
\]
Then for any \(0<\epsilon\lesssim_{M,\xs,S}1\) and any
\(0<\epsilon_{0}\lesssim_{\epsilon,M,\xs,S}1\), appearing in the definition of the norms
\(X^{\xs}\) and \(Y^{S}\) (see \eqref{eq:def:X-norm} and
\eqref{eq:def:Y-norm}), the following assertions hold.

\begin{description}

\item[Local existence of solutions] There exists \(T\in(0,T_{0}]\) and a solution \(u\) to \eqref{eq:NLS-deterministic} with
\[
\|u^{\#}_{M}\|_{X^{\xs}([0,T))}+\|u^{\#}_{M}\|_{C^{0}\big([0,T);H^{\xs}_{x}(\R^{d})\big)}<\infty.
\]
The time of existence $T$ admits the lower bound
\[\eqnum\label{eq:time-of-existence}
T\geq\min\Big(T_{0},\frac{1}{C}\big\|1+\big(\max_{k\leq M}\|z_{k}\|_{Y^{\mu(k,S)+\epsilon}(\R)}\big)^{3}\big\|^{-\sfrac{3}{c}}\Big)\,,
\]
for some constants \(C=C(\epsilon,\epsilon_{0},M,\xs,S)>1\) and \(c=c(\epsilon,\epsilon_{0},M,\xs,S)>0\).

\item[Time-continuity of solutions] Any solutions \(u\) to \eqref{eq:NLS-deterministic}, defined on the time interval \([0,T)\), with \(\|u^{\#}_{M}\|_{X^{\xs}([0,T))}<\infty\) satisfies
\[\eqnum\label{eq:v-classical-space}
\|u^{\#}_{M}\|_{C^{0}\big([0,T);H^{\xs}_{x}(\R^{d})\big)}\lesssim\|u^{\#}_{M}\|_{X^{\xs}([0,T))}.
\]

\item[Uniqueness of solutions] Any two solutions \(u_1, u_2\) to \eqref{eq:NLS-deterministic}, defined on the time interval \([0,T)\), with \(\|u^{\#}_{j,M}\|_{C^{0}\big([0,T);H^{\xs}_{x}(\R^{d})\big)}<\infty\), $j = 1, 2$ coincide.

\item[Blow-up criterion]
Let \(T_{\mrm{max}}\) be the least upper bound of \(T\leq T_{0}\) for which a local
solution as above exists on \([0,T)\). Then either \(T_{\mrm{max}}=T_{0}\) or
\[\eqnum \label{eq:blowup}
\begin{aligned} 
& \|u^{\#}_{M}\|_{X^{\xs}([0,T_{\max}))}=\lim_{t\to T_{\mrm{max}}}\|u^{\#}_{M}(t)\|_{H^{\xs}_{x}(\R^{d}))}=+\infty.
\end{aligned}
\]

\item[Global existence of solutions]
There exists \(0<\delta_{0}\lesssim_{\epsilon_{0}, \epsilon, M,\xs, S} \!1\) such that if
\[\eqnum\label{eq:multilinear-datum-smallness}
\begin{aligned}
\|f\|_{H^{S+\epsilon}_{x}(\R^{d})}<\delta_{0} \quad \text{and} \quad
 \|z_{k}\|_{Y^{\mu(k,S)+\epsilon}(\R)}<\delta_{0}\qquad \text{for all } k\in\{1,\ldots,M\}\,,
\end{aligned}
\]
then \(T_{\mrm{max}}=+\infty\).

\item[Time-continuity and scattering of multilinear data]
It holds that 
\begin{equation}\label{eq:sobolev-isometry-of-linear-evolution}
\|z_1\|_{C^{0}\big([0,T_{0});H^{S}_{x}(\R^{d})\big)} \lesssim \|f\|_{H^S_x(\R^d)}    
\end{equation}
and for any $k \geq 2$, 
\[\eqnum\label{eq:time-continuity-sca}
\|z_{k}\|_{C^{0}\big([0,T_{0});H^{\mu(k,S)}_{x}(\R^{d})\big)}\lesssim\max_{k'\in\{1,\ldots,k-1\}}\|z_{k}\|_{Y^{\mu(k',S)+\epsilon}([0,T_{0}))}.
\]
If \(T_{0}=+\infty\) then \(z_{k}\), \(k\leq M\), scatter in
\(H^{\mu(k,S)}_{x}(\R^{d})\), that is, there exist
\(w_{k}\in H^{\mu(k,S)}_{x}(\R^{d})\) such that
\[\eqnum\label{eq:zk-scattering}
\lim_{t\to+\infty}\big\| z_{k}-e^{it\Delta}w_{k}\big\|_{H^{\mu(k,S)}_{x}(\R^{d})}=0.
\]

\item[Scattering of global solutions] Under the assumptions
\eqref{eq:multilinear-datum-smallness} for global existence, 
\(u^{\#}_{M}\) scatters in \(H^{\xs}_{x}(\R^{d})\), that is, there exists
\(w^{\#}_{M}\in H^{\xs}_{x}(\R^{d})\) such that
\[\eqnum\label{eq:remainder-scattering}
\lim_{t\to+\infty}\big\|u^{\#}_{M}-e^{it\Delta}w_{M}^{\#}\big\|_{H^{\xs}_{x}(\R^{d})}=0.
\]
\end{description}

Furthermore, the mapping from the enriched initial condition to the
solution is locally Lipschitz-continuous in the following sense. For
fixed parameters as above, define the distance on the initial data as
\[\eqnum\label{eq:initial-data-distance}
d(f,g)\eqd\|f-g\|_{H^{S+\epsilon}_{x}(\R^{d})}+\max_{k\in{1,\ldots,M}}\|\vec{z}_{M}(f)_{k}-\vec{z}_{M}(g)_{k}\|_{Y^{\mu(k,S)}([0,T_{0}))}\,,
\]
with \(\vec{z}_{M}(f)\) and \(\vec{z}_{M}(g)\) being the multilinear data
associated to \(f\) and \(g\) according to
\eqref{eq:def:zk-deterministic}.  Then for any \(R>0\) and corresponding time of existence 
\[
T=
\begin{dcases}
\min\Big(T_{0},\frac{1}{C}\big(1+R^{3}\big)^{-\sfrac{3}{c}}\Big) &\text{if } R\geq\delta_{0}\,,
\\
 T=T_{0} &\text{if }R<\delta_{0}\,,
\end{dcases}
\]
the map from \((f,\vec{z}_{M})\) with \(d(0,f)\leq R\) to the remainder term
\(u^{\#}_{M}\in X^{\xs}([0,T))\cap C^{0}([0,T);H^{\xs}_{x}(\R^{d}))\) 
is Lipschitz-continuous.
\end{theorem}

In \Cref{thm:intro-deterministic-controlled-well-posedness} the
initial datum \(f\in H^{S+\epsilon}_{x}(\R^{d})\) is deterministic, but there are
additional conditions on $\vec{z}_M$. If $f$ is randomized, then the
following theorem provides a refinement of
\Cref{thm:random-multilinear-classical-bounds}, showing that 
almost-surely estimates on
\(\|\rz_{k}\|_{Y^{\mu(k,S)+\epsilon}(\R)}\), \(k\in\N\), hold, as required by 
\Cref{thm:intro-deterministic-controlled-well-posedness}.

\begin{theorem}[Probabilistic estimates on multilinear expansions]\label{thm:random-multilinear-directional-bounds} 
Fix any \(M\in \N\), \(S>0\), any \(0<\epsilon\lesssim_{M,S} 1\), and any
\(0<\epsilon_{0}\lesssim_{\epsilon,M,S} 1\), appearing in the definition of the norm
\(Y^{S}\) (see \eqref{eq:def:Y-norm}). The following assertions
hold.

For any $f \in H^{S+\epsilon}_x(\R^d)$ let \(\rf\) be its  unit-scale Wiener
randomization (see \Cref{sec:randomization}) and let
\( \rz_{k}\), \(k\in\N\), be the multilinear expansions defined by
\eqref{eq:def:zk-randomized}.

Then there exists \(C=C(\epsilon_{0},\epsilon,M,S)>0\) such that
\[\eqnum\label{eq:rzn-Ys-probability}
\mbb{P}\Big(\|\rz_{k}\|_{Y^{\mu(k,S)}(\R)}>\lambda\Big)\leq
C \exp\Bigg(- \frac{ \lambda^{\frac{2}{k}}}{C\|f\|_{H^{S+\epsilon}_{x} (\R^{d})}^{2}}\Bigg)\,,
\]
for any \(k\leq M\) with $\mu(k, S)$ defined in \eqref{eq:def:multilinear-regularity}.
\end{theorem}

\subsection{The unit-scale Wiener randomization}\label{sec:randomization}

Next, we describe the unit-scale Wiener randomization \(\rf\) of the
initial condition \(f\in H^{S}_{x}(\R^{d})\). First, we fix a sequence
$(g_k)_{k\in \Z^d}$ of i.i.d. complex valued random variables on a
probability space $(\Omega , \mathcal{A}, \mathbb{P})$ and we assume that all their moments
are bounded, that is, \(\E\big[ |g_{k}|^{p}\big]<\infty\) for all
\(p\in\N\). For example, these assumptions are satisfied when
\(g_{k}\) are independent, standard (unit variance and \(0\) mean) complex
Gaussian random variables; the reader may assume this to be the case.

Let $\psi \in C_c^\infty \big(\R^d\big)$ be an even, non-negative cut-off function
supported in the unit ball of $\R^d$ centered at $0$ and such that, for
all $\xi \in \R^d$,
 \[\eqnum\label{eq:def:unit-scale-bump-condition}
 \sum_{k\in \Z^d} \psi (\xi -k)=1.
 \]
For $k\in \Z^d$, we define the operators \(\QP_{k}\) by setting
\[\eqnum\label{eq:def:unit-scale-projection}
(Q_k f)(x)= \mathcal{F}^{-1}\big(\psi (\xi -k) \FT{f} (\xi)\big)(x), \quad\text{for }x\in \R^d \,,
\]
where $\mathcal{F}(f)$ stands for the Fourier transform of
$f\in H^{S}_{x}(\R^{d}),\;S\geq 0$. We then set
\[\eqnum\label{eq:initial-data-randomization}
\rf\eqd f^{\omega}  \eqd \sum_{k\in \Z^{d}} g_{k} (\omega) \QP_{k} f.
\]
The random function \(\rf\) (random variable valued in functions on
\(\R^{d}\)) is then used as an initial condition in
\eqref{eq:NLS-randomized}.

The randomization \(f\mapsto \rf \) does not change (and in particular does
not improve) the differentiability properties of \(f\), in the sense
that
\[
\|f\|_{H^{S}_{x} (\R^{d})}^{2}\approx\E\|\rf\|_{H^{S}_{x} (\R^{d})}^{2},\qquad \text{for all } S\in \R,
\]
as it follows from a direct application of Plancherel's formula and
the independence.  Thus if
$f \in H^{S}_{x} (\R^{d}) \setminus H^{S + \epsilon}_{x} (\R^{d})$ for some
$\epsilon > 0$, then
$\rf \in H^{S}_{x} (\R^{d}) \setminus H^{S + \epsilon}_{x} (\R^{d})$ almost-surely (see
\cite{burqRandomDataCauchy2008a}).  However, the integrability
properties of $\rf $ are improved compared to $f$ (see
\Cref{prop:tree-bound-probabilistic}).

We interpret the infinite sum in \eqref{eq:initial-data-randomization}
as the limit in \(\|\cdot\|_{L^{p}_{\Omega}H^{S}_{x}}\) for any
\(p<\infty\) of finite partial sums. Equivalently, we define
\eqref{eq:initial-data-randomization} for functions
\(f\in L^{2}(\R^{d})\) with compact Fourier support and then extend the
definition by continuity.  It would be interesting to understand
whether alternative approximation procedures may yield different
distributions of multilinear terms \(\rz_{k}\), defined by
\eqref{eq:def:zk-randomized}.

\subsection{Motivation and history}\label{sec:history}
Next, we present an overview of the history and motivation for the
problem \eqref{eq:NLS-deterministic} and its probabilistic
counterpart \eqref{eq:NLS-randomized}, leading to our
result. The deterministic nonlinear Schrödinger equation
\eqref{eq:NLS-deterministic} possesses several remarkable
mathematical features. First, the equation
\eqref{eq:NLS-deterministic} is ``dispersive'', meaning that
solution's components supported on disjoint parts of the frequency
spectrum propagate with different velocities. Second, both the linear Schrödinger equation \((i \partial_{t}+\Delta)u=0\) and the nonlinear
Schrödinger equation
\eqref{eq:NLS-deterministic} are infinite
dimensional Hamiltonian systems. The linear equation has a Hamiltonian
$\mrm{H}_{0}[u]=\int_{\R^{d}}\frac{1}{2}|\nabla u(x)|^{2}\dd x$ while the nonlinear one is
associated with the Hamiltonian
$\mrm{H}[u]=\int_{\R^{d}}\frac{1}{2}|\nabla u(x)|^{2}\pm\frac{1}{4}|u(x)|^{4}\dd x$. For a
comprehensive discussion of the Hamiltonian structure of Schrödinger equations, we refer the reader to
\cite{mendelsonRigorousDerivationHamiltonian2020}.

Equation \eqref{eq:NLS-deterministic} is known to be locally
well-posed for \(S\geq\xs_c \), where we recall that \(\xs_{c}\eqd \frac{d-2}{2}\) is known as the
critical scaling regularity (see for instance
\cite{cazenaveSemilinearSchrodingerEquations2003,cazenaveCauchyProblemCritical1990,collianderGlobalWellposednessScattering2008,ryckmanGlobalWellposednessScattering2007,pausaderGlobalWellposednessEnergy2007,pausaderMasscriticalFourthorderSchrodinger2010}). Conversely,
ill-posedness for $S<\xs_c$ has been established by Christ, Colliander,
and Tao \cite{christIllposednessNonlinearSchrodinger2003}.  The local well-posedness
of \eqref{eq:NLS-deterministic} is obtained by finding a fixed point of
the Duhamel iteration map \eqref{eq:duhamel}, which in turn follows from the Banach fixed point theorem in appropriate norms that  capture the dispersive
nature of \eqref{eq:NLS-deterministic}. An example of such norms are
the Strichartz space-time norms (see \Cref{sec:linear-preliminaries}),
and the associated bounds are called Strichartz estimates (see
\Cref{lem:strichartz} below and originally in \cite{strichartzRestrictionsFourierTransforms1977}).

An interest in the probabilistic aspects of physical equations, which were
originally formulated deterministically, can be traced back at least
to the work of Poincaré \cite[Chapter 1 Section
IV]{poincareScienceMethod2007} who observed that regardless of the precise knowledge of the laws of physics, the
initial conditions can only be known approximately. This inherent
uncertainty profoundly influences the observed behavior of a physical
system, especially for systems that are naturally described as
statistical limits. The nonlinear Schrödinger equation, which is the
focus of this paper, is one such equation that can be derived (see
\cite{frohlichGibbsMeasuresNonlinear2017, mendelsonRigorousDerivationHamiltonian2020}, and references
therein) as the asymptotic limit of certain physical systems, such as
a large number of interacting Bosonic particles, or as the equation
for water waves in the small amplitude regime.  Therefore, it is
natural to investigate properties of solutions with initial conditions
governed by probability distributions. In fact, our result, stated in
\Cref{thm:random-local-wellposedness}, suggests that the carefully
constructed blow-up solutions of \cite{christIllposednessNonlinearSchrodinger2003} with initial data in
\(H^{S}_{x}(\R^{d})\), \(S<\xs_{c}\), are statistically irrelevant, and the
physical systems described by the nonlinear Schrödinger equation are
well-behaved even in low regularity regimes.

\subsubsection*{NLS on the torus}
The effect of randomization of the initial condition on improving
local (or global) well-posedness was pioneered by Bourgain in the
context of the cubic NLS on the torus \(\mbb{T}^{d}\). He established
almost-sure local well-posedness of \eqref{eq:NLS-deterministic} on
\(\mbb{T}^1\) (\cite{bourgainPeriodicNonlinearSchrodinger1994}) and on
\(\mbb{T}^{2}\) (\cite{bourgainInvariantMeasuresDefocusing1996}).  The
introduction of randomized initial conditions was related to
constructing a Gibbs measure: a probability measure on
\(H^{S}_{x}(\R^{d})\), for appropriate $S$, which is invariant under the
flow of the equation. Bourgain used almost-sure local well-posedness
and the invariance of the Gibbs measure to obtain almost-sure global
well-posedness of the cubic NLS. The Gibbs measure is supported on
spaces \(H^{S}_{x}(\TT^{d})\) with regularity
\(S<S_{\mrm{Gibbs}}\eqd1-\frac{d}{2}\), and therefore it is essential to
understand the equation with very rough, random initial conditions.
Recently, Deng, Nahmod and Yue in
\cite{dengInvariantGibbsMeasures2019} extended the result
\cite{bourgainInvariantMeasuresDefocusing1996} for \(\mbb{T}^{2}\) to
arbitrary odd power nonlinearities.  In
\cite{dengRandomTensorsPropagation2022}, the same authors introduced
the theory of random tensors, which allowed for the proof of the
almost-sure local well-posedness for \(d\geq3\) up to a natural
probabilistic regularity threshold \(\xs_{\mrm{pr}}=-\frac{1}{2}\) for the
cubic NLS. The problem of almost-sure global existence in low
regularity essentially remains open in \(d\geq3\) with particular interest
in the dimension \(d=3\), where the probabilistic scaling
\(\xs_{\mrm{pr}}=-\frac{1}{2}\) coincides with the regularity threshold
\(S_{\mrm{Gibbs}}\eqd1-\frac{d}{2}\) of the Gibbs measure.  Later the problem
was solved in \cite{bringmannInvariantGibbsMeasures2022} for the wave
equation with \(d=3\).

Bourgain's techniques from
\cite{bourgainPeriodicNonlinearSchrodinger1994,
  bourgainInvariantMeasuresDefocusing1996} were later used to study
other equations on compact domains, such as the cubic wave equation
\cite{burqInvariantMeasureThree2010}, or the Hartree NLS equation
\cite{bourgainInvariantMeasuresGrossPiatevskii1997} (see also
\cite{lebowitzStatisticalMechanicsNonlinear1988,syAlmostSureGlobal2021,dengInvariantGibbsMeasure2021,burqRandomDataCauchy2008}
for other results in this direction). For a more detailed survey of
the known results for the nonlinear Schrödinger equation on compact
domains, we direct the interested reader to
\cite{nahmodNonlinearSchrodingerEquation2015}. We refer to
\cite{kenigWorkJeanBourgain2020} for an overview of Bourgain's seminal
contributions to the study of dispersive PDEs.

\subsubsection*{NLS on Euclidean space}

If the domain is $\R^d$, as in the present manuscript, then the
dynamics of Schrödinger equation differs from the dynamics on compact
domains. In particular, on $\R^d$, different frequency components of a
solution interact weakly after long enough time as they disperse in
space, which lead to local smoothing estimates (see
\Cref{lem:dir-local-smoothing} below; also
\cite{constantinLocalSmoothingProperties1988a} and \cite[Theorem
4.3]{linaresIntroductionNonlinearDispersive2015}). On a torus the
solutions are spatially confined and local smoothing is not available.

Furthermore, on $\R^d$ there is no countable basis of eigenfunctions
for the Laplacian, and therefore a canonical randomization of initial
conditions in \(H^{S}_{x}(\R^{d})\) is less well-understood. Instead, it
is common to take the unit-scale Wiener randomization of a fixed
function \(f\in H^{S}_{x}(\R^{d})\) (see
\eqref{eq:initial-data-randomization} and the related discussion for
precise definitions), since it closely mimics the randomization on a
torus.  We remark that other randomizations have also been considered:
for instance see
\cite{burqRandomDataCauchy2008,spitzAlmostSureLocal2021,shenAlmostSureScattering2021}.

The question of almost-sure local well-posedness studied in the
current manuscript is of primary importance and has a rich
history. The first local well-posedness result on \(\R^{d}\) with
Wiener-randomized initial data (as in
\eqref{eq:initial-data-randomization}) was obtained by Bényi, Oh, and
Pocovnicu \cite{benyiProbabilisticCauchyTheory2015}, where the authors
established a result in the spirit of
\Cref{thm:random-local-wellposedness} for $d\geq 3$ with
$S>\frac{d-1}{d+1} \frac{d-2}{2}$. In \cite{shenAlmostSureWellposedness2021} Shen, Soffer, and
Wu improved the result for $d=3$ covering the range
$S\geq \frac{1}{6}$. The works \cite{benyiProbabilisticCauchyTheory2015,
  shenAlmostSureWellposedness2021} rely on a fixed point argument for the map
\eqref{eq:duhamel} with solutions controlled in Bourgain spaces
$X^{s,b}$ or their variation-norm variants $V^p$ and \(U^{p}\),
introduced by Koch, Tataru, and collaborators \cite{hadacWellposednessScatteringKPII2009,
  herrGlobalWellposednessEnergycritical2011, kochDispersiveEquationsNonlinear2014}. Dodson, Lührmann, and Mendelson \cite{dodsonAlmostSureLocal2019}
used direction spaces introduced by Ionescu and Kenig
\cite{ionescuLowregularitySchrodingerMaps2006,ionescuLowregularitySchrodingerMaps2007} to extended the local well-posedness result
in \cite{benyiProbabilisticCauchyTheory2015} for \(d=4\) to $S>\frac{1}{3}$.

The papers \cite{campsScatteringCubicSchrodinger2023, shenAlmostSureWellposedness2021, dodsonAlmostSureLocal2019}, cited
above, also obtained global well-posedness results in a smaller range
of regularities or under additional geometric assumptions on the
initial data (for example for radial initial conditions).  We do not address
the global existence of large solutions, leaving such extensions of
the current framework for future work.  We also mention that Pocovnicu
and Wang \cite{pocovnicuLptheoryAlmostSure2018} obtained almost-sure
local well-posedness of \eqref{eq:initial-data-randomization} for an
extended range of regularities but in \(L^{p}\)-based spaces with
\(p\neq2\). However, their approach does not allow for a refinement in the
spirit of \Cref{thm:random-multilinear-classical-bounds} and
\Cref{thm:random-remainder-regularity} as their remainder term has the
same regularity as the initial data.

The functional framework of \cite{dodsonAlmostSureLocal2019} served as
inspiration for \cite{casterasAlmostSureLocal2022}, where the present
authors extended the result of
\cite{benyiProbabilisticCauchyTheory2015} in arbitrary dimensions to
the case of the Laplacian in \eqref{eq:NLS-randomized} replaced by a
more general operator \(\mc{L}\). For example in
\cite{casterasAlmostSureLocal2022} one considers operators of the form
$\mc{L} = (-\Delta)^{\sigma/2} + \mc{L}^{\#}$ with $ \mc{L}^{\#}$ being of lower order
operator.  In \cite{casterasAlmostSureLocal2022} we obtained an
almost-sure local existence for solutions to \eqref{eq:NLS-randomized}
with the form
\[
u(t) = \rz_1 + u^{\#}_{1}
\]
with $u^{\#}_{1} \in C\big([0,T); H_x^{\frac{d-\sigma}{2}} (\R^d)\big)$, provided that
$S>\tilde{S}_{\min}(\sigma,d)$ for explicitly given $\tilde{S}_{\min}(\sigma,d)$.
Specifically, for the Laplacian $\mc{L} = \Delta$ 
\[
\tilde{S}_{\min}(2,d) = (d - 2)\times
\begin{cases}
 \frac{(d - 3)}{2(d - 1)} & d\geq 4 \,, \\
\frac{1}{6} & \ d=3 \,,
 \end{cases}
\]
which was the most general result for the second order NLS, except for
the endpoint case in $d = 3$ proved in \cite{shenAlmostSureWellposedness2021}. Our
theory of multilinear expansions given by
\Cref{thm:random-multilinear-classical-bounds} suggests (and proves
for \(\mc{L}=\Delta\)) that, as long as one looks for solutions of the form
$u = e^{it\mc{L}} \rf + u^{\#}_{1}$ with
$u^{\#}_{1} \in C\big([0,T); H_x^{\frac{d-\sigma}{2}} (\R^d)\big)$, the results in
\cite{casterasAlmostSureLocal2022} for $d \in\{3, 4\}$ are optimal, with exception of the endpoint regularities. Indeed, $u^{\#}_{1}$ cannot be smoother than \(\rz_{3}\) which, as
already observed in \cite{shenAlmostSureWellposedness2021}, does not belong to
\(C\big([0,T); H_x^{\frac{d-\sigma}{2}} (\R^d)\big)\) for
$f \in H^{\frac{d - \sigma}{6}}(\R^d)$ given by
$\FT{f}(\xi) = |k|^{-\frac{d - \sigma}{6}}\1_{B_{1}(k)} (\xi)$ with large \(k\in\Z^{d}\).

In this paper, we lower the required regularity on the initial data by
including higher order expansion to the solution. The idea of going
beyond the first order expansion given by the Da Prato - Debussche -
Bourgain trick has multiple precedents in literature.  Christ in
\cite{christChapterPowerSeries2009} gave the meaning to rough
solutions to the 1D cubic Schrödinger equation by expressing the
solution as an infinite multilinear series.  Developing a correct
functional framework to deal with expansions of arbitrary order is, for
example, the base for Lyon's theory of rough paths
\cite{lyonsDifferentialEquationsDriven2007}, for paracontrolled
distributions (see
\cite{gubinelliParacontrolledDistributionsSingular2015}), for regularity
structures (see \cite{hairerTheoryRegularityStructures2014}).  In our
context, multilinear expansions have been adopted by Bényi, Oh, and
Pocovnicu \cite{benyiHigherOrderExpansions2019} for equation
\eqref{eq:NLS-randomized} in dimension $d=3$. In their work, the
multilinear correction terms \(\rz_{k}\) are controlled in the Bourgain
$X^{s, b}$ spaces (introduced in
\cite{bourgainPeriodicNonlinearSchrodinger1994}). The methods in
\cite{benyiHigherOrderExpansions2019} provide regularity estimates on
\(\rz_{k}\) that are substantially worse compared to
\Cref{thm:random-multilinear-classical-bounds}, necessitating
expansions of arbitrarily high order to prove almost-sure local
well-posedness in \(H^{S}_{x}(\R^{3})\) with \(S>\frac{1}{6}\). In our
approach, we use directional norms to prove almost-sure local
well-posedness in \(H^{S}_{x}(\R^{3})\) with \(S>0\) if we use arbitrarily
high order expansions.  However our results for \(d\geq5\) and $d = 4$ use
respectively one and two terms in the expansion.  This suggests that
an improvement of regularity estimates for \(\rz_{k}\) with \(k\geq3\) in
\(d\geq4\) would lead to improvements in
\Cref{thm:random-local-wellposedness}. A different, and more
elaborated form, of high-order expansions was given in
\cite{dengRandomTensorsPropagation2022}, however such approach seems
to be limited to the torus.

\subsection{Outline of the paper}\label{sec:outline}

The solution to equation \eqref{eq:NLS-randomized}, postulated
by \Cref{thm:intro-deterministic-controlled-well-posedness}, is a fixed point of the
iteration map \(u\mapsto\mbb{I}_{\rf}(|u|^{3}u)\), defined in
\eqref{eq:duhamel}, which can be found by Picard iterations 
starting from an initial guess, for example
\(u_{0}(t,x)=0\). It is immediate to check that
\[
\begin{aligned}
&\mbb{I}_{\rf}(0)= z_{1}=e^{it\Delta} f,\qquad\mbb{I}_{f}\big(|z_{1}|^{2}z_{1}\big)=z_{1}+z_{3}.
\end{aligned}
\]
For the next step, one has
\[
\mbb{I}_{f}\big(|z_{1}+z_{3}|^{2}(z_{1}+z_{3})\big)\neq z_{1}+ z_{3} + z_{5} \,.
\]
However, for any
\(n\in\N\) the \(n\)-th iteration of the mapping
\(u\mapsto\mbb{I}_{\rf}(|u|^{3}u)\) starting with \(u=0\) is a linear combination
of \(k\)-linear, \(k\leq3^{n-1}\), operators \(R_{\tau}\) applied to the initial data
\(f\). Each operator \(R_{\tau}\) is characterized by a ternary tree \(\tau\), 
which are introduced in the first part of \Cref{sec:multilinear-probabilistic-estimates}. We remark that the order of multilinearity $k$ of \(R_{\tau}\) is the number of leaves of the ternary tree \(\tau\).
Adding together all terms of a fixed multilinear order \(k\), that is,  all 
\(R_{\tau}\) with \(\tau\) having $k$ leaves, we
obtain \(\rz_{k}\) described inductively in \eqref{eq:def:zk-randomized}. 

The main novelty of the present paper is the introduction of a family
of new directional space-time norms \eqref{eq:def:directional-norm} to
control the multilinear correction terms \(z_{k}\) and the remainder
term \(u^{\#}_{M}\).   A special case of our directional spaces were already used in 
\cite{casterasAlmostSureLocal2022} and there, they were inspired by \cite{dodsonAlmostSureLocal2019} that were, in turn, motivated by \cite{ionescuLowregularitySchrodingerMaps2006,ionescuLowregularitySchrodingerMaps2007}. We stress that our spaces are more general and to authors' knowledge were not yet used in the literature. Then, we define an appropriately weighted
combination of the directional space-time norms and classical
Strichartz norms on each Littlewood-Paley projection (see
\eqref{eq:LP-projection}) to define two families of norms \(X^{s}\) and \(Y^{s}\) indexed by
a regularity parameter \(s\in\R\) (see \eqref{eq:def:X-norm} and
\eqref{eq:def:Y-norm}). The spaces \(Y^{s}\) are well-suited to control
the explicit terms \(R_{\tau}[f,\ldots,f]\) and \(z_{k}\), while the norms
\(X^{s}\) are used to control the remainder term \(u^{\#}_{M}\) arising
from the decomposition \eqref{eq:solution-decomposition}. To relate our results to a more classical setting, we establish that the boundedness in \(X^{s}\) or
\(Y^{s}\) implies boundedness in \(C^{0}\big([0,T];H^{s}_{x}(\R^{d})\big)\) (see
\Cref{prop:linear-nonhomogeneous-continuity-bound}).

To prove \Cref{thm:intro-deterministic-controlled-well-posedness} we look for a solution
of the form \eqref{eq:solution-decomposition} for \(M\) such that
\(\mu(M+1,k)>\xs_{c}\). The function \(u\) is a solution, that is,  a fixed point
of \eqref{eq:duhamel} if and only if \(u^{\#}_{M}\) is a fixed point of
the map (see \eqref{eq:duhamel} for the definition of $\mbb{I}_{f}$)
\[\eqnum\label{eq:remainder-iteration-map}
\begin{aligned}[t]
\mc{J}_{z,M}(u^{\#}_{M})
=&\mbb{I}_{0}\big(\Phi_{z_{\leq  M}}[u^{\#}_{M}]+[z,z,z]_{>M}\big)
\\
& \eqd \mp i\int_{0}^{t}e^{i(t-s)\Delta}\big(\Phi_{z_{\leq
    M}}[u^{\#}_{M}]+[z,z,z]_{>M}\big)\dd s,
\end{aligned}
\]
where 
\[\eqnum\label{eq:def:z-sum}
z_{\leq M}\eqd\sum_{k\leq M} z_{k}\,,
\qquad
[z,z,z]_{> M}\eqd\sum_{\mcl{\qquad \substack{ k_{1}+k_{2}+k_{3}> M\\ k_{j}\leq M}}} z_{k_{1}}\bar{z_{k_{2}}} z_{k_{3}}\,,
\]
and
\[\eqnum\label{eq:remainder-cubic-nonlinearity}
\begin{aligned}
\Phi_{z_{\leq M}}[u^{\#}_{M}] & \eqd\big|z_{\leq M}+u^{\#}_{M}\big|^{2}\big(z_{\leq M}+u^{\#}_{M}\big)-\big|z_{\leq M}\big|^{2}z_{\leq M}.
\end{aligned}
\]
We show that \(u^{\#}_{M}\mapsto\mc{J}_{z,M}(u^{\#}_{M})\) is a contraction on
sufficiently small bounded sets of \(X^{\xs}\big([0,T)\big)\). The smallness
can be guaranteed if one chooses \(T>0\) to be small enough.

We prove this in two steps. First, assuming that
\(u^{\#}_{M}\in X^{\xs}\big([0,T)\big)\), we estimate the non-linearity
\[
h\eqd\big(\Phi_{z_{\leq M}}[u^{\#}_{M}]+[z,z,z]_{>M}\big)
\]
in the space \(X^{*,\xs+\epsilon}\big([0,T)\big)\), which is formally
dual to \(X^{\xs}\big([0,T)\big)\) (see \eqref{eq:def:X-norm-dual}). We show that 
the map 
\[
u^{\#}_{M}\mapsto\big(\Phi_{z_{\leq M}}[u^{\#}_{M}]+[z,z,z]_{>M}\big)
\]
from \(X^{\xs}([0,T))\) to \(X^{*,\xs+\epsilon}([0,T))\) has small Lipshitz
constant if \(T>0\) is chosen small enough. A minor modification shows
that shows that this is the case also if \(T=+\infty\) and the initial
condition and \(u^{\#}_{M}\) have small norms.  The required bound on
\[
\Big\|\big(\Phi_{z_{\leq M}}[u^{\#}_{M}]+[z,z,z]_{>M}\big)\Big\|_{X^{*,\xs+\epsilon}([0,T))}
\]
is obtained through bilinear Strichartz estimates in
\Cref{lem:bilinear-2-2-bound}.

Second we show that for any $h\in X^{*,\xs+\epsilon}([0,T))$ it holds that 
\[\eqnum\label{eq:intro-duhamel-bound}
\Big\|\mp i\int_{0}^{t}e^{i(t-s)\Delta}h(s,\cdot)\dd s\Big\|_{X^{\xs}\big([0,T)\big)}\lesssim \|h\|_{X^{*,\xs+\epsilon}\big([0,T)\big)} \,.
\]
The estimate \eqref{eq:intro-duhamel-bound} follows from duality and a
Christ-Kiselev type argument
\cite{christMaximalFunctionsAssociated2001} once we establish the
linear theory for the flow \(f\mapsto e^{it\Delta}f\) in the directional spaces
\(X^{s}\big([0,T)\big)\).

This implies the global well-posedness result contained in
\Cref{thm:random-scattering}. 

The role of bilinear Strichartz estimates in the study of the
nonlinear Schrödinger equation is widely recognized. Our directional
spaces crucially allow us to deduce
\(L^{2}_{t,x}\times L^{2}_{t,x}\mapsto L^{2}_{t,x}\) Strichartz estimates with a
gain of up to \(\sfrac{1}{2}\) derivative as long as one of the
functions is randomized, or more generally, bounded in the \(Y^{S'}\)
norm.  The proof of bilinear Strichartz estimates
\Cref{lem:bilinear-2-2-bound} for functions localized on fixed
Littlewood Paley annuli is surprisingly simple.  One of the major
technical steps of the proof consists of combining these localized
estimates by summing over all dyadic frequency bands.  We refer to
\Cref{prop:trilinear-estimate} for these details. We believe the study
of such directional spaces has independent interest for multilinear
Fourier extension estimates in Harmonic analysis. We recently learned
of a directional norms, similar to ours, were independently studied by
Beltran and Vega in \cite{beltranBilinearIdentitiesInvolving2020}.

As already mentioned above, the previous works
\cite{benyiProbabilisticCauchyTheory2015,benyiHigherOrderExpansions2019,shenAlmostSureWellposedness2021}
used variants of Bourgain spaces, however they seem to
provide suboptimal regularity estimates in higher dimensions.  In \cite{casterasAlmostSureLocal2022} we already estimated
\(\Phi_{\rz_{1}}[u^{\#}_{M}]+|z_{1}|^{2}z_{1}\) in a
dual space analogous to \(X^{*,s}([0,T))\) and, although it did not improve the
result of \cite{dodsonAlmostSureLocal2019}, it streamlined the proof and allowed for extensions to other dimensions and more general differential operators.
 In the current manuscript, we substantially extend the
spaces \(X^{s}\) and \(Y^{s}\) from \cite{casterasAlmostSureLocal2022} by introducing a larger
range of integrability exponents, allowing for
\(\mf{c}\notin\{2,\infty\}\) in \eqref{eq:def:directional-norm} and thus in
\Cref{lem:dir-maximal} and in \Cref{lem:dir-local-smoothing}. This is
essential for obtaining sharp regularity estimates on \(\rz_{k}\).

Our second central result and the second technical challenge is
proving the regularity estimates on the multilinear correction terms
\(\rz_{k}\) postulated in
\Cref{thm:random-multilinear-classical-bounds}.  We want to estimate
\(\rz_{k}\) in \(Y^{\mu(k,S)}\) norms, which are natural for solution of
\eqref{eq:NLS-deterministic} with the initial condition
\(f\in H^{S}(\R^{d})\) having compact Fourier support with unit radius.
Heuristically, the functions \(\rz_{k}\) cannot have more than
$\mu(k,S)$ degrees of regularity, since for initial data of the form
\(\FT{f}(\xi)=\1_{B_{1}(k)}(\xi)\) for some \(k\in\Z^{d}\), the Wiener randomization
does not provide any regularity improvement (see
\Cref{sec:randomization}).  However, the proof that stochastic initial
conditions behave like ones with unit-radius Fourier support requires
efficient bookkeeping provided by ternary trees, as described in
\Cref{sec:multilinear-probabilistic-estimates}. We then inductively
use deterministic bounds together with a multi-parameter stochastic
chaos estimate \Cref{lem:hypercontractivity} to prove the regularity
estimates of \Cref{thm:random-multilinear-directional-bounds}.

Finally, the scattering results of
\Cref{thm:intro-deterministic-controlled-well-posedness} follow
directly from the fact that a solution \(u\) is fixed point of
\eqref{eq:duhamel} and from a dual estimate to the linear flow
\(f\mapsto e^{it\Delta}f\). However, as detailed in \Cref{cor:graded-scatter}, our
results also provide a new perspective on scattering. Indeed, our
methods show that the classical scattering holds, that is, the
nonlinear dynamics approaches the linear ones, in $H^S_x(\R^d)$, a very
coarse norm. However,
\Cref{thm:intro-deterministic-controlled-well-posedness} and
\Cref{cor:graded-scatter} contain more precise information as they
provide an explicit expansion of the scattering data in terms of the
initial conditions, with a more regular remainder term.

\subsubsection*{Organization of the paper}

In \Cref{sec:linear-preliminaries}, we introduce directional space-time norms and review
the Littlewood-Paley decomposition and Bernstein's inequality. We also
recall classical Strichartz estimates for the flow of the linear
Schrödinger evolution \(f\mapsto e^{it\Delta}f\). We conclude  the section by proving two
crucial estimates for the linear Schrödinger evolution
\(f\mapsto e^{it\Delta}f\) in terms of directional spaces: the directional maximal
function estimate \Cref{lem:dir-maximal} and the directional smoothing
estimate \Cref{lem:dir-local-smoothing}. 

In \Cref{sec:nonhomogeneous-estimates} we define the spaces
$X^\xs$, $Y^S$, and \(X^{*,\xs}\) and prove estimates on the solution of
the linear non-homogeneous Schrödinger equation in these
spaces. Equivalently, such bounds provide estimates on the Duhamel
iteration map \eqref{eq:duhamel}. In
Propositions \ref{prop:linear-nonhomogeneous-continuity-bound} and \ref{prop:adjoint-bound}, we obtain dual
bounds for linear flow, and mapping properties of the Duhamel iteration map \eqref{eq:duhamel} into classical spaces
\(C^{0}\big([0,T),H^{S}_{x}(\R^{d})\big)\). These estimates are respectively  needed to prove scattering, and the uniqueness of solutions, relating them to more classical continuous-in-time \(L^{2}\)-in-space -based notions. 

\Cref{sec:multilinear-estimates} is devoted to the proof of bilinear
and multilinear estimates. We establish boundedness of pointwise
products of space-time functions in terms of the norms $X^\xs$, $Y^S$,
and \(X^{*,\xs}\).

In \Cref{sec:deteministic-fixed-point} contains the proof of the deterministic \Cref{thm:intro-deterministic-controlled-well-posedness} using a fixed point argument. 

In \Cref{sec:multilinear-probabilistic-estimates} is dedicated to
probabilistic estimates and to the proof of
\Cref{thm:random-multilinear-classical-bounds}. The section begins by
defining ternary trees, an efficient bookkeeping tool to track the
dependence of the multilinear correction terms \(\rz_{k}\) on the
initial data. In this section we also state the multilinear chaos
estimate that is the starting point of our probabilistic
analysis. While this result is known, especially in the case of Gaussian random variables, we provide a self-contained elementary proof in \Cref{sec:multilinear-chaos-appendix}.

Finally, in \Cref{sec:conclusion}, we combine the deterministic
results with the probabilistic estimates to prove
\Cref{thm:random-local-wellposedness} and
\Cref{thm:random-multilinear-classical-bounds}.

\subsection{Notation}\label{sec:notation}

\begin{itemize}
\item For two expressions $\LHS{}$ and $\RHS{}$ we write
$\LHS{} \lesssim\! \RHS{}$ if there exists a constant $C > 0$ depending only
on the parameters of the problem such that $\LHS{} \leq C \RHS{}$. For
example, $C$ can always depend on the dimension \(d\). If we want to
emphasize, that \(C\) may depend on a parameter \(\epsilon\) we write
$\LHS{} \lesssim_{\epsilon} \RHS{}$. Often \(\LHS\) and
\(\RHS{}\) are norms: then the implicit constant cannot depend on any
functions appearing in the bound. 

\item We write $\LHS{} \approx\! \RHS{}$, if $\LHS{} \lesssim\! \RHS{}$ and $\RHS{} \lesssim\! \LHS{}$.

\item If \(\RHS\) is non-negative and infinite, then \(\LHS{}\lesssim\!\RHS{}\) is true by default.

\item The Fourier transform of the function $f$ is denoted by 
\[
\mathcal{F} (f) (\xi) = \FT{f} (\xi) \eqd \int_{\R^{d}} f (x) e^{- 2 \pi i \xi x} \dd x. 
\]
The dimension $d$ is deduced from context. The Fourier inversion
formula holds:
\[
f (x) = \int_{\R^{d}} \FT{f} (\xi) e^{2 \pi i \xi x} \dd x .
\]
\item When dealing with a function of multiple variables
\(\R^{d_{1}}\times\R^{d_{2}}\ni (x,y)\mapsto f(x,y)\) we emphasize the variable in
which we are taking the Fourier transform by using the notation
\[
\Fourier_{y}(f)(x,\eta)=\int_{\R^{d_{2}}}f(x,y)e^{-2\pi i \eta y }\dd y.
\]
For functions \(f: I\times\R^{d}\to\C\) that depend on a ``time'' and ``space''
variables the Fourier transform is taken only in the space variable:
\[
\FT{f}(t,\xi)\eqd\int_{\R^{d}}f(t,x)e^{-2\pi i\xi x}\dd x.
\]

\item The symbol $O (\epsilon)$ stands for any function
$[0, 1) \rightarrow \R$ such that
\(|O (\epsilon) | \lesssim \epsilon \) for all $\epsilon \in (0, 1]$ with a constant uniform in
$\epsilon$.  The specific function intended by $O (\epsilon)$ can change from line
to line.

\item The ball of radius $r$ and center $x$ is denoted by $B_{r} (x)$;
if $x = 0$ we simply write $B_{r}$. The dimension of the ball is to be
understood from context.

\item For $p \geq 1$, $p'$ stands for the dual of $p$, that is, $\frac{1}{p'} + \frac{1}{p} = 1$.

\item The notation $\langle x \rangle \eqd (1 + |x|^{2})^{\frac{1}{2}}$ is the Japanese bracket.

\item For $\sigma \in\R$, we denote
$\langle \Delta \rangle^{\sigma / 2}$ the operator with the Fourier multiplier
$(1 + |2 \pi \xi |^{4})^{\sfrac{\sigma}{4}}$, that is,
$\mathcal{F} (\langle \Delta \rangle^{\sigma / 2} f) (\xi) = \langle 2 \pi |\xi|^{2} \rangle^{\sigma / 2} \FT{f} (\xi)$. Then,
$H^{\sigma}_{x} (\R^{d})$ denotes the Sobolev space endowed with the norm
\[\eqnum\label{eq:Sobolev-spaces}
\|u\|_{H^{\sigma} (\R^{d})}^{2} = \| \langle \Delta \rangle^{\sigma / 2} u (x)\|_{L^{2} (\R^{d})} .
\]

\item We denote by $\1_{A}$ the characteristic function of a set $A$,
that is, $\1_{A} (x) = 1$ if $x \in A$ and $\1_{A} (x) = 0$ otherwise. In
addition, if, for example, $x > y$ then we write $\1_{x > y}$ to
indicate the function that is equal to $1$ when $x > y$ and vanishes
otherwise. The variable of the function is to be deduced from context.

\item We denote by $\spt(f) \eqd \opclo\left( \{x \in \R^{d} : |f (x) | \neq 0\}
\right)$ the support of the function $f$, where $\opclo (A)$ denotes the closure of the set $A$.

\item We denote by $\diam (A) \eqd \sup_{x, y
  \in A} |x - y|$ the diameter of a set $A \subset \R^{d}$.

\item Given two sets \(A,B\subset\R^{d}\) the sum  and difference sets \(A\pm B\) are given by
\[
A \pm B := \{a + b: a \in A, b \in B\}\,.
\]

\end{itemize}

\section{The linear evolution}\label{sec:linear-preliminaries}

The solution of the linear Schrödinger equation 
\[\eqnum\label{eq:linear-schroedinger}
\begin{cases} 
(i\partial_{t} + \Delta) u = 0 & \text{ on } [0,T)\times \R^{d} , 
\\ 
u(0,x)=f(x)
\end{cases}
\]
with \(f\in L^{2}(\R^{d})\) is given by the linear Schrödinger evolution
group \(u(t,x)\eqd e^{it\Delta}f(x)\), which can be expressed using the Fourier
inversion formula as follows:
\[\eqnum\label{eq:schroedinger-propagation-group}
e^{it\Delta} f(x) \eqd \int_{\R^{d}}e^{2\pi i\xi x-4 \pi^{2} it |\xi|^{2} }\FT{f}(\xi) \dd\xi \,.
\]

In this section we introduce function spaces to bound solutions to
\eqref{eq:linear-schroedinger}. It is straightforward to see that the
map \(f\mapsto(t\mapsto e^{it\Delta}f)\) is continuous from
\(H^{\sigma}(\R^{d})\) to
\(C^{0}(\R;H^{\sigma}_{x}(\R^{d}))\cap C^{1}(\R;H^{\sigma-2}_{x}(\R^{d}))\) (see e.g. the
proof of \Cref{prop:linear-nonhomogeneous-continuity-bound}). However,
the space-time norms in this section are better suited for controlling
the dispersive nature of the linear evolution.

We begin with a review of Strichartz estimates for solutions to
\eqref{eq:linear-schroedinger}. To capture the regularity properties
of our solutions we use a Littlewood-Paley decomposition (dyadic
frequency annuli decomposition). We review the definitions and recall
basic results like Bernstein's inequality. Finally, we introduce
directional space-time norms, which are the main functional framework
for our paper and we prove the estimates for the linear evolution
\eqref{eq:linear-schroedinger} in terms of these norms.

A space-time norm is any norm that involves integrals in both the
temporal (\(t\)) and spatial (\(x\)) variables. We refer to the space-time norms
\[
\|u\|_{L_{t}^{p} L_{x}^{q} (I)}\eqd\Big(\int_{I}\|u(t,\cdot)\|_{L^{q}(\R^{d})}^{p}\dd t\Big)^{\frac{1}{p}}
\]
as Strichartz norms. We say that a pair of exponents 
$(p,q)$ is Strichartz-admissible if
\[\eqnum\label{eq:def:admissible}
\frac{2}{p} + \frac{d}{q} = \frac{d}{2}, \qquad q \in \Big[ 2, \frac{2 d}{d  - 2} \Big),\; p \in \big(2, \infty\big],
\]
and we call admissible those Strichartz norms that have such pairs of
integration exponents. For space-time norms, we omit \(\R^{d}\) from the
notation and emphasize only the time interval \(I \subseteq \R\).

\begin{lemma}[{\cite{strichartzRestrictionsFourierTransforms1977,ginibreSmoothingPropertiesRetarded1992,keelEndpointStrichartzEstimates1998}}] \label{lem:strichartz}
If $(p, q)$ be a Strichartz-admissible pair (in the sense of
\eqref{eq:def:admissible}), then for any \(f\in L^{2}(\R^{d})\) we have
that
\[\eqnum\label{eq:strichartz}
\| e^{ it \Delta} f \|_{L_{t}^{p} L_{x}^{q} (\R)} \lesssim \|f\|_{L^{2}(\R^{d})} \,,
\]
where the implicit constant does not depend on \(I\) or $f$.
\end{lemma}

Littlewood-Paley projections $\LP_{N}$, $N \in 2^{\N}$ isolate the
behavior of a function at frequencies of order \(N\). They provide us
with control over the differentiability properties of the solution to
\eqref{eq:linear-schroedinger} at different scales. Note that the
Littlewood-Paley projections $\LP_{N}$ below are different from the unit
scale projection $\QP_{n}$, introduced in
\eqref{eq:def:unit-scale-projection}.

To define \(\LP_{N}\) we fix a real valued cutoff function
$\phi \in C_{c}^{\infty} (B_{1})$ with $\phi (\xi) = 1$ on \(B_{1- 2^{- 100}}\) and set
\[\eqnum\label{eq:def:LP-bump}
\phi_{N} (\xi) \eqd \begin{dcases}
\phi (\xi) & \text{if } N = 1,
\\
\phi \Big( \frac{\xi}{N} \Big) - \phi \Big( \frac{\xi}{N / 2} \Big) & \text{if } N \in 2^{\N\setminus\{0\}},
\end{dcases}
\]
so that for each $\xi \in \R^d$ it holds that
\(\sum_{N \in 2^{\N}} \phi_{N} (\xi) = 1\).  For $N > 1$ we have
\[\eqnum\label{eq:LP-bump-geometry}
\spt \phi_{N} \subset \overline{B_{N}\setminus B_{(1-2^{-100})N/2}},\qquad \phi_{N}(\xi)=1 \text{ on } B_{(1-2^{-100})N}\setminus B_{N/2}.
\]
The Littlewood-Paley (approximate) projections are then obtained by setting
\[\eqnum\label{eq:LP-projection}
\FT{\LP_{N} f} (\xi) \eqd \phi_{N} (\xi) \FT{f} (\xi) 
\qquad \text{or, equivalently,} \qquad 
\LP_{N}f=f*\FT{\phi_{N}} \,.
\]
Consequently, \(f=\sum_{N\in2^{\N}}\LP_{N}f\) with convergence in
\(H^{\sigma}_{x}(\R^{d})\) for \(f\in H^{\sigma}_{x}(\R^{d})\), as can be seen on the
Fourier side using Lebesgue dominated convergence. In this manuscript
we also introduce modified Littlewood-Paley projections
\(\tilde{\LP}_{N}\), \(N\in2^{\N}\), defined as
\[\eqnum\label{eq:def:LP-modified}
\tilde{\LP}_{N}=\sum_{\mcl{\substack{N'\in2^{\N}\\
\sfrac{1}{8}< \sfrac{N'}{N} <8}}}\LP_{N'},
\]
so that, conveniently,
\(\tilde{\LP}_{N}\LP_{N}=\LP_{N}\tilde{\LP}_{N}=\LP_{N}\).  One can also express
\(\tilde{\LP}_{N}f\) as a convolution operator:
\[
\tilde{\LP}_{N}f=f*\FT{\chi_{N}} \quad \text{with} \quad \chi_N (\xi) \eqd
\sum_{\mcl{\substack{N'\in2^{\N}\\\sfrac{1}{8}< \sfrac{N'}{N} <8}}}\phi_{N'} (| \xi |).
\]
Note that \(\|\FT{\phi_{N}}\|_{L^{1}}\lesssim1\) and
\(\|\FT{\chi_{N}}\|_{L^{1}}\lesssim1\) with an implicit constant independent of
\(N\), so the Littlewood-Paley projections \(\LP_{N}\) and
\(\tilde{\LP}_{N}\) behave like averaging operators at the scale \(N^{-1}\).

\begin{remark}
The evolution group $e^{it \Delta}$ commutes with the projections
$\LP_{N}$, $N\in2^{\N}$, and $\QP_{n}$, $n\in \Z^{d}$ (and, more
generally, with Fourier multiplier operators). Thus,
\eqref{eq:strichartz} also holds with $f$ replaced by $\LP_{N} f$ or
by $\QP_{n} f$ on both sides of the inequality.
\end{remark}

Bernstein's inequality allows one to control the Lebesgue norms with
higher exponents by ones with lower exponents as long as the support
in frequency of the estimated function has bounded diameter.
\begin{lemma}[Bernstein's inequality]\label{lem:bernstein}
For any $r_{1},r_{2}\in[1,\infty]$ with \(r_{2}\geq r_{1}\) it holds that
\[
\|f\|_{L^{r_{2}}_{x} (\R^{d})}
\lesssim
\diam\big(\spt(\FT{f})\big)^{ \frac{d}{r_{1}} - \frac{d}{r_{2}} }
\|f\|_{L^{r_{1}}_{x} (\R^{d})} .
\]

In particular, for any $k \in \Z^{d}$ we obtain
\[\eqnum\label{eq:unit-scale-bernstein}
\big\|\QP_{k} f\big\|_{L^{r_{2}} _{x}(\R^{d})} \lesssim \big\|\QP_{k} f\big\|_{L^{r_{1}}_{x}(\R^{d})},
\]
since $\diam \big( \spt (\FT{\QP_{k} f}) \big) \lesssim 1$, and
\[\eqnum\label{eq:LP-bernstein}
\big\|\LP_{N} f\big\|_{L^{r_{2}}_{x} (\R^{d})} \leq N^{d( \frac{1}{r_{1}} - \frac{1}{r_{2}} )} \big\|\LP_{N} f\big\|_{L^{r_{1}}_{x} (\R^{d})},
\]
since $\diam\big( \spt( \FT{\LP_{N} f})\big) \lesssim N$. The implicit constants are allowed to depend only on the dimension, \(r_{1}\), and \(r_{2}\).
\end{lemma}

Finally, we introduce, directional space-time norms. These norms are designed to capture directional behavior of solutions to
\eqref{eq:linear-schroedinger}. For any $l \in \{1, \ldots, d\}$ we decompose
$x \in \R^{d}$ as
\[
x = x_{l} e_{l} + \sum_{i = 1, i \neq l}^{d} x_{i} e_{i} = : x_{l} e_{l} + x'_{l} .
\]
Given a time interval $I \subseteq \R$ and a fixed coordinate direction
$l \in \{1, \ldots, d\}$, the directional space-time norms
$ (\mf{a}, \mf{b}, \mf{c}) \in[1,\infty)^{3}$ are given by
\[\eqnum\label{eq:def:directional-norm}
\|h\|_{L_{e_{l}}^{\left( \mf{a}, \mf{b}, \mf{c} \right)} (I)}
=
\Big( \int_{\R} \Big( \int_{I} \Big( \int_{\R^{d - 1}} |h (t, x_{l} e_{l} + x_{l}') |^{\mf{c}}  \dd {x_{l}'} \Big)^{\frac{\mf{b}}{\mf{c}}} \dd t \Big)^{\frac{\mf{a}}{\mf{b}}}  \dd x_{l} \Big)^{\frac{1}{\mf{a}}} \,,
\]
where \(x_{l}^{'}\in\R^{d-1}\) is identified with a vector of
\(\R^{d}\) with \(0\) for its \(l\)-th component.  If
$\{ \mf{a}, \mf{b}, \mf{c} \} \ni \infty$ we use the standard modifications by the
essential supremum norm. We refer to directional norms
$L_{e_{l}}^{(\mf{a}, \infty, \mf{c} )} (I)$ as ``directional maximal norms''
because of the supremum in \(t\in I\).  When \(\mf{b} = 2\) we refer to the
norms as ``directional smoothing norms'' due to the negative power of
\(N\) appearing below in \eqref{eq:dir-local-smoothing}, which manifests
dampening of high oscillations, or smoothing.

The following lemmata are generalizations of \cite[Lemma 2.4 and Lemma
2.5]{casterasAlmostSureLocal2022} and they establish bounds on solutions to
\eqref{eq:linear-schroedinger} in terms of directional norms from
\eqref{eq:def:directional-norm}. \Cref{lem:dir-maximal} is
 a directional estimate maximal estimate, whereas
\Cref{lem:dir-local-smoothing} is a smoothing
result.

\begin{lemma}\label{lem:dir-maximal} Fix $N \in 2^{\N}$, $l \in \{ 1, \ldots, d \}$, and $\mf{c}\in(\mf{c}_{0},\infty]$ with $\mf{c}_{0}\eqd2\frac{d-1}{d-2}$. For any $f \in L^{2} (\R^{d})$ we have that
\[\eqnum\label{eq:dir-maximal}
\big\|e^{it \Delta} \LP_{N} f\big\|_{L_{e_{l}}^{( 2, \infty, \mf{c})} (\R)}
\lesssim_{\mf{c}}
N^{\frac{1}{2}+\big(\frac{d-2}{2}-\frac{d-1}{\mf{c}}\big)} \|\LP_{N} f\|_{L ^{2} (\R^{d})} \,.
\]
Furthermore, if \(\diam\big(\spt( \FT{f})\big)\leq 2R\), then for any \(\epsilon>0\) we obtain that
\[\eqnum\label{eq:dir-maximal-unit-scale}
\big\|e^{it \Delta} \LP_{N} f\big\|_{L_{e_{l}}^{( 2, \infty, \mf{c})} (\R)}
\lesssim_{R,\epsilon}
N^{\frac{1}{2}+\epsilon} \|\LP_{N} f\|_{L ^{2} (\R^{d})}.
\]
The implicit constants are allowed to depend on \(\mf{c}\), \(R>0\), and
\(\epsilon>0\), but are independent of \(N\) and the function \(f\).
\end{lemma}

\begin{lemma} \label{lem:dir-local-smoothing} Fix
$N \in 2^{\N}$, $l \in \{ 1, \ldots, d \}$, and $\mf{c}\in[2,\infty]$. Define the directional frequency
cone projections as
\[\eqnum\label{eq:microlocal-projections}
\begin{aligned}
& \FT{\UP_{e_{l}} f} (\xi) \eqd \1_{\mf{U}_{e_{l}}} (\xi) \FT{f} (\xi),\quad \text{where}
\\
& \mf{U}_{e_{l}} \eqd \Big\{ \xi \in \R^{d} \st \frac{| \xi_{l} |}{| \xi |} > \frac{1}{2 \sqrt{d}} \Big\}
\setminus \bigcup_{l' = 1}^{l - 1} \Big\{ \xi \in \R^{d} \st \frac{| \xi_{l'} |}{| \xi |} > \frac{1}{2
  \sqrt{d}} \Big\}\,.
\end{aligned} 
\]
Then, for any $f \in L^{2} (\R^{d})$ we have that 
\[\eqnum\label{eq:dir-local-smoothing}
\| e^{ it \Delta} \LP_{N} \UP_{e_{l}} f \|_{L_{e_{l}}^{(\infty, 2, \mf{c})} (\R)}
\lesssim
N^{-\frac{1}{2}+\big(\frac{d-1}{2}-\frac{d-1}{\mf{c}}\big)}
\|\LP_{N}\UP_{e_{l}} f\|_{L^{2} (\R^{d})} 
\]
and if \(\diam\big(\spt(\FT{f})\big)\leq 2R\), then
\[\eqnum\label{eq:dir-local-smoothing-unit-scale}
\big\|e^{it \Delta} \LP_{N} \UP_{e_{l}} f\big\|_{L_{e_{l}}^{( \infty,2, \mf{c})} (\R)}
\lesssim_{R}
N^{-\frac{1}{2}} \|\LP_{N}\UP_{e_{l}} f\|_{L ^{2} (\R^{d})}\,.
\]
The implicit constants are allowed to depend on \(\mf{c}\) and \(R>0\) but are independent of \(N\) and the function \(f\).
\end{lemma}

Clearly, by replacing the time interval $\R$ by any $I \subset \R$, the
space-time norms on the left-hand sides of the estimates in
\Cref{lem:strichartz}, \Cref{lem:dir-maximal}, and
\Cref{lem:dir-local-smoothing} decrease, so local-in-time versions of the
bounds above also trivially hold.

\begin{proof}[Proof of \Cref{lem:dir-maximal}]
Without loss of generality, we assume that $l = 1$ so that \(x=(x_{1},x')\in\R\times\R^{d-1}\). For any fixed $(t, x_1) \in \R \times \R$, $\xi' \in \R^{d - 1}$, and any \(g\in C^{\infty}_{c}(\R\times\R^{d})\) denote the Fourier transform in the $x'$ variable as
\[
\Fourier_{x'}g(t, x_{1},\xi')\eqd\int_{\R^{d-1}}g(t,x_{1},x')e^{-2\pi i\xi'x'}\dd x' \,.
\]
First, we claim that  \eqref{eq:dir-maximal-unit-scale} follows from \eqref{eq:dir-maximal} via a fiberwise application of Bernstein's inequality (\Cref{lem:bernstein}). Indeed, if $\diam\big(\spt(\hat{f})\big) \leq 2R$, then there exists $\tilde{\xi} = (\tilde{\xi}_1, \tilde{\xi}') \in \R^d$ such that $\hat{f}(\xi) = 0$ whenever $|\tilde{\xi} - \xi| > 2R$ and in particular, $\hat{f}(\xi_1, \xi') = 0$ whenever $|\tilde{\xi}' - \xi'| > 2R$. Since the evolution group \(e^{it\Delta}\) conserves the Fourier support, then 
\[
\Fourier_{x'} (e^{it\Delta} f)(x_1, \xi')
=
\int_{\R} e^{2\pi i \xi_{1}x_{1} - 4 \pi^{2}i t (\xi_1^{2} +  |\xi'|^{2})}\hat{f}(\xi_1, \xi') \dd \xi_1 = 0
\]
whenever $|\tilde{\xi}' - \xi'| > 2R$, and therefore 
\[\eqnum\label{eq:dmoff}
\diam\Big(\big\{\xi'\in\R^{d-1}\st \Fourier_{x'}(e^{it\Delta}\LP_{N}f)(x_{1},\xi')\neq 0\big\}\Big)\leq 4R.
\]
Then, by Bernstein's inequality (\Cref{lem:bernstein}), for given $\epsilon >0$ and \(\mf{c}\) it holds that 
\[
\|e^{it\Delta}\LP_{N}f(x_{1}, \cdot)\|_{L^{\mf{c}}(\R^{d-1})}
\lesssim_{R}
\|e^{it\Delta}\LP_{N} f(x_{1},\cdot)\|_{L^{\tilde{\mf{c}}}(\R^{d-1})} \,,
\]
where $\tilde{\mf{c}}\in(2,\mf{c})$ is chosen to satisfy \(\frac{d-2}{2}-\frac{d-1}{\tilde{\mf{c}}}<\epsilon\).
The claimed bound  \eqref{eq:dir-maximal-unit-scale} then follows from \eqref{eq:dir-maximal}. 

Next, we focus on the proof of \eqref{eq:dir-maximal}, which is based on a \(TT^{*}\) argument and classical techniques for oscillatory integrals. 
Let us introduce the operator \(T_{N}\) given by 
\[\eqnum\label{eq:def:T}
(T_{N} f) (t,\cdot) \eqd e^{it\Delta} \tilde{P}_{N}f \,,
\]
where $\tilde{P}_{N}$ is as in \eqref{eq:def:LP-modified}. 
Then, \(T_{N}\LP_{N}f=e^{it\Delta}\LP_{N}f\) and \eqref{eq:dir-maximal} follows once we prove the estimate
\[\eqnum\label{eq:stni}
\|T_{N} f\|_{L_{e_{1}}^{( 2, \infty, \mf{c} )} (\R)}
\lesssim
N^{\frac{1}{2} +  (\frac{d-2}{2}-\frac{d-1}{\mf{c}})} \|f\|_{L^{2}(\R^{d})}
\]
and replace $f$ by $P_N f$.  By duality, \eqref{eq:stni} is equivalent to
\[\eqnum\label{eq:dgfp}
\|T^{*}_{N} g\|_{L^{2} (\R^{d})} 
\lesssim N^{\frac{1}{2} + (\frac{d-2}{2}-\frac{d-1}{\mf{c}})}  \|g\|_{L_{e_l}^{( 2, 1, \mf{c}' )}  (\R)} \,,
\]
where $T^{*}_{N}$ is the dual operator to $T_{N}$ with respect to the
$L^{2}(\R \times\R^{d})$ inner product.  We claim that \eqref{eq:dgfp} follows by showing the \(TT^{*}\) bound
\[\eqnum\label{eq:TTstar-goal}
\|T_{N} T_{N}^{*} g\|_{L_{e_l}^{( 2, \infty, \mf{c} )} (\R)}
\lesssim
N^{1 +  2(\frac{d-2}{2}-\frac{d-1}{\mf{c}})} \|g\|_{L_{e_l}^{( 2, 1, \mf{c}' )} (\R)} .
\]
Indeed, if \eqref{eq:TTstar-goal} holds, then
\[
\|T_{N}^{*} g\|_{L^{2}_x (\R^{d})}^{2}
\begin{aligned}[t]
& = |\langle T_{N}^{*} g, T_{N}^{*} g \rangle| =  |\langle T_{N}T_{N}^{*} g,  g \rangle|
\\
& \leq \|T_{N}T_{N}^{*} g\|_{L_{e_l}^{ ( 2, \infty, \mf{c} )}  (\R \times)}  \|g\|_{L_{e_l}^{( 2, 1, \mf{c}' )}  (\R)}  
\\
& \lesssim  N^{1 + 2(\frac{d-2}{2}-\frac{d-1}{\mf{c}})} \|g\|^{2}_{L_{e_l}^{\left( 2, 1, \mf{c}' \right)}  (\R)} 
\end{aligned}
\]
and \eqref{eq:dgfp} follows.

To prove \eqref{eq:TTstar-goal} we may assume that \(g\in C^{\infty}_{c}(\R\times\R^{d})\), as the general result follows by standard approximation arguments. Using the Fourier inversion representation of the Schrödinger evolution group \eqref{eq:schroedinger-propagation-group} and the defining identity
\[
\int_{\R^{d}} (T_{N}^{*}g)(x) \bar{f(x)}\dd x=\int_{\R\times \R^{d}} g(x) \bar{(T_{N}f)(t,x)}\dd t \dd x\,,
\]
we can express \(T_{N}\) and \(T_{N}^{*}\) as
\[\eqnum\label{eq:tnft}
\begin{aligned}
& (T_{N}f)(t,x)=
\int_{\R^{d}} e^{2 \pi i \xi x - 4 \pi^{2} i t  | \xi  |^{2}} \chi_{N} (\xi) \FT{f} (\xi) d \xi.
\\
& (T_{N}^{*}g)(x) \eqd \int_{\R\times\R^{d}} e^{2\pi i \xi x + 4 \pi^{2}it|\xi|^{2}}\FT{g}(t,\xi) \chi_{N} (\xi) \dd t \dd \xi.
\end{aligned}
\]
Combining the two expressions in \eqref{eq:tnft} followed by a further use of the Fourier inversion formula shows that
\[\eqnum\label{eq:TTstar-kernel}
\begin{aligned}
T_{N}T_{N}^{*} g (t, x) = &K_N * g(t,x) = \;\int_{\mrl{\nquad \R  \times \R^{d}}} K_{N}
(t - s, x- y) g (s, y) \dd s \dd y ,
\\ 
& K_{N} (t, x) \eqd  \int_{\R^{d}} e^{2 \pi i \xi x - 4 \pi^{2} i t  | \xi |^{2}} \chi_{N}^{2} (\xi) \dd \xi .
\end{aligned}
\]

For fixed $(t, x_{1}) \in \R \times \R$ and $p \in [1, \infty]$, let $\| K_{N} (t, x_{1},\cdot) \|_{\mathcal{L}(L^{p'}, L^{p})}$ denote the operator norm of the map
\[
h \mapsto \int_{\mcl{\R^{d-1}}} K_{N} (t, x_{1}, x' - y') h (y') \dd y' 
\]
from \(L^{p'}(\R^{d-1})\) to \(L^{p}(\R^{d-1})\). To prove \eqref{eq:TTstar-goal}, it suffices to prove the bounds
\begin{align*}
\eqnum\label{eq:Kn-L1Linfty-bound} 
& \| K_{N} (t, x_{1} , \cdot) \|_{\mathcal{L}(L^{1}, L^{\infty})}\lesssim N^{d}(1+N |x_{1}|)^{-\frac{d}{2}}\,,
\\
\eqnum\label{eq:Kn-L2L2-bound}
& \|K_{N} (t, x_{1} , \cdot) \|_{\mathcal{L}(L^{2}, L^{2} )}\lesssim N(1+N|x_{1}|)^{-\frac{1}{2}}\,.
\end{align*}
with implicit constants independent of any fixed $( t,x_1) \in \R\times\R$. 
Indeed, the Riesz-Thorin interpolation theorem applied to the bounds
\eqref{eq:Kn-L1Linfty-bound} and \eqref{eq:Kn-L2L2-bound} yields, for
any fixed $(t, x_1) \in \R \times \R$ and $\mf{c}\geq2$, that
\[
\|K_{N} (t, x_{1} , \cdot) \|_{\mathcal{L}(L^{\mf{c}'},  L^{\mf{c}})}
\lesssim
N^{\frac{2 + \mf{c}d-2d}{\mf{c}}}  (1+N|x_{1}| )^{-\frac{1}{2} -(d-1)(\frac{1}{2}-\frac{1}{\mf{c}})}.
\]
Therefore, if \(\mf{c}>\mf{c}_{0}\eqd2\frac{d-1}{d-2}\), then
$-\frac{1}{2} -(d-1)(\frac{1}{2}-\frac{1}{\mf{c}}) < -1$, and consequently
\[\eqnum\label{eq:K-goal}
\Big\|\big\| \| K_{N}(t,x_{1},\cdot)\|_{\mathcal{L}(L^{\mf{c'}}, L^{\mf{c})}} \big\|_{L^{\infty}_{t}(\R)}\Big\|_{L^{1}_{x_{1}} ( \R )}  \leq N^{1+2\big(\frac{d-2}{2}-\frac{d-1}{\mf{c}}\big)}.
\]
Hence, Young's convolution inequality in \(t\) implies for any
$x_1 \in \R$ that
\[
\begin{aligned}
\hspace{3em}&\hspace{-3em} \big\| \| K_{N} * g (t, x_{1}, \cdot) \|_{L^{\mf{c}} ( \R^{d-1} )} \big\|_{L^{\infty}_{t}
  (\R)}
\\ 
& \leq \Big\| \int_{\mcl{\R \times \R}} \| K_{N} (t - s, x_{1} - y_{1} , \cdot)
\|_{\mathcal{L}( L^{\mf{c}'}, L^{\mf{c}})} \| g (s, y_{1}, \cdot) \|_{L^{\mf{c}'} ( \R^{d-
    1})} \dd s \dd y_{1} \Big\|_{L^{\infty}_{t} ( \R )}
\\
& \leq \int_{\R} \big\| \| K_{N} (t,
x_{1} - y_{1} , \cdot) \|_{\mathcal{L}(L^{\mf{c}'}, L^{\mf{c}})} \big\|_{L^{\infty}_{t}(\R)}\; \big\|
\| g (t, y_{1}, \cdot) \|_{L^{\mf{c}'} ( \R^{d- 1})}\big\|_{L^{1}_{t}(\R)}\dd y_{1}.
\end{aligned}
\]
Another application of Young's convolution inequality, this time in \(x_{1}\), gives us 
\[
\begin{aligned}
\hspace{10em}&\hspace{-10em}\Big\| \big\| \| K_{N} * g(t,x_{1},\cdot) \|_{L^{\mf{c}} ( \R^{d-1} )} \big\|_{L^{\infty}_{t} (\R)}\Big\|_{L^{2}_{x_{1}} ( \R )}
\\ & \leq \Big\|\big\| \| K_{N}(t,x_{1},\cdot)  \|_{\mathcal{L}(L^{\mf{c}'},  L^{\mf{c}} )} \big\|_{L^{\infty}_{t}(\R)}\Big\|_{L^{1}_{x_{1}} ( \R )} \|g\|_{L_{e_{1}}^{( 2, 1, \mf{c}')} (\R)},
\end{aligned}
\]
and \eqref{eq:TTstar-goal} follows from \eqref{eq:K-goal}.

To show \eqref{eq:Kn-L1Linfty-bound}, we use the rescaling 
$\xi \to N \xi$ and  $\chi_1(\xi) = \chi_N(N \xi)$ in \eqref{eq:TTstar-kernel} to obtain
\[
K_{N} (t, x)
=
N^{d} \int_{\R^{d}} e^{2 \pi i N \xi x - 4\pi^2 iN^{2} t |\xi|^{2}} \chi_1^{2}(\xi) \dd \xi.
\]
The function \(\chi_1\) is a bump function independent of \(N\) for any
\(N\geq4\).  By interchanging the absolute value and integral, we obtain
a trivial bound \(\big|K_{N} (t, x)\big|\lesssim N^{d}\), while using Stein's Lemma
\cite[Section 5.13, p363]{steinHarmonicAnalysisRealvariable1993} we obtain for any
\((t,x_{1})\in \R\times\R\) that
\[
\big|K_{N} (t, x)\big|\lesssim N^{d}\big(N |x|+N^{2}|t|\big)^{-\frac{d}{2}}.
\]
A combination with the trivial bound yields
\[
\big|K_{N} (t, x_{1}, x')\big|
\begin{aligned}[t]
& \lesssim N^{d}(1+N |x_{1}|+N |x'|+N^{2}|t|)^{-\frac{d}{2}}
\\
& \lesssim N^{d}(1+N |x_{1}|)^{-\frac{d}{2}}
\end{aligned}
\]
and \eqref{eq:Kn-L1Linfty-bound} follows from Young's convolution
inequality in $x'$.

To show \eqref{eq:Kn-L2L2-bound}, by Plancherel's identity in the $x'$
variable, we have for any fixed $(t, x_1) \in \R \times \R$ that
\[
\| K_N (t, x_1, \cdot) \|_{\mathcal{L}(L^{2}, L^{2})} 
\begin{aligned}[t]
& = \sup_{\|f\|_{L^{2}(\R^{d-1})} \leq 1 } \|K_N (t, x_1, \cdot) * f\|_{L^{2}(\R^{d-1})} 
\\
& = \sup_{\|\hat{f}\|_{L^{2}(\R^{d-1})} \leq 1 } \|\Fourier_{x'} K_N (t, x_1, \cdot) \FT{f}\|_{L^{2}(\R^{d-1})}
\\
& \leq \|\Fourier_{x'} K_N (t, x_1, \cdot)\|_{L^\infty(\R^{d-1})} \,.
\end{aligned}
\]
Using Fourier inversion and the change of variables
\(\xi_{1} \to N \xi_1\) we obtain for any fixed $(t, x_1) \in \R \times \R$ and $\eta' \in \R^{d-1}$ that
\[
\Fourier_{x'} {K}_N (t, x_1, \eta')  
\begin{aligned}[t]
& = 
\int_{\mcl{\R^{d - 1}}} K_{N} (t, x_1, x') e^{- 2 \pi i \eta' x'} \dd x'
\\
& = \int_{\mcl{\R^{d - 1}}} \int_{\;\R^{d}} e^{2 \pi i \xi (x_{1}, x') -  4\pi^2 i t |\xi|^{2}}\chi^{2}_{N} (\xi)  e^{- 2 \pi i \eta' x'} \dd \xi \dd x'
\\
& = \int_{\R} e^{2 \pi i \xi_{1}x_{1} - 4\pi^2 it  \xi_{1}^{2}-4\pi^2 it  |\eta'|^{2}}\chi^{2}_{N} (\xi_{1},\eta') \dd \xi_{1}
\\
& = N\int_{\R} e^{2 \pi i N \xi_{1}x_{1} - 4\pi^2 it  N^{2}  \xi_{1}^{2}-4\pi^2 it  |\eta'|^{2}}\chi_1^{2}\big(\xi_{1},\eta'/N\big) \dd \xi_{1}.
\end{aligned}
\]
Since \(\xi_{1}\mapsto \chi_1^{2}\big(\xi_{1},\eta'/N \big) \) is \(C^{\infty}_{c}(\R)\) function
bounded by 1, 
 Stein's lemma \cite[Section
5.13, p363]{steinHarmonicAnalysisRealvariable1993} applied in the \(\xi_{1}\) variable implies 
\[
\Big| \int_{\R} e^{2 \pi i N \xi_1 x_1 - it 4\pi^2 N^{2} \xi_1^{2}} \chi_1^{2} (\xi_1, \eta') \dd \xi_1 \Big|
\lesssim
(1 + | N^{2} t | + | N x_1 |)^{- \frac{1}{2}}
\leq 
(1 + | N x_1 |)^{- \frac{1}{2}}
\]
with an implicit constant independent of \(\eta'\). Hence, \eqref{eq:Kn-L2L2-bound} follows, completing the proof of
\eqref{eq:dir-maximal}.
\end{proof}

\begin{proof}[Proof of \Cref{lem:dir-local-smoothing}]
Without loss of generality, we assume that $l = 1$. Bounds \eqref{eq:dir-local-smoothing} and \eqref{eq:dir-local-smoothing-unit-scale} for \(\mf{c} \geq 2\) follow from \eqref{eq:dir-local-smoothing} for \(\mf{c}=2\).
Indeed, for any fixed \((t,x_{1}) \in \R \times \R\) from Bernstein's inequality (\Cref{lem:bernstein} and see the proof of Lemma \ref{lem:dir-maximal} for details) in the \(x'\in\R^{d-1}\) variable we obtain
\[
\|\LP_{N}\UP_{e_{1}}f(x_{1},\cdot)\|_{L^{\mf{c}}(\R^{d-1})}
\lesssim N^{(d-1)(\frac{1}{2}-\frac{1}{\mf{c}})} \|\LP_{N} \UP_{e_{1}}f(x_{1},\cdot)\|_{L^{2}(\R^{d-1})}
\]
and if \(\diam(\spt( \FT{f}))\leq 2R\) we obtain
\[
\|\LP_{N}\UP_{e_{1}}f(x_{1},\cdot)\|_{L^{\mf{c}}(\R^{d-1})}
\lesssim
 R^{(d-1)(\frac{1}{2}-\frac{1}{\mf{c}})} \|\LP_{N} \UP_{e_{1}}f(x_{1},\cdot)\|_{L^{2}(\R^{d-1})} \,.
\]

We prove \eqref{eq:dir-local-smoothing} for \(\mf{c}=2\) by showing that  
\[
\sup_{x_1 \in \R} \| T_{N} \UP_{e_{1}}f (\cdot, x_1, x') \|_{L^{2}_{t} L^{2}_{x'} (\R \times \R^{d - 1} )} \lesssim N^{-\frac{1}{2}} \|\UP_{e_{1}}f\|_{L^{2}_{x}(\R^{d})},
\]
where \(T_{N}\) is defined in \eqref{eq:def:T}. Plancherel's identity in $x'$ and \eqref{eq:tnft} imply for any $x_1 \in \R$ that
\[
\begin{aligned}
\hspace{3em}&\hspace{-3em} \| T_{N}\UP_{e_{1}} f (\cdot, x_1, \cdot) \|^{2}_{L^{2}_{t} L^{2}_{x'} (\R \times \R^{d - 1} )}
\\ 
& = \int_{\mcl{\R\times\R^{d-1}}} \;\, \bigg|\! \int_{\R} e^{2 \pi i \xi_1 x_1 - 4\pi^{2} i  t (|\xi_{1}|^{2}+|\xi'|^{2})} \chi_N (\xi_1, \xi') \1_{\mf{U}_{e_1}} (\xi_1, \xi') \FT{f} (\xi_1, \xi') d\xi_1 \bigg|^{2} \!\!\dd t \dd \xi'.
\end{aligned}
\]
Note that if we prove the fiber-wise bound
\[\eqnum\label{eq:ibfls}
\begin{aligned}
\hspace{10em}&\hspace{-10em} \int_\R \bigg| \int_{\R} e^{ 2 \pi i \xi_1 x_1  -4\pi^{2} i  t |\xi_{1}|^{2}} \chi_N (\xi_{1},\xi') \1_{\mf{U}_{e_1}} (\xi_{1},\xi') \FT{f} (\xi_{1},\xi') \dd \xi_1 \bigg|^{2}\!\! \dd t
\\
& \lesssim N^{-1}\int_{\R} |\chi_N (\xi_{1},\xi') \1_{\mf{U}_{e_1}} (\xi_{1},\xi') \FT{f} (\xi_{1},\xi') |^{2} \dd   \xi_1 
\end{aligned}
\]
for any fixed $x_1 \in \R$ and $\xi' \in \R^{d- 1}$,  then  \eqref{eq:dir-local-smoothing} with $\mf{c} = 2$ follows after integration in $\xi'$ and an application of  Plancherel's identity. Split the integration domain in $\xi_1$ as \(\R=(0, \infty)\cup(-\infty, 0)\) and we only show \eqref{eq:ibfls} for \((0,\infty)\) in place of \(\R\). The case of \((-\infty,0)\) is treated similarly.  A change of variables \(\theta=\xi_{1}^{2}\) and Plancherel's identity in $t$ imply that 
\[ 
\begin{aligned}
\hspace{3em}&\hspace{-3em}\int_{\R} \bigg| \int_{0}^{\infty} e^{i 2 \pi \xi_1 x_1  -4\pi^{2} i  t |\xi_{1}|^{2}} \chi_N (\xi_{1},\xi') \1_{\mf{U}_{e_1}} (\xi_{1},\xi') \FT{f} (\xi_{1},\xi') \dd \xi_{1} \bigg|^{2}\!\! \dd t
\\
& = \int_{\R} \bigg| \int_{ \R} e^{2 \pi i x_1 \sqrt{\theta} - 4 \pi^{2} i  t \theta} \chi_N (\sqrt{\theta}, \xi') \1_{\mf{U}_{e_1}} (\sqrt{\theta}, \xi') \FT{f} (\sqrt{\theta}, \xi') \chi_{\theta \geq 0} \frac{\dd  \theta}{\sqrt{4\theta}} \bigg|^{2}\!\! \dd t
\\
& = \int_{\R} \Big|  e^{2 \pi i x_1 \sqrt{\theta}} \chi_N (\sqrt{\theta}, \xi') \1_{\mf{U}_{e_1}} (\sqrt{\theta}, \xi') \FT{f} (\sqrt{\theta}, \xi') \chi_{\theta \geq 0} \frac{1}{\sqrt{4\theta}} \Big|^{2} \dd \theta
\\
& = \int_{ 0}^{\infty} \Big|  e^{2 \pi i x_1 \sqrt{\theta}} \chi_N (\sqrt{\theta}, \xi') \1_{\mf{U}_{e_1}} (\sqrt{\theta}, \xi') \FT{f} (\sqrt{\theta}, \xi') \Big|^{2} \frac{1}{4\theta} \dd \theta \,.
\end{aligned}
\]
Finally, the inverse change of variables \(\xi_{1}=\sqrt{\theta}\) gives that
\[
\begin{aligned}
\hspace{10em}&\hspace{-10em} \int_{ 0}^{\infty} \Big|  e^{2 \pi i x_1 \sqrt{\theta}} \chi_N (\sqrt{\theta}, \xi') \1_{\mf{U}_{e_1}} (\sqrt{\theta}, \xi') \FT{f} (\sqrt{\theta}, \xi') \Big|^{2} \frac{1}{4\theta} \dd \theta
\\ 
& \lesssim \int_{ 0}^{\infty} \Big|  e^{2 \pi i x_1 \xi_{1}} \chi_N (\xi_{1}, \xi') \1_{\mf{U}_{e_1}} (\xi_{1}, \xi') \FT{f} (\xi_{1}, \xi') \Big|^{2} \frac{1}{\xi_{1}} \dd \xi_{1}
\\
& \lesssim N^{-1} \int_{\R} |\chi_N (\xi_{1}, \xi') \1_{\mf{U}_{e_1}} (\xi_{1}, \xi')  \FT{f} (\xi_1, \xi') |^{2} d \xi_1,
\end{aligned}
\]
where the last inequality holds because $\chi_N (\xi_{1}, \xi') \1_{\mf{U}_{e_1}} (\xi_{1}, \xi') = 0$ unless 
\[
 |(\xi_{1},\xi')|\approx N \qquad \text{and} \qquad  \frac{| \xi_{1} |}{| (\xi_{1},\xi') |} >  (\sqrt{4d})^{-1} \,,
\]
which implies that $\frac{1}{\xi_1} \lesssim N^{-1}$.  The required bound \eqref{eq:ibfls} follows.
\end{proof}

\section{Non-homogeneous estimates}\label{sec:nonhomogeneous-estimates}

In \Cref{sec:linear-preliminaries} we established estimates on
solutions of the linear homogeneous Cauchy problem
\eqref{eq:linear-schroedinger}. In this section we study solutions
to the non-homogeneous Schrödinger equation
\[\eqnum\label{eq:non-homogeneous-schroedinger}
\begin{dcases}
(i \partial_{t} + \Delta) v = h(t,x) & \text{on } \R \times \R^{d},
\\
v (0,x) = v_{0}(x) \,. &
\end{dcases}
\]
We begin by introducing two norms, \(X^{\xs}\) and \(Y^{S}\), that are
obtained as combinations of norms introduced in
\Cref{sec:linear-preliminaries}. Next, we obtain boundedness
properties of the solutions to \eqref{eq:non-homogeneous-schroedinger}
in terms of \(X^{\xs}\) and \(Y^{S}\). The solution $v$ of
\eqref{eq:non-homogeneous-schroedinger} can be expressed using
Duhamel's formula \eqref{eq:duhamel} as
\[
v=\mbb{I}_{v_{0}}(h)\eqd  e^{ i t \Delta} v_{0} - i\int_{0}^{t} e^{i (t - s) \Delta} h(s, x) \dd s\,.
\]
Therefore our results can be seen as properties of
the mapping $(v_{0}, h) \mapsto \mbb{I}_{v_{0}}(h)=v$.

For any time interval \(I\subseteq\R\), \(I\ni0\) we introduce two families of
space-time norms denoted \(X^{\xs}(I)\) and \(Y^{S}(I)\). The norms
\(X^{\xs}(I)\) are well suited to control the solution \(v\) with no
restriction on the Fourier support, whereas \(Y^{S}(I)\) is adapted to
controlling the solution \(v\) when
\(\diam\big(\spt (\FT{v_{0}})\big) \lesssim 1\) and
\(\diam\big(\spt (\FT{h})\big)\lesssim 1\). Later, in
\Cref{sec:multilinear-probabilistic-estimates}, we show that, thanks
to the unit scale randomization of the initial data, the terms in the
explicit multilinear expansion of the solution
\eqref{eq:NLS-deterministic} can also be controlled in the
\(Y^{S}(I)\) norms.

To be more precise, let $\epsilon_{0}\in\big(0,2^{-100}\big)$ be sufficiently small
(to be chosen appropriately below depending on other parameters) and
we allow all subsequent constants to depend implicitly on
\(\epsilon_{0}\). For any interval \(I\subseteq\R\) and for any $\sigma \in \R$ we set
\[\eqnum\label{eq:def:X-norm}
\begin{aligned}[c] &
\|v\|_{X^{\sigma} (I)} \eqd \Big( \sum_{N \in 2^{\mathbb{N}}} N^{2 \sigma} \|\LP_{N} v\|^{2}_{X_{N} (I)} \Big)^{\frac{1}{2}} \,,
\\
&  \begin{aligned}[t]
& \|v\|_{X_{N} (I)}  \eqd \|v\|_{L_{t}^{\frac{2}{\epsilon_0}} L_{x}^{\frac{2}{1 - \epsilon_{0}}} (I)} + \|v\|_{L_{t}^{\frac{2}{1 - \epsilon_{0}}} L_{x}^{\frac{2d}{d-2}\frac{1}{1 - \epsilon_{0}}} (I)}
\\
& \quad + \sum_{l = 1}^{d} 
\Big( N^{-\frac{d-1}{2}} \|v\|_{L_{e_{l}}^{( \frac{2}{1 - \epsilon_{0}}, \frac{2}{\epsilon_{0}}, \frac{2}{\epsilon_{0}})} (I)}
+ N^{-\frac{1}{2}} \|v\|_{L_{e_{l}}^{( \frac{2}{1 - \epsilon_{0}},\frac{2}{\epsilon_{0}}, \frac{\mf{c}_{0}}{1-\epsilon_{0}})} (I)}\Big)
\\
& \quad + \sum_{l = 1}^{d}  \Big( N^{\frac{1}{2}} \| \UP_{e_{l}} v \|_{L_{e_{l}}^{(\frac{2}{\epsilon_{0}}, \frac{2}{1 - \epsilon_{0}}, \frac{2}{1 - \epsilon_{0}} )} (I)} + N^{-\frac{d-2}{2}}\|\UP_{e_{l}} v\|_{L_{e_{l}}^{(\frac{2}{\epsilon_{0}}, \frac{2}{1 - \epsilon_{0}}, \frac{2}{\epsilon_{0}} )} (I)}\Big)\,,
\end{aligned}
\end{aligned}
\]
and 
\[\eqnum\label{eq:def:Y-norm}
\begin{aligned}[c] 
& \|v\|_{Y^{\sigma} (I)} \eqd \Big( \sum_{N \in 2^{\mathbb{N}}} N^{2 \sigma} \|\LP_{N} v\|^{2}_{Y_{N} (I)} \Big)^{\frac{1}{2}},
\\
& \begin{aligned}[t]
& \|v\|_{Y_{N} (I)} \eqd  \|v\|_{L_{t}^{\frac{2}{\epsilon_0}} L_{x}^{\frac{2}{1 - \epsilon_{0}}} (I)} + \|v\|_{L_{t}^{\frac{2}{1 - \epsilon_{0}}} L_{x}^{\frac{2d}{d-2}\frac{1}{1 - \epsilon_{0}}} (I)}
\\
& \quad+ \sum_{l = 1}^{d}  \Big( N^{-\frac{1}{2}} \|v\|_{L_{e_{l}}^{( \frac{2}{1 - \epsilon_{0}}, \frac{2}{\epsilon_{0}}, \frac{2}{\epsilon_{0}} )} (I)} + N^{-\frac{1}{2}} \|v\|_{L_{e_{l}}^{( \frac{2}{1 - \epsilon_{0}},\frac{2}{\epsilon_{0}}, \frac{\mf{c}_{0}}{1-\epsilon_{0}} )} (I)}\Big)
\\
&\quad+ \sum_{l = 1}^{d} \Big( N^{\frac{1}{2}} \| \UP_{e_{l}} v \|_{L_{e_{l}}^{(\frac{2}{\epsilon_{0}}, \frac{2}{1 - \epsilon_{0}}, \frac{2}{1 - \epsilon_{0}})} (I)} + N^{\frac{1}{2}}\|\UP_{e_{l}} v\|_{L_{e_{l}}^{(\frac{2}{\epsilon_{0}}, \frac{2}{1 - \epsilon_{0}}, \frac{2}{\epsilon_{0}})} (I)}\Big).
\end{aligned}
\end{aligned}
\]
Recall that $\mf{c}_{0}\eqd2\frac{d-1}{d-2}$ is the lower bound for the spatial
integrability exponent that appears in \Cref{lem:dir-maximal}.

The norms $X_N$ and $Y_N$ consist of a combination of different
space-time norms. In particular, the norms
$L_{t}^{\frac{2}{\epsilon_{0}}} L_{x}^{\frac{2}{1 - \epsilon_{0}}}$ and
$L_{t}^{\frac{2}{1 - \epsilon_{0}}} L_{x}^{\frac{2d}{d-2}\frac{1}{1 - \epsilon_{0}}} $ are close
to the admissible Strichartz norms \(L_{t}^{\infty} L_{x}^{2}\) and
\(L_{t}^{2} L_{x}^{\frac{2d}{d-2}}\) respectively (cfr.
\Cref{lem:strichartz}). The other components appearing in the
definitions of the norms \(X_{N}\) and \(Y_{N}\) are directional
space-time norms, close to the directional maximal norms of
\Cref{lem:dir-maximal} and to the directional smoothing norms of
\Cref{lem:dir-local-smoothing}. Note that the norms appearing in the
definitions of \(X_{N}\) and \(Y_{N}\) are the same, and they differ by
the \(N\)-dependent scaling factors that appears next to each
summand. These scaling factors are determined by bounds
\eqref{eq:dir-maximal}, \eqref{eq:dir-local-smoothing} or by bounds
\eqref{eq:dir-maximal-unit-scale} and
\eqref{eq:dir-local-smoothing-unit-scale} respectively.

To control the non-homogeneous term $h$ we use the norms
\(X^{*,\sigma}(I) \), related through duality to the space
$X^{\sigma} (I)$. Specifically, we define
\[\eqnum\label{eq:def:X-norm-dual}
\begin{aligned}[c]
& \|h\|_{X^{*,\sigma} (I)} \eqd \big( \sum_{N \in 2^{\mathbb{N}}} N^{2 \sigma} \|\LP_{N} h\|^{2}_{X_{N}^{*}(I)} \big)^{\frac{1}{2}},
\\
& \|h\|_{X_{N}^{*} (I)} \eqd \sup \Big\{ \Big| \int_{I \times \R^{d}} h_{*} (t,x ) h (t,x )  \dd t \dd x \Big| \st \|h_{*} \|_{X_{N} (I)} \leq 1 \Big\}
\end{aligned}
\]

We record the following simple facts about all the norms introduced above.
\begin{itemize}
\item The norms \(X^{\sigma}\), \(X^{*,\sigma}\), and \(Y^{\sigma}\) are non-decreasing with respect to \(\sigma\).
\item Since $N \geq1$, for any \(v\colon I\times\R^{d}\to\C\) we have
\[\eqnum\label{eq:X-Y-comparison}
\|v\|_{X_{N}(I)}\leq\|v\|_{Y_{N}(I)},\quad\text{and}\quad \|v\|_{X^{\sigma}(I)}\leq\|v\|_{Y^{\sigma}(I)}\quad\forall \sigma\in\R.
\]
\item For any \(N,N'\in2^{\N}\) with \(N'\approx N \) it holds that
\[\eqnum\label{eq:X-N-Nprime-comparison}
\|v\|_{X_{N}(I)}\approx \|v\|_{X_{N'}(I)}, \qquad\|v\|_{Y_{N}(I)}\approx \|v\|_{Y_{N'}(I)} \,,
\]
and in particular, the definition \eqref{eq:def:LP-modified} implies that
\[\eqnum\label{eq:X-N-comparability}
\big\|\tilde{\LP}_{N}v\big\|_{X_{N}(I)}\approx \sum_{\mcl{\substack{N'\in2^{\N}\\ \sfrac{1}{8}< \sfrac{N'}{N} <8}}} \|\LP_{N'}v\|_{X_{N'}(I)} \,.
\]
\item Since \(\LP_{N}\) are convolution operators with
\(L^{1}\)-bounded kernels, Young's convolution inequality yields
that
\[\eqnum\label{eq:X-N-PN-boundedness}
\big\|\LP_{N}v\big\|_{X_{N}(I)}\lesssim\big\|v\big\|_{X_{N}(I)}\,.
\]

\item For any $\sigma \in \R$, the space \( C^{\infty}_{c}(I\times\R)\) is dense in the spaces \(X^{\sigma}(I)\) and \(Y^{\sigma}(I)\) since all exponents appearing in the definition of the corresponding norms are finite.

\item The space $X_{N} ^{*}(I)$ has been introduced as the dual space
of $X_{N} (I)$; it can be shown that
\[\eqnum\label{eq:X-duality}
\|h\|_{X^{*,\sigma}(I)}\approx\sup\Big\{\Big| \int_{I\times\R^{d}} h^{*}(t,x)h(t,x)\dd t \dd x\Big| \st \big\|h^{*}\big\|_{X^{-\sigma}(I)}\leq 1
\Big\}\,,
\]
where the supremum is taken over
\(h^{*}\in C^{\infty}_{c}(I\times\R^{d})\). However, we never use this fact in this
paper.
\end{itemize}

Next, we formulate the main result of this section.

\begin{proposition}\label{prop:main-linear-estimate}
Fix $\sigma \in \R$, $0 < \epsilon\lesssim 1$, and $0 < \epsilon_0 \lesssim_{\epsilon} 1$.  Then there exists a constant
\(c=c(\sigma,\epsilon,\epsilon_{0})>0\) such that for any \(I\subseteq\R\), \(I\ni0\),
it holds that 
\[\eqnum\label{eq:non-homogeneous-bound-2}
\big\|\mbb{I}_{v_{0}}(h)\big\|_{X^{\sigma} (I)} \lesssim_{\epsilon} \big\langle|I|^{-1}\big\rangle^{-c} \Big( \|v_{0} \|_{H^{\sigma+\epsilon} _{x}(\R^{d})} +\|h\|_{X^{*,\sigma+\epsilon} (I)} \Big)\,.
\]
If \(\diam\big(\spt(\FT{v_{0}})\big) \leq2R \) and \(\diam\big(\spt(\FT{h})\big)\leq2R\), then 
\[\eqnum\label{eq:non-homogeneous-bound-unit-scale-2}
\big\|\mbb{I}_{v_{0}}(h)\big\|_{Y^{\sigma} (I)}
\lesssim_{R,\epsilon}
\big\langle|I|^{-1}\big\rangle^{-c}  \Big(\|v_{0} \|_{H^{\sigma+\epsilon}_{x} (\R^{d})} +\|h\|_{X^{*,\sigma+\epsilon} (I)} \Big)\,.
\]
More specifically, for any $N\in 2^{\mathbb{N}}$, we have 
\[\eqnum\label{eq:non-homogeneous-bound}
\big\|\LP_{N} \mbb{I}_{v_{0}}(h)\big\|_{X_{N} (I)}
\lesssim_{\epsilon} N^{\epsilon}  \big\langle|I|^{-1}\big\rangle^{-c} \big( \| \LP_{N} v_{0} \|_{L^{2}_{x}(\R^{d})}+ \|\LP_{N} h\|_{X_{N} ^{*}(I)}  \big) 
\]
and, if \(\diam\big(\spt(\FT{v_{0}})\big) \leq2R\), and \(\diam\big(\spt(\FT{h})\big)\leq2R\), then
\[\eqnum\label{eq:non-homogeneous-bound-unit-scale}
\big\|\LP_{N} \mbb{I}_{v_{0}}(h)\big\|_{Y_{N} (I)}
\lesssim_{R,\epsilon} N^{\epsilon}  \big\langle|I|^{-1}\big\rangle^{-c}  \big( \| \LP_{N} v_{0} \|_{L^{2} _{x}(\R^{d})}+ \|\LP_{N} h\|_{X_{N}^{*} (I)}  \big)\,.
\]
All implicit constants are allowed to depend on \(\epsilon\) and \(\epsilon_{0}\), but are independent of \(\sigma\), $I$, \(N\), and \(v_{0}\) and \(h\).
\end{proposition}

\begin{remark}
Notice that the function $|I| \mapsto \langle|I|^{-1}\big\rangle^{-c}$ is increasing and
satisfies
\[
\begin{aligned}[t]
\langle|I|^{-1}\big\rangle^{-c} \lesssim_{c} |I|^{c},\qquad
\lim_{|I| \to 0} \langle|I|^{-1}\big\rangle^{-c} = 0, \qquad 
\lim_{|I| \to \infty} \langle|I|^{-1}\big\rangle^{-c} = 1. 
\end{aligned}
\]
\end{remark}

Next, we state \Cref{prop:linear-nonhomogeneous-continuity-bound} and
\Cref{prop:adjoint-bound}. We then provide the proofs of these
statements before focusing on the proof of
\Cref{prop:main-linear-estimate}.

\Cref{prop:linear-nonhomogeneous-continuity-bound} is an 
analogue of \Cref{prop:main-linear-estimate} except with
\(\mbb{I}_{v_{0}}(h)\) controlled by Sobolev norms. This statement is used
to bound solutions to \eqref{eq:NLS-deterministic} in more
classical spaces. \Cref{prop:adjoint-bound} is the adjoint bound to
\Cref{prop:main-linear-estimate} with \(h=0\). This estimate is central
to prove scattering results.

\begin{proposition}\label{prop:linear-nonhomogeneous-continuity-bound}
Under the assumptions of \Cref{prop:main-linear-estimate} it holds that 
\[\eqnum\label{eq:continuity-bound}
\big\|\mbb{I}_{v_{0}}(h)\big\|_{C^{0}(I, H^{\sigma}_{x}(\R^{d}))}\lesssim_{\epsilon} \|v_{0}\|_{H^{\sigma}_{x}(\R^{d})}+ \|h\|_{X^{*,\sigma+\epsilon}(I)}.
\]
\end{proposition}
\begin{proposition}\label{prop:adjoint-bound}
Under the assumptions of \Cref{prop:main-linear-estimate}, for any
function \(h\in X^{*,\sigma+\epsilon}(\R)\) the limit
\[\eqnum\label{eq:limit-adjoint-existence}
\lim_{t\to+\infty}\int_{0}^{t}e^{-is\Delta} h(s)\dd s
\]
exists in \(H_x^{\sigma}(\R^{d})\) and satisfies
\[\eqnum\label{eq:limit-adjoint-bound}
\Big\| \lim_{t\to\infty}\int_{0}^{t}e^{-is\Delta} h(s)\dd s\Big\|_{H_x^{\sigma}(\R^{d})} \lesssim \|h\|_{X^{*,\sigma+\epsilon}(\R)}.
\]
Furthermore, if \(v_{0}\in H^{\sigma}_{x}(\R^{d})\) then the limit
\[\eqnum\label{eq:limit-scattering-existence}
\lim_{t\to+\infty} e^{-it\Delta}\mbb{I}_{v_{0}}(h)(t)
\]
exists in \(H_x^{\sigma}(\R^{d})\) and satisfies
\[\eqnum\label{eq:limit-scattering-bound}
\big\|e^{-it\Delta} \mbb{I}_{v_{0}}(h)(t)\big\|_{H^{\sigma}_{x}(\R^{d}))}\lesssim_{\epsilon} \|v_{0}\|_{H^{\sigma}_{x}(\R^{d})}+ \|h\|_{X^{*,\sigma+\epsilon}(I)}.
\]
\end{proposition}

\begin{proof}[Proof of \Cref{prop:linear-nonhomogeneous-continuity-bound}]
For conciseness, let \(v\eqd\mbb{I}_{v_{0}}(h)\). To show \eqref{eq:continuity-bound} it is sufficient to prove the point-wise bound
\[\eqnum\label{eq:Duhamel-Hs-pointwise-bound}
\|v(t)\|_{H^{\sigma}_{x}(\R^{d})}\lesssim \|v_{0}\|_{H_x^{\sigma}(\R^{d})}+ \|h\|_{X^{*,\sigma+\epsilon}(I)}.
\]
for any $t \in I$. Indeed, \eqref{eq:Duhamel-Hs-pointwise-bound} implies that 
\[
\|v\|_{L^{\infty}_{t}(I, H^{\sigma}(\R^{d}))}\lesssim \|v_{0}\|_{H^{\sigma}_{x}(\R^{d})}+ \|h\|_{X^{*,\sigma+\epsilon}(I)}\,,
\]
and continuity of $t \mapsto \big\|\mbb{I}_{v_{0}}(h)\big\| (t)$ is trivial for the dense class of
functions \(v_{0}\in C^{\infty}_{c}(\R^{d})\) and \(h \in C^{\infty}_{c}(I\times\R^{d})\). The
claim \eqref{eq:continuity-bound} then follows by approximation.

Fix  \(t\in I\) and let us deduce \eqref{eq:Duhamel-Hs-pointwise-bound} by duality from \Cref{prop:main-linear-estimate}. By Duhamel's formula \eqref{eq:duhamel} we have 
\[
\big\|v(t)\big\|_{H^{\sigma}_{x}(\R^{d})} 
\lesssim  \big\|e^{it\Delta} v_{0}\big\|_{H^{\sigma}_{x}(\R^{d})}  + 
\Big\|\int_{0}^{t} e^{i(t-s)\Delta} h(s) \dd s\Big\|_{H^{\sigma}_{x}(\R^{d})}.
\]
The definition of the Sobolev spaces \eqref{eq:Sobolev-spaces} and of the Schrödinger evolution group \eqref{eq:schroedinger-propagation-group}, imply for any $t\in I$ that
\[
\|e^{it\Delta} v_{0}\|_{H^{\sigma}_{x}(\R^{d})}=\|v_{0}\|_{H^{\sigma}_{x}(\R^{d})}
\]
and 
\[
\Big\|\int_{0}^{t} e^{i(t-s)\Delta} h(s) \dd s\Big\|_{H^{\sigma}_{x}(\R^{d})}
=\Big\|\int_{0}^{t} e^{-is\Delta} h(s) \dd s\Big\|_{H^{\sigma}_{x}(\R^{d})}.
\]
Then, using the \(H^{\sigma}_{x}(\R^{d})\) -- \(H^{-\sigma}_{x}(\R^{d})\) duality and that \(\sum_{N\in\N}\tilde{\LP}_{N}\LP_{N}\) is the identity operator (see \eqref{eq:def:LP-modified}), we obtain that 
\[
\begin{aligned}[t]
 \Big\|\int_{0}^{t} e^{-is\Delta} h(s) \dd s\Big\|_{H^{\sigma}_{x}(\R^{d})} &= \sup_{v_{*}} \Big| \int_{\R^{d}}\int_{0}^{t} \bar{v_{*}(x)} e^{-is\Delta} h(s,x) \dd s \dd x \Big|
\\
  &= \sup_{v_{*}} \Big| \int_{\R^{d}}\int_{0}^{t} \bar{e^{is\Delta} v_{*}(x)}  h(s,x) \dd s \dd x \Big|\\
&= \sup_{v_{*}} \Big|\sum_{N\in\N}\, \int_{\R^{d}}\!\int_{0}^{t} \bar{\LP_{N}e^{is\Delta} v_{*}(x)}  \tilde{\LP}_{N}h(s,x) \dd s \dd x \Big|\,,
\end{aligned}
\]
where the supremum is taken over $v_{*} \in H^{-\sigma}_{x}(\R^d)$ with $\|v_{*}\|_{H^{-\sigma}_{x}(\R^d)} \leq 1$. Using the triangle inequality, the \(X_{N}(I)\) -- \(X_{N}^{*}(I)\) duality, the Cauchy-Schwarz inequality, as well as \eqref{eq:X-N-comparability} with the definition \eqref{eq:def:X-norm} of the \(X_{N}\) norm, we obtain that 
\[
\begin{aligned}[t]
 \Big\|\int_{0}^{t} e^{-is\Delta} h(s) \dd s\Big\|_{H^{\sigma}_{x}(\R^{d})} &\leq   \sup_{v_{*}}  \sum_{N \in 2^{\N}}  N^{-\sigma-\epsilon}\big\|\LP_{N} e^{it\Delta}v_{*}\big\|_{X_{N}(I)} N^{\sigma+\epsilon}\big\|\tilde{\LP}_{N}h\big\|_{X_{N}^{*}(I)} \,,
\\
&  \leq   \sup_{v_{*}}  \big\| e^{it\Delta}v_{*}\big\|_{X^{-\sigma-\epsilon}(I)} \big\|h\big\|_{X^{\sigma+\epsilon,*}(I)} \,,
\end{aligned}
\]
Then \eqref{eq:non-homogeneous-bound-2} with \(\sigma\) replaced by \(-\sigma-\epsilon\), together with the condition \(\|v_{*}\|_{H^{-\sigma}_{x}(\R^{d})}\leq1\) provides us with the bound
\[
\big\| e^{it\Delta}v_{*}(t,x)\big\|_{X^{-\sigma-\epsilon}(I)}\lesssim\|v_{*}\|_{H^{-\sigma}_{x}(\R^{d})}\lesssim 1,
\]
and, therefore,
\[
\Big\|\int_{0}^{t} e^{i(t-s)\Delta} h(s) \dd s\Big\|_{H^{\sigma}_{x}(\R^{d})}\lesssim \big\|h\big\|_{X^{\sigma+\epsilon,*}(I)} \,,
\]
as desired.
\end{proof}
\begin{proof}[Proof of \Cref{prop:adjoint-bound}]
Since \(\|h\|_{X^{*,\sigma+\epsilon}(\R)}<\infty\) and all norms in $X^{*,\sigma+\epsilon}(\R)$ are Lebesgue norms with finite integrability exponent, it holds that 
\[ \eqnum\label{eq:alfh}
\lim_{T\to\infty}\big\|\1_{[T,\infty)} h\big\|_{X^{*,\sigma+\epsilon}(\R)}=0\,.    
\]
We claim that \( \Big(t\mapsto \int_{0}^{t} e^{-is\Delta} h(s)\dd s\Big)_{t > 0}\) is Cauchy in \(H^{\sigma}(\R^{d})\) as \(t\to+\infty\). Indeed, by \eqref{eq:alfh} 
for any \(\tilde{\epsilon}>0\) there exists \(T_{\tilde{\epsilon}}>0\) such that
\[ 
\big\|\1_{[T_{\tilde \varepsilon},\infty)} h\big\|_{X^{*,\sigma+\epsilon}(\R)} \leq   \tilde{\epsilon} \,.
\]
Then, for any $t_1, t_2$ with $T_{\tilde{\epsilon}}\leq t_{1}\leq t_{2}$ one has, by \eqref{eq:continuity-bound},
\[
\begin{aligned}
\Big\|\int_{0}^{t_2}e^{-is\Delta} h(s) \, \dd s - \int_{0}^{t_1}e^{-is\Delta} h(s) \, \dd s \Big\|_{H_x^{\sigma}(\R^{d})}
& =
\Big\|\int_{0}^{\infty}e^{-is\Delta}\1_{[t_{1},t_{2})}(s)h(s) \, \dd s \Big\|_{H_x^{\sigma}(\R^{d})}
\\
& \hspace{-10em}\lesssim \big\|\1_{[t_1,t_2]} h\big\|_{X^{*,\sigma+\epsilon}(\R)} \lesssim  \big\|\1_{[T_{\tilde{\epsilon}}, \infty)} h\big\|_{X^{*,\sigma+\epsilon}(\R)} \lesssim \tilde{\epsilon}\,.
\end{aligned}
\]
Hence, 
\[
\lim_{t\to\infty}\int_{0}^{t}e^{-is\Delta} h(s)\dd s 
\qquad \text{ exists in }  H^{\sigma}_{x}(\R^{d}).
\]
 Finally, using \eqref{eq:continuity-bound} once more
\[
\Big\|\lim_{t\to\infty}\int_{0}^{t}e^{-is\Delta} h(s)\dd s \Big\|_{H_x^\sigma(\R^d)}
\leq
\lim_{t\to\infty}\Big\|\int_{0}^{t}e^{-is\Delta} h(s)\dd s \Big\|_{H^\sigma_x(\R^d)} \lesssim \|h\|_{X^{*,\sigma+\epsilon}(\R)}
\]
and  \eqref{eq:limit-adjoint-bound} follows.

The existence of the limit \eqref{eq:limit-adjoint-existence} and the bound \eqref{eq:limit-scattering-bound} follow from \eqref{eq:limit-adjoint-existence} and from the bound \eqref{eq:limit-adjoint-bound} by noticing that
\[
e^{-it\Delta} \mbb{I}_{v_{0}}(h)(t)=v_{0} \mp \int_{0}^{t}e^{-s\Delta}h(s)\dd s.
\]
\end{proof}

The proof of \Cref{prop:main-linear-estimate} relies on two key
steps. First, in \Cref{lem:linear-evolution-disjoint-times} below, we
prove estimates analogous to \eqref{eq:non-homogeneous-bound} and
\eqref{eq:non-homogeneous-bound-unit-scale} with (see Duhamel's
formula \eqref{eq:duhamel})
\[\eqnum\label{eq:duh-max}
h\mapsto v = \int_{I} \1_{s<t}\,e^{i(t-s)\Delta}h(s,x)\dd s
\]
replaced by 
\[\eqnum\label{eq:duh-non-max}
h\mapsto \int_{I}e^{i(t-s)\Delta}h(s,x)\dd s \,.
\]
Second, the ``Christ-Kiselev'' procedure illustrated in the proof of
\Cref{prop:main-linear-estimate} allows one to deduce bounds on \eqref{eq:duh-max} from the bounds on \eqref{eq:duh-non-max}, completing the proof.

\begin{lemma} \label{lem:linear-evolution-disjoint-times} Fix any
\(0 < \epsilon \lesssim 1\), any \(0 < \epsilon_{0} \lesssim_{\epsilon} 1\). There exists a constant
\(c=c(\epsilon,\epsilon_{0})>0\) such that for any $I \subset \R$ and any $N\in 2^{\mathbb{N}}$ one has
\[\eqnum\label{eq:non-homogeneous-bound-J-Jprime}
\Big\| \int_{I} e^{ i (t - s) \Delta} \LP_{N} h (s) \dd s \Big\|_{X_{N} (I)}
\lesssim_{\epsilon}\big\langle|I|^{-1}\big\rangle^{-c} N^{\epsilon} \big\| \LP_{N} h \big\|_{X_{N}^{*} (I)} \,.
\]
If in addition \(\diam\big(\spt (\FT{h})\big)\leq2R\), then
\[\eqnum\label{eq:non-homogeneous-bound-unit-scale-J-Jprime}
\Big\| \int_{I} e^{ i (t - s) \Delta} \LP_{N} h (s) \dd s \Big\|_{Y_{N} (I)}
\lesssim_{R,\epsilon} \big\langle|I|^{-1}\big\rangle^{-c} N^{\epsilon} \big\| \LP_{N} h \big\|_{X_{N}^{*} (I)}.
\]
All implicit constants are allowed to depend on \(\epsilon,\epsilon_{0}\) but are independent of \(I\), $N$, and \(h\). 
\end{lemma}

\begin{proof}[Proof of \Cref{lem:linear-evolution-disjoint-times}]
For this proof we allow implicit constants to depend without mention on $\epsilon, \epsilon_0$, and  \(R>0\) (if the support is assumed to be bounded). 

\textbf{Step 1} We claim that for any $g\in C^{\infty}_{c}(\R^{d})$ we have
\[\eqnum\label{eq:L2-to-X-bound}
\Big\| e^{ i t \Delta} \LP_{N} g \Big\|_{X_{N} (I)} \lesssim N^{\epsilon} \| \LP_{N} g \|_{L^{2}(\R^{d})}
\]
and if \(\diam\big(\spt(\FT g)\big)\leq 2R\), then
\[\eqnum\label{eq:L2-to-Y-bound}
\Big\| e^{ i t \Delta} \LP_{N} g \Big\|_{Y_{N} (I)}
\lesssim
N^{\epsilon} \| \LP_{N} g \|_{L^{2}(\R^{d})} \,.
\]

Indeed, using Bernstein's inequality (see \Cref{lem:bernstein}) and Strichartz's estimate (see \Cref{lem:strichartz}) 
for admissible pairs $(\frac{2}{\epsilon_0}, \frac{2}{1-2\epsilon_0/d})$ and $(\frac{2}{1- \epsilon_0}, \frac{2}{d-2+2\epsilon_0})$
we find
\[\eqnum\label{eq:fefl}
\begin{aligned}[c]
& \big \| e^{ i t \Delta} \LP_{N} g \big\|_{L^{\frac{2}{\varepsilon_0}}_{t} L^{\frac{2}{1 - \epsilon_0}}_{x} ( I )} 
\begin{aligned}[t]
& \lesssim N^{| O (\epsilon_0) |} \big\| e^{ i t \Delta} \LP_{N} g \big\|_{L^{\frac{2}{\varepsilon_0}}_{t} L^{\frac{2}{1 - 2\epsilon_0/d}}_{x} ( I )}
\\
& \lesssim N^{  | O (\epsilon_0) |} \big\| \LP_{N} g\big \|_{L^2 (\R^{d})},
\end{aligned}
\\
& \big\| e^{ i t \Delta} \LP_{N} g \big\|_{L_{t}^{ \frac{ 2}{1-\varepsilon_0}} L_{x}^{ \frac{2d }{d-2} \frac{1}{1-\varepsilon_0}}  (I )} 
\begin{aligned}[t]
& \lesssim
N^{  | O (\epsilon_0) |} \big\| e^{ i t \Delta} \LP_{N} g \big\|_{L_{t}^{ \frac{ 2}{1-\varepsilon_0}} L_{x}^{ \frac{2d }{d-2 + 2\varepsilon_0} }  (I )}
\\
& \lesssim N^{  | O (\epsilon_0) |}\big\| \LP_{N} g \big\|_{L^2 (\R^{d})}.
\end{aligned}
\end{aligned}
\]
Using interpolation between $L_{e_l}^{(2,\infty ,\infty)}(I)$ and $L_{e_l}^{(\infty ,2,2)}(I)$ norms, the directional maximal estimate \eqref{eq:dir-maximal} with $\mf{c} = \infty$, the directional smoothing estimate \eqref{eq:dir-local-smoothing} with $\mf{c} = 2$, and the boundedness of \(\UP_{e_{l}}\) on \(L^{2}(\R^{d})\), we obtain  
\[
\begin{aligned}
\big\| e^{ i t \Delta} \LP_{N} \UP_{e_l} g \big\|_{L_{e_{l}}^{(\frac{2}{1-\varepsilon_0}, \frac{2}{ \epsilon_0}, \frac{2}{ \epsilon_{0}})} (I )}
\lesssim 
\big\| e^{ i t \Delta} \LP_{N} \UP_{e_l} g \big\|_{L_{e_{l}}^{(2, \infty, \infty)} (I )}^{1 - |O(\epsilon_0)|} 
\big\| e^{ i t \Delta} \LP_{N} \UP_{e_l} g \big\|_{L_{e_{l}}^{(\infty, 2, 2)} (I )}^{|O(\epsilon_0)|} 
\\ \qquad 
\lesssim N^{ \frac{d-1}{2} + | O (\epsilon_{0}) |} \big\| \LP_{N} \UP_{e_l} g \big\|_{L^2 (\R^{d})}
\lesssim N^{ \frac{d-1}{2} + | O (\epsilon_{0}) |} \big\| \LP_{N} g \big\|_{L^2 (\R^{d})}
\end{aligned}
\]
and 
\[
\begin{aligned}
\big\| e^{ i t \Delta} \LP_{N}\UP_{e_l} g \big\|_{L_{e_{l}}^{(\frac{2}{\varepsilon_0}, \frac{2}{1- \epsilon_0}, \frac{2}{ 1-\epsilon_{0}})} (I )}
\lesssim 
\big\| e^{ i t \Delta} \LP_{N} \UP_{e_l} g \big\|_{L_{e_{l}}^{(2, \infty, \infty)} (I )}^{|O(\epsilon_0)|} 
\big\| e^{ i t \Delta} \LP_{N} \UP_{e_l} g \big\|_{L_{e_{l}}^{(\infty, 2, 2)} (I )}^{1-|O(\epsilon_0)|} 
\\ \qquad 
\lesssim N^{ -\frac{1}{2} + | O (\epsilon_{0}) |} \big\| \LP_{N} \UP_{e_l} g \big\|_{L^2 (\R^{d})}
\lesssim N^{ -\frac{1}{2} + | O (\epsilon_{0}) |} \big\| \LP_{N}  g \big\|_{L^2 (\R^{d})}\,.
\end{aligned}
\]
Also, an interpolation between $L_{e_l}^{(2,\infty , \frac{\mf{c}_0}{1- \epsilon_{0}})}(I)$ and $L_{e_l}^{(\infty ,2,\frac{\mf{c}_0}{1- \epsilon_{0}})}(I)$ norms, directional estimates \eqref{eq:dir-maximal} and \eqref{eq:dir-local-smoothing} with $\mf{c} = \frac{\mf{c}_0}{1- \epsilon_{0}} > \mf{c}_0$, and the boundedness 
of \(\UP_{e_{l}}\) on \(L^{2}(\R^{d})\)   yield
\[
\big\| e^{ i t \Delta} \LP_{N} \UP_{e_l} g \big\|_{L_{e_{l}}^{(\frac{2}{1-\varepsilon_0}, \frac{2}{\epsilon_0}, \frac{\mf{c}_0}{1- \epsilon_{0}})} (I)}
\lesssim N^{ \frac{1}{2} + | O (\epsilon_{0}) |} \big\| \LP_{N} g \big\|_{L^{2}(\R^{d})}\,.
\]
Finally,  an interpolation between $L_{e_l}^{(\infty, 2, \frac{2}{ \epsilon_{0}})}(I)$ and $L_{e_l}^{(2 ,\infty, \frac{2}{ \epsilon_{0}})}(I)$ norms and directional estimates \eqref{eq:dir-maximal} and \eqref{eq:dir-local-smoothing} with $\mf{c} = \frac{2}{ \epsilon_{0}}$ and small $\epsilon_0 > 0$ yield
\[
\big\| e^{ i s \Delta} \LP_{N} \UP_{e_l} g \big\|_{L_{e_{l}}^{(\frac{2}{\varepsilon_0}, \frac{2}{1- \epsilon_0}, \frac{2}{ \epsilon_{0}})} (I )}
\lesssim N^{ \frac{d-2}{2} + | O (\epsilon_{0}) |} \big\| \LP_{N} g \big\|_{L^{2}_{x} (\R^{d})}\,. 
\]
Summing these estimates provides us with \eqref{eq:L2-to-X-bound}
as long as $\epsilon_0 \lesssim_\epsilon 1$.

If \(\diam\big(\spt(\FT{g})\big)\leq2R\), then 
the proof of \eqref{eq:L2-to-Y-bound} is analogous to the proof of 
\eqref{eq:L2-to-X-bound}, where we respectively use 
\eqref{eq:dir-maximal-unit-scale} and \eqref{eq:dir-local-smoothing-unit-scale} instead of \eqref{eq:dir-maximal} and \eqref{eq:dir-local-smoothing}.

\textbf{Step 2} We claim that if \(|I|<1\), then (independently of the support of \(\FT{g}\)):
\[\eqnum\label{eq:trivial-duhamel-bound}
\big\| e^{ i t \Delta} \LP_{N} g \big\|_{X_{N} (I)}
+
\big\| e^{ i t \Delta} \LP_{N} g \big\|_{Y_{N} (I)} \lesssim N^{d} |I|^{\frac{\epsilon_{0}}{2}} \| \LP_{N} g \|_{L^{2}(\R^{d})}.
\]
Note that every summand in the definitions \eqref{eq:def:X-norm} and \eqref{eq:def:Y-norm} of $X_N(I)$ and $Y_N(I)$ is a norm of the form $L^p_tL^q_x(I)$ with $p,q\in\big[2,\sfrac{2}{\epsilon_{0}}\big]$, or $L^{(\mf{a},\mf{b},\mf{c})}_{e_{l}}(I)$ for some $\mf{a},\mf{b},\mf{c}\in\big[2,\sfrac{2}{\epsilon_{0}}\big]$, and \(l\in\{1,\ldots,d\}\). Then the Strichartz estimates (\Cref{lem:strichartz}), Bernstein's inequality (\Cref{lem:bernstein}), and Hölder's inequality, yield
\[
\begin{aligned}
\big\|e^{ i t \Delta} \LP_{N} g \big\|_{L^p_tL^q_x(I)} & \lesssim N^{\frac{d}{2}}\big\|e^{ i t \Delta} \LP_{N} g \big\|_{L^p_tL^2_x(I)}
\\
& \lesssim N^{d} |I|^\frac{1}{p} \big\|e^{ i t \Delta} \LP_{N} g \big\|_{L^\infty_tL^2_x(I)} 
 \lesssim  N^{d} |I|^{\frac{\epsilon_{0}}{2}} \big\| \LP_{N} g \big\|_{L^{2}(\R^{d})} \,.
\end{aligned}
\]
and if \(\mf{a}\geq\mf{b}\), then (after an exchange of integrals with respect to $x_1$ and $t$)
\[
\begin{aligned}
\big\|e^{ i t \Delta} \LP_{N} g \big\|_{L^{(\mf{a}, \mf{b}, \mf{c})}_{e_{l}}(I)}
& \lesssim N^{\frac{d}{2}}\|e^{ i t \Delta} \LP_{N} g \|_{L^{\mf{b}}_{t} L^{2}_{x}(I)}
\\
&\lesssim N^{d} |I|^{\frac{1}{\mf{b}}} \|e^{ i t \Delta} \LP_{N} g \|_{L^{\infty}_{t} L^{2}_{x}(I)}
 \lesssim  N^{d} |I|^{\frac{\epsilon_{0}}{2}} \| \LP_{N} g \|_{L^{2}(\R^{d})} \,.
\end{aligned}
\]
If \(\mf{a}<\mf{b}\), then \(\mf{c}>\mf{c}_{0}\) and the directional maximal estimate \eqref{eq:dir-maximal} implies that 
\[
\begin{aligned}[t]
\big\|e^{ i t \Delta} \LP_{N} g \big\|_{L^{(\mf{a}, \mf{b}, \mf{c})}_{e_{l}}(I)}
&\lesssim N^{\frac{1}{2}}\|e^{ i t \Delta} \LP_{N} g \|_{L^{(2,\mf{b}, \mf{c})}_{e_{l}}(I)}
\\
& \lesssim N^{\frac{1}{2}} |I|^{\frac{1}{\mf{b}}} \|e^{ i t \Delta} \LP_{N} g \|_{L^{(2, \infty, \mf{c})}_{e_{l}}(I)} 
\lesssim  N^{d} |I|^{\frac{\epsilon_{0}}{2}} \| \LP_{N} g \|_{L^{2}(\R^{d})} \,.
\end{aligned}
\]
Thus, \eqref{eq:trivial-duhamel-bound} follows from Minkowski's inequality.

\textbf{Step 3}
Interpolating bounds \eqref{eq:trivial-duhamel-bound} with \eqref{eq:L2-to-X-bound} gives for $\theta=\frac{\epsilon}{d}\in(0,1)$ that 
\[\eqnum\label{eq:srdesx}
\begin{aligned}
\big\| e^{ i t \Delta} \LP_{N} g \big\|_{X_{N} (I)}
& \lesssim \min\Big( N^{\theta d + (1- \theta)\epsilon} |I|^{\theta\frac{\epsilon_0}{2}}, N^{\epsilon} \Big) \big\| \LP_{N} g \big\|_{L^{2}(\R^{d})}
\\
& \lesssim N^{2 \epsilon}\langle |I|^{-1}\rangle^{-c} \|\LP_{N}g \|_{L^{2}(\R^{d})} \,.
\end{aligned}
\]
for \(c=c_{\theta}=\frac{\epsilon\epsilon_{0}}{2d}\). Analogously, if \(\diam\big(\spt(\FT {g})\big)\leq 2R\), 
by interpolating \eqref{eq:trivial-duhamel-bound} with \eqref{eq:L2-to-Y-bound}
one obtains
\[\eqnum\label{eq:srdesy}
\big\| e^{ i t \Delta} \LP_{N} g \big\|_{Y_{N} (I)} \lesssim N^{2\epsilon}\langle |I|^{-1}\rangle^{-c} \|g \|_{L^{2}(\R^{d})} \,.
\]

\textbf{Step 4}
The rest of the proof closely follows \cite[Lemma $3.2$]{casterasAlmostSureLocal2022}, but for completeness we recall details here. By dualizing \eqref{eq:L2-to-X-bound}  we obtain that 
\[ \eqnum\label{eq:iwtgtb}
\Big\| \int_{I} e^{- i s \Delta} \LP_{N} h \dd s \Big\|_{L^{2} (\R^{d})} \lesssim N^{ \epsilon} \big\| \LP_{N} h \big\|_{X_{N}^{*} (I)}
\]
for any \(h\in C^{\infty}_{c}(I\times\R^{d})\). Consequently, \eqref{eq:srdesx} applied to the function \(g := \int_{I}e^{ -is \Delta} \LP_{N} h (s) \dd s\) implies
\[
\Big\| \int_{I} e^{ i (t - s) \Delta} \LP_{N} h (s) \dd s \Big\|_{X_{N} (I)}
\begin{aligned}[t]
& = \Big\| e^{it\Delta}\int_{I} e^{ -i s \Delta} \LP_{N} h (s) \dd s \Big\|_{X_{N} (I)}
\\
& \lesssim N^{ 2\epsilon}\big\langle|I|^{-1}\big\rangle^{-c}  \Big\| \int_{I} e^{- i s \Delta} \LP_{N} h (s) \dd s \Big\|_{L^{2}(\R^{d})}
\\ 
& \lesssim  N^{ 3\epsilon} \big\langle|I|^{-1}\big\rangle^{-c}\big\| \LP_{N} h \big\|_{X_{N}^{*} (I)}
\end{aligned}
\]
and \eqref{eq:non-homogeneous-bound-J-Jprime} follows
with $\epsilon$ replaced by $3\epsilon$.
Similarly, if \(\diam\big(\spt(\FT{h})\big)\leq2R\), then from \eqref{eq:L2-to-Y-bound} and \eqref{eq:iwtgtb} it follows that
\[
\Big\| \int_{I} e^{ i (t - s) \Delta} \LP_{N} h (s) \dd s \Big\|_{Y_{N} (I)}
\begin{aligned}[t]
& = \Big\| e^{it\Delta}\int_{I} e^{ -i s \Delta} \LP_{N} h (s) \dd s \Big\|_{Y_{N} (I)}
\\
& \lesssim \big\langle|I|^{-1}\big\rangle^{-c} N^{ 2\epsilon} \Big\| \int_{I} e^{- i s \Delta} \LP_{N} h (s) \dd s \Big\|_{L^{2}(\R^{d})}
\\ 
& \lesssim \big\langle|I|^{-1}\big\rangle^{c} N^{ 3\epsilon} \big\| \LP_{N} h \big\|_{X_{N}^{*} (I)}\,,
\end{aligned}
\]
where we used that if \(\diam\big(\spt(\FT{h})\big)\leq2R\), then \(\diam\big(\spt(\FT{g})\big)\leq2R\). This concludes the proof.
\end{proof}

Next, we finish the proof of \Cref{prop:main-linear-estimate} by
employing the ``Christ-Kiselev'' procedure: a dyadic partitioning
argument of the time direction. In its simplest form the
``Christ-Kiselev'' Lemma can be stated as follows. Suppose that
\(T:L^{p}(\R) \to L^{q}(\R)\) is a bounded linear operator, and
\(1\leq p<q\leq\infty\), then the maximal operator
\(T_{*} : L^{p}(\R) \to L^{q}(\R)\) defined as
\[
T_{*}h(t)\eqd\sup_{c}\Big|T\big(\1_{(-\infty,c)}h\big)(t)\Big|.
\]
is also bounded from \(L^{p}(\R) \) to \( L^{q}(\R)\).

In our case the operator \eqref{eq:duh-non-max} plays the role of
\(T\), while \eqref{eq:duh-max} is point-wise controlled by \(T_{*}\).  The
added difficulty stems from the fact that our norms (\(X_{N}\),
\(Y_{N}\), and \(X_{N}^{*}\)) are not simple Lebesgue norms in the time
variable. However, we use that all integrability exponents in the
space-time norm \(X_{N}\) and \(Y_{N}\) are larger than \(2\), while the
norm \(X_{N}^{*}\) can be thought of as a combination of mixed-exponent
Lebesgue space with exponents smaller than \(2\). Since we employ a
procedure similar to our previous work \cite{casterasAlmostSureLocal2022}, we merely sketch the
proof, and we highlight the differences  due to the different
definitions of the spaces $X_{N}(I)$ and \(Y_{N}(I)\).

\begin{proof}[Proof of \Cref{prop:main-linear-estimate}]
Since the proof follows the arguments of \cite[Lemma 3.3]{casterasAlmostSureLocal2022}, we only
provide main ideas and highlight differences, which are due to the
specific definitions of the spaces $X_{N}(I)$ and \(Y_{N}(I)\). We allow
implicit constants to depend on $\epsilon$ and $\epsilon_0$ without mention.

The estimate \eqref{eq:non-homogeneous-bound-2} follows from
\eqref{eq:non-homogeneous-bound} after squaring, multiplying by
$N^{2\sigma}$, and summing with respect to $N \in 2^\N$. Thus, we concentrate
on \eqref{eq:non-homogeneous-bound}. Also, we only focus on the proof
of \eqref{eq:non-homogeneous-bound} since the proof of
\eqref{eq:non-homogeneous-bound-unit-scale} only differ in the
definition of the space $\tilde{X}_N$ below, with obvious modifications.

The estimate \eqref{eq:non-homogeneous-bound} follows from the
Duhamel's formula \eqref{eq:duhamel}, and from the linear homogeneous
bound \eqref{eq:L2-to-X-bound} once we prove
\[ \eqnum\label{claim:main-linear-estimate}
\Big\| \int_{0}^{t} e^{ i (t - s) \Delta} \LP_{N} h (s) \dd s \Big\|_{X_{N} (I)} \lesssim N^{|O(\epsilon_{0})|}\langle|I|^{-1}\big\rangle^{-c} \| \LP_{N} h \|_{X_{N} ^{*}(I)}.
\]
To prove \eqref{claim:main-linear-estimate}, let us define the norm
\[
\|v\|_{\tilde{X}_{N} (I)}^{\frac{2}{1 - \epsilon_{0}}} \eqd
\begin{aligned}[t]
& \| v \|_{L_{t}^{\frac{2}{\varepsilon_0}} L_{x}^{\frac{2}{1- \epsilon_{0}}} (I)}^{\frac{2}{1 - \epsilon_{0}}}
+ \| v \|_{L_{t}^{ \frac{2}{1-\varepsilon_0}} L_{x}^{ \frac{2d }{d-2} \frac{1}{1-\varepsilon_0}} (I)}^{\frac{2}{1 - \epsilon_{0}}}
\\
& \quad + \sum_{l = 1}^{d} \Big(N^{-\frac{d-1}{1-\varepsilon_0}} \big\| v \big\|_{L_{e_{l}}^{(\frac{2}{1-\varepsilon_0}, \frac{2}{ \epsilon_{0}}, \frac{2}{ \epsilon_{0}} )} (I)}^{\frac{2}{1 - \epsilon_{0}}}
+ N^{-\frac{1}{1-\varepsilon_0}} \big\| v \big\|_{L_{e_{l}}^{(\frac{2}{1-\varepsilon_0}, \frac{2}{ \epsilon_{0}}, \frac{\mf{c}_0}{1- \epsilon_{0}} )} (I)}^{\frac{2}{1 - \epsilon_{0}}} \Big)
\\
& \quad + \sum_{l = 1}^{d} \Big(
\begin{aligned}[t]
& N^{\frac{1}{ 1-\varepsilon_0 }} \|U_{e_{l}} v \|_{L_{e_{l}}^{(\frac{2}{\epsilon_{0}}, \frac{2}{1-\varepsilon_0}, \frac{2}{1-\varepsilon_0})} (I)}^{\frac{2}{1 - \epsilon_{0}}}
\\
& \qquad + N^{-\frac{d-2}{ 1-\varepsilon_0 }} \|U_{e_{l}} v \|_{L_{e_{l}}^{(\frac{2}{\epsilon_{0}}, \frac{2}{1-\varepsilon_0}, \frac{2}{\varepsilon_0})} (I)}^{\frac{2}{1 - \epsilon_{0}}} \Big),
\end{aligned}
\end{aligned}
\]
that is equivalent to \(X_{N} (I)\) uniformly in \(N\) and in the interval \(I\). We claim that the norm $\tilde{X}_N(I)$ possesses the following property:
\[\eqnum\label{eq:X-concat-sim}
\| v_{1} + v_{2} \|_{\tilde{X}_{N} (I)}^{\frac{2}{1 - \epsilon_{0}}} \leq \| v_{1} \|_{\tilde{X}_{N} (J_{1})}^{\frac{2}{1 - \epsilon_{0}}} + \| v_{2}
\|_{\tilde{X}_{N} (J_{2})}^{\frac{2}{1 - \epsilon_{0}}}
\]
as long as $v_{1}, v_{2} \st I\times \R^{d} \rightarrow \C$ are supported respectively on disjoint time intervals $J_{1}, J_{2} \subset I$. 
Let us postpone showing \eqref{eq:X-concat-sim} until the end of the proof of the proposition. 

By assuming \eqref{eq:X-concat-sim}, the rest of the proof follows
the proof of \cite[Lemma 3.3]{casterasAlmostSureLocal2022} line by line , and therefore we only
outline main ideas.  The norm $\tilde{X}_{N}^{*}(I)$, dual to $\tilde{X}_{N}(I)$
(cf. \eqref{eq:def:X-norm-dual}), possesses a property converse to \eqref{eq:X-concat-sim}: for
any functions $h_{1}, h_{2} \st I \times \R^{d} \rightarrow \C$ supported on disjoint time
intervals $J_{1}, J_{2} \subset I$ it holds that
\[ \eqnum\label{eq:Xstar-concat-sim}
\big\| h_{1} + h_{2} \big\|_{\tilde{X}_{N}^{*} (I)}^{\frac{2}{1 +    \varepsilon_{0}}}
\geq
\| h_{1} \|_{\tilde{X}_{N}^{*} (J_{1})}^{\frac{2}{1 + \varepsilon_{0}}} + \| h_{2} \|_{\tilde{X}_{N}^{*} (J_{2})}^{\frac{2}{1 + \varepsilon_{0}}}\,.
\]
Consequently, by induction, if functions $(v_{k})_{k\in\N}$ are supported
on disjoint time intervals $(J_k)_{k \in \N}$, then
\begin{align}
\eqnum\label{eq:X-concat}
& \Big\| \sum_{k} v_{k}  \Big\|_{\tilde{X}_{N} (I)}^{\frac{2}{1 -    \varepsilon_{0}}} \leq \sum_{k} \| v_{k} \|_{\tilde{X}_{N} (J_{k})}^{\frac{2}{1- \varepsilon_{0}}}\,.
\end{align}
Analogously, if $(h_{k})_{k\in\N}$ are supported on disjoint time
intervals $(J_k)_{k \in \N}$, then
\begin{align}
\eqnum\label{eq:Xstar-concat} &\Big\| \sum_{k} h_{k} \Big\|_{\tilde{X}_{N}^{*}
(I)}^{\frac{2}{1 + \varepsilon_{0}}} \geq \sum_{k} \| h_{k} \|_{\tilde{X}_{N}^{*}
(J_{k})}^{\frac{2}{1 + \varepsilon_{0}}} \,.
\end{align}

Fix $h \st I \times \R^{d} \rightarrow \C$ and without loss of generality suppose $I =
[0, T_{0}]$. Then, there is a sequence of intervals $\big\{ I^{n}_{k} =
[t^{n}_{k}, t^{n}_{k + 1}) \big\}_{n \in \N, k \in \{ 0, \ldots, 2^{n} - 1 \}}$ such
that $t_{k}^{n} \leq t^{n}_{k + 1}$, the intervals $\{ I^{n}_{k} \}_{k \in \{
0, \ldots, 2^{n} - 1 \}}$ form a partition of $I$ and $I^{n}_{k} = I_{2
k}^{n + 1} \cup I_{2 k + 1}^{n + 1}$. Furthermore, the intervals are
constructed so that
\[\eqnum\label{eq:dgct}   
\| h  \|_{\tilde{X}_{N}^{*} (I^{n}_{k})} \lesssim 2^{- \frac{1 + \varepsilon_{0}}{2}n} \| h \|_{\tilde{X}_{N}^{*} (I)} . 
\]
The construction of such intervals is provided in \cite[Lemma 3.3]{casterasAlmostSureLocal2022}
and uses only the continuity of the map
\[
s\mapsto \|\1_{(-\infty,s)}(t)h(t,x)\|_{\tilde{X}_{N}^{*}}.
\]
Then for any $t \in I$ we have
\[
\1_{[0, t)} (s) = \sum_{n = 1}^{\infty} {\sum^{2^{n }-1}_{k = 0}}  \1_{I_{ k}^{n}} (s) \1_{I_{ k + 1}^{n}} (t) 
\]
and by triangle inequality and \eqref{eq:X-concat} for any regular
$(t, s, x) \mapsto F(t, s, x)$
\[
\Big\| \int_{0}^{t} F(t, s) \dd s \Big\|_{\tilde{X}_{N} (I)}  
\lesssim \sum_{n = 1}^{\infty} \Big( {\sum^{2^{n} - 1}_{k = 0}}  \Big\| \int_{I_{ k}^{n}} F(t, s) \dd s \Big\|_{\tilde{X}_{N} (I_{ k + 1}^{n})}^{\frac{2}{1 - \varepsilon_{0}}} \Big)^{\frac{1 - \varepsilon_{0}}{2}} \,,
\]
where we suppressed the dependence of $F$ on $x$. For more details see
\cite[proof of Lemma 3.3]{casterasAlmostSureLocal2022}.  Then, since norms $X_{N} (I)$ and
$\tilde{X}_{N} (I)$ are equivalent and $I_{ k}^{n}$ and $I_{ k + 1}^{n}$
are disjoint, it follows from \eqref{eq:Xstar-concat}, from
\Cref{lem:linear-evolution-disjoint-times} with $F(t, s) = e^{ i (t -
s)\Delta} \LP_{N} h (s)$, and from \eqref{eq:dgct} that
\[
\begin{aligned}[t]
\Big\| \int_{0}^{t} e^{i (t - s) \Delta} \LP_{N} h (s) \dd s  \Big\|_{X_{N} (I)}
\hspace{-3em}&\hspace{3em} \lesssim 
\sum_{n = 1}^{\infty} \Big( {\sum^{2^{n} - 1}_{k = 0}}  \Big\| \int_{I_{ k}^{n}} e^{i (t - s) \Delta} \LP_{N} h (s) \dd s \Big\|_{\tilde{X}_{N} (I_{ k + 1}^{n})}^{\frac{2}{1 - \varepsilon_{0}}} \Big)^{\frac{1 - \varepsilon_{0}}{2}}
\\
& \lesssim N^{\epsilon} \langle |I|^{-1}\rangle^{-c}  \sum_{n = 1}^{\infty} \Big( \sum^{2^{n} - 1}_{k = 0} \|h\|_{\tilde{X}_{N}^{*} (I_{k}^{n})}^{\frac{2}{1 - \varepsilon_{0}}} \Big)^{\frac{1 - \varepsilon_{0}}{2}} 
\\
& \lesssim N^{\epsilon}\langle |I|^{-1}\rangle^{-c} \sum_{n = 1}^{\infty} \Big( \sum^{2^{n} - 1}_{k = 0}( {2^{- \frac{1 + \varepsilon_{0}}{2} n}}  )^{\frac{2}{1 - \varepsilon_{0}}} \Big)^{\frac{1 - \varepsilon_{0}}{2}} \big\| \LP_{N} h \big\|_{\tilde{X}_{N}^{*} (I)}
\\
& \lesssim N^{\epsilon} \langle |I|^{-1}\rangle^{-c}  \sum_{n = 1}^{\infty}  \big( {2^{n}}  2^{- \frac{1 + \varepsilon_{0}}{1 - \varepsilon_{0}} n} \big)^{\frac{1 - \varepsilon_{0}}{2}} \big\| \LP_{N} h \big\|_{X_{N}^{*} (I)}
\\
& \lesssim N^{\epsilon} \langle |I|^{-1}\rangle^{-c} \big\| \LP_{N} h \big\|_{X_{N}^{*} (I)} 
\end{aligned}
\]
and the proof of \eqref{eq:non-homogeneous-bound} and \eqref{eq:non-homogeneous-bound-unit-scale} follows.

It remains to prove \eqref{eq:X-concat-sim}. In the following, we prove appropriate inequalities for every summand in the definition of $\tilde{X}_{N}$. Since the summands of $\tilde{Y}_{N}$ are the same, the claimed bound \eqref{eq:X-concat-sim} also follows for $\tilde{X}_{N}$ replaced by $\tilde{Y}_{N}$.

Let us prove a more general statement for functions $v_{1}$ and $v_{2}$ that have disjoint support in time.  Since $(a + b)^{s} \leq a^{s} + b^{s}$ for any $s \in [0, 1]$, we can use the triangle inequality to deduce that for any $1 \leq r \leq p \leq \infty$ we have
\[
\big\| v_{1} + v_{2} \|_{L_{t}^{p} L_{x}^{q} (I )}^{r}
\begin{aligned}[t]
& \leq  \Big\| \| v_{1} \|_{L_{x}^{q}} + \|v_{2}\|_{L_{x}^{q}} \Big\|_{L_{t}^{p}(I)}^{r}
\\
& = \Big(\int_{I} \big(  \| v_{1} (t) \|^{p}_{L^{q}_{x}} + \| v_{2} (t)\|^{p}_{L^{q}_{x}}  \big) \dd t\Big)^{\frac{r}{p}}
\\
& \leq \Big(\int_{I}   \| v_{1} (t) \|^{p}_{L^{q}_{x}} \dd t\Big)^{\frac{r}{p}}
+ \Big( \int_{I}\| v_{2} (t)\|^{p}_{L^{q}_{x}}   \dd t\Big)^{\frac{r}{p}}
\\
& =  \| v_{1} \|_{L_{t}^{p} L_{x}^{q} (I )}^{r} + \| v_{2} \|_{L_{t}^{p} L_{x}^{q} (I )}^{r} .
\end{aligned}
\]
For the directional norms, fix $\mf{c} \in [1, \infty)$ 
and assume $1 \leq r \leq \min \{ \mf{a},\mf{b}\}$. Let us show that for $v_{1}$ and $v_{2}$ with disjoint supports (in time) we have
\[\eqnum\label{eq:r-convexity}
\begin{aligned}
\|v_{1} + v_{2}\|_{L^{(\mf{a}, \mf{b}, \mf{c})}_{e_{1}}}^{r} \leq \|v_{1}\|_{L^{(\mf{a}, \mf{b}, \mf{c})}_{e_{1}}}^{r} + \|v_{2}\|_{L^{(\mf{a}, \mf{b}, \mf{c})}_{e_{1}}}^{r}.
\end{aligned}
\]
Indeed for any fixed $x_{1} \in \R$, that
\[\eqnum\label{eq:sisst}
 \begin{aligned}
 \hspace{3em} &  \hspace{-3em}\Big\| \| v_{1} (t,x_{1},x') + v_{2} (t,x_{1},x') \|_{L^{\mf{c}}_{x'}(\R^{d-1})} \Big\|^{\mf{b}}_{L^{\mf{b}}_{t} ( I )}
 \\
 & \leq \big\| \|v_{1} (t,x_{1},x')\|_{L^{\mf{c}}_{x'}(\R^{d-1})}
 + \|v_{2} (t,x_{1},x')  \|_{L^{\mf{c}}_{x'}(\R^{d-1})} \big\|^{\mf{b}}_{L^{\mf{b}}_{t} ( I )}
\\
& = \Big\| \|v_{1} (t,x_{1},x')\|_{L^{\mf{c}}_{x'}(\R^{d-1})} \Big\|^{\mf{b}}_{L^{\mf{b}}_{t} ( I )}
+ \Big\|\|v_{2} (t,x_{1},x') \|_{L^{\mf{c}}_{x'}(\R^{d-1})}  \Big\|^{\mf{b}}_{L^{\mf{b}}_{t} ( I )}.
\end{aligned}
\]

First, we assume that $\mf{a}\leq \mf{b}$ and in particular $r \leq \mf{a}$. Using that $(a + b)^{s} \leq a^{s} + b^{s}$ for any $s \in [0, 1]$ we obtain 
\[\begin{aligned}
\hspace{5em} &\hspace{-5em}
\Big\| \|v_{1} (t,x_{1},x') + v_{2} (t,x_{1},x')\|_{L^{\mf{c}}_{x'}(\R^{d-1})} \Big\|^{\mf{a}}_{L^{\mf{b}}_{t}( I )}
\\
& \leq \Big(\big\| \|v_{1} (t,x_{1},x')\|_{L^{\mf{c}}_{x'}(\R^{d-1})} \big\|^{\mf{b}}_{L^{\mf{b}}_{t} ( I )}
+ \big\|\|v_{2} (t,x_{1},x') \|_{L^{\mf{c}}_{x'}(\R^{d-1})} \big\|^{\mf{b}}_{L^{\mf{b}}_{t} ( I )}\Big)^{\frac{\mf{a}}{\mf{b}}}
\\
&  \leq \big\| \|v_{1} (t,x_{1},x')\|_{L^{\mf{c}}_{x'}(\R^{d-1})} \big\|^{\mf{a}}_{L^{\mf{b}}_{t} (  I )}
+ \big\|\|v_{2} (t,x_{1},x') \|_{L^{\mf{c}}_{x'}(\R^{d-1})} \big\|^{\mf{a}}_{L^{\mf{b}}_{t} ( I )} .
\end{aligned}
\]
Integrating in $x_{1}$ yields
\[
\|v_{1} + v_{2}\|_{L^{(\mf{a}, \mf{b}, \mf{c})}_{e_{1}}}^{\mf{a}}
\leq
\|v_{1}\|_{L^{(\mf{a}, \mf{b}, \mf{c})}_{e_{1}}}^{\mf{a}} + \|v_{2}\|_{L^{(\mf{a}, \mf{b}, \mf{c})}_{e_{1}}}^{\mf{a}}\,.
\]
After raising both sides to the power $\frac{r}{\mf{a}}$ and using again that $(a + b)^{s} \leq a^{s} + b^{s}$, for $s = \frac{r}{\mf{a}} \leq 1$, it follows
\[
\|v_{1} + v_{2}\|_{L^{(\mf{a}, \mf{b}, \mf{c})}_{e_{1}}}^{r} \leq
\|v_{1}\|_{L^{(\mf{a}, \mf{b}, \mf{c})}_{e_{1}}}^{r} + \|v_{2}\|_{L^{(\mf{a}, \mf{b}, \mf{c})}_{e_{1}}}^{r},
\]
as desired.

Second, we assume that $\mf{a} \geq \mf{b}$, and in particular that $r \leq \mf{b}$. Then, \eqref{eq:sisst} and the triangle inequality yield
\[
\| v_{1} + v_{2} \|_{L_{e_{1}}^{(\mf{a}, \mf{b}, \mf{c})} (I \times  \R^{d})}^{r}
\begin{aligned}[t]
& \leq \Big(\int_{\R}  \big( \big\|\| v_{1} (x_{1}) \|_{L^{\mf{c}}_{x'}}\|^{\mf{b}}_{L^{\mf{b}}_{t}}
+ \big\|\| v_{2} (x_{1})\|_{L^{\mf{c}}_{x'}}\big\|^{\mf{b}}_{L^{\mf{b}}_{t}} \big)^{\frac{\mf{a}}{\mf{b}}}  \dd x_{1}\Big)^{\frac{r}{\mf{a}}}
\\
& = \Big\| \big\| \| v_{1}  \|_{L^{\mf{c}}_{x'}} \big\|^{\mf{b}}_{L^{\mf{b}}_{t}}
+ \big\| \| v_{2}\|_{L^{\mf{c}}_{x'}}\big\|^{\mf{b}}_{L^{\mf{b}}_{t}} \Big\|_{L^{\frac{\mf{a}}{\mf{b}}}_{x_{1}}}^{\frac{r}{\mf{b}}}
\\
& \leq \Big(\Big\| \big\| \| v_{1}  \|_{L^{\mf{c}}_{x'}}\big\|^{\mf{b}}_{L^{\mf{b}}_{t}} \Big\|_{L^{\frac{\mf{a}}{\mf{b}}}_{x_{1}}}
+ \Big\| \big\| \|v_{2} \|_{L^{\mf{c}}_{x'}} \big\|^{\mf{b}}_{L^{\mf{b}}_{t}} \Big\|_{L^{\frac{\mf{a}}{\mf{b}}}_{x_{1}}} \Big)^{\frac{r}{\mf{b}}}
\\
& = \Big( \| v_{1} \|_{L_{e_{1}}^{(\mf{a}, \mf{b}, \mf{c})} (I \times\R^{d})}^{\mf{b}}
+ \| v_{2} \|_{L_{e_{1}}^{(\mf{a}, \mf{b}, \mf{c})} (I \times\R^{d})}^{\mf{b}}\Big)^{\frac{r}{\mf{b}}}
\\
& \leq \| v_{1} \|_{L_{e_{1}}^{(\mf{a}, \mf{b}, \mf{c})} (I \times\R^{d})}^{r} + \| v_{2} \|_{L_{e_{1}}^{(\mf{a}, \mf{b}, \mf{c})} (I \times\R^{d})}^{r},
\end{aligned}
\]
as desired.

Having shown that \eqref{eq:r-convexity} holds, the desired bound \eqref{eq:X-concat-sim} follows for $r = \frac{2}{1 - \varepsilon_{0}}$, which indeed satisfies $r \leq \min\{\mf{a}, \mf{b}\}$ 
for any directional norm in the definition of $X_N(I)$ or $Y_N(I)$. 
The proof of the desired results is finished. 
\end{proof}

\section{The multilinear estimates}\label{sec:multilinear-estimates}

We start this section by proving a bilinear estimates in terms of
directional space-time norms.  Then, we derive crucial trilinear
estimates that allow us to control the cubic nonlinearity in
\eqref{eq:NLS-deterministic}.

\begin{lemma}\label{lem:bilinear-2-2-bound}
Fix any \(0 < \epsilon \lesssim 1\) and \(0 < \epsilon_{0} \lesssim_{\epsilon} 1\). For any two functions
$h_{+}, h_{-} \st \R \times \R^{d} \rightarrow \C$ and any $N_{+}, N_{-} \in 2^{\N}$ with $N_{+} \geq N_{-}$
it holds that 
\[\eqnum\label{eq:bilinear-bound}
\begin{aligned}
\|h_{+} h_{-} \|_{L^{2}_{t} L^{2}_{x}  (\R \times \R^{d})} \lesssim_{\epsilon_{0}}N_{+}^{\epsilon} \Big(\frac{N_{+}}{N_{-}}\Big)^{-\frac{1}{2}}
\times
\begin{dcases}
N_{-}^{\xs_{c}}\|h_{+} \|_{X_{N_{+}} (\R) }  \|h_{-} \|_{X_{N_{-}} (\R) },
\\
\|h_{+} \|_{Y_{N_{+}} (\R) } \| h_{-} \|_{Y_{N_{-}} (\R) },
\\
\| h_{+} \|_{X_{N_{+}} (\R) } \| h_{-} \|_{Y_{N_{-}} (\R) },
\\
\| h_{+} \|_{Y_{N_{+}} (\R) } \| h_{-} \|_{X_{N_{-}} (\R) } \,,
\end{dcases}
\end{aligned}
\]
where \(\xs_{c}\eqd\frac{d-2}{2}\). 
\end{lemma}

\begin{proof}
Let $\UP_{e_{l}}$ be as in \Cref{lem:dir-local-smoothing}, and recall that 
\(\sum_{l} \UP_{e_{l}} = \Id\). The triangle inequality implies
\[
\|h_{+}\, h_{-} \|_{L^{2}_{t} L^{2}_{x}  (\R \times \R^{d})}\leq\sum_{l=1}^{d}\big\|(\UP_{e_{l}}h_{+} )\, h_{-} \big\|_{L^{2}_{t} L^{2}_{x}  (\R \times \R^{d})} \,, 
\]  
and therefore it suffices to prove \eqref{eq:bilinear-bound}
for $h_+$ replaced by $\UP_{e_{l}}h_{+}$. 

By Fubini's Theorem and Hölder's inequality we obtain that 
\[
\big\|(\UP_{e_{l}}h_{+} ) \, h_{-} \big\|_{L^{2}_{t} L^{2}_{x}  (\R \times \R^{d})}
\begin{aligned}[t]
& = \big\|(\UP_{e_{l}}h_{+} )\, h_{-} \big\|_{L_{e_{l}}^{(2,2,2)} (\R )} 
\\
& \lesssim \|\UP_{e_{l}}h_{+}\|_{L_{e_{l}}^{(\frac{2}{\epsilon_{0}}, \frac{2}{1 - \epsilon_{0}}, \frac{2}{1-\epsilon_{0}} )} (\R) }  \| h_{-} \|_{L_{e_{l}}^{( \frac{2}{1 - \epsilon_{0}},\frac{2}{\epsilon_{0}}, \frac{2}{\epsilon_{0}})} (\R) } \,.
\end{aligned}
\]
Then, the first three bounds in \eqref{eq:bilinear-bound} follow 
from the definitions \eqref{eq:def:X-norm} and \eqref{eq:def:Y-norm} of the norms \(X_{N}\) and \(Y_{N}\).

To prove the last estimate in \eqref{eq:bilinear-bound}, we use Hölder's inequality to obtain that
\[
\big\|(\UP_{e_{l}}h_{+} )\,h_{-} \big\|_{L_{e_{l}}^{(2,2,2)} (\R )}
\lesssim
\|\UP_{e_{l}}h_{+}\|_{L_{e_{1}}^{(\frac{2}{\epsilon_{0}}, \frac{2}{1 - \epsilon_{0}}, \frac{2\mf{c}_{0}}{\mf{c}_{0} - 2(1-\epsilon_{0})} )} (\R) }
\|h_{-} \|_{L_{e_{l}}^{( \frac{2}{1 - \epsilon_{0}},\frac{2}{\epsilon_{0}}, \frac{\mf{c}_{0}}{1-\epsilon_{0}})} (\R) }\,.
\]
Then from the definition \eqref{eq:def:X-norm} of the norm \(X_{N}\) it follows that 
\[
\|h_{-} \|_{L_{e_{l}}^{( \frac{2}{1 - \epsilon_{0}},\frac{2}{\epsilon_{0}}, \frac{\mf{c}_{0}}{1-\epsilon_{0}})} (\R) } 
\lesssim N_{-}^{\frac{1}{2}}\|h_{-} \|_{X_{N_{-}}(\R) } \,.
\]
In addition, since \(\mf{c}_{0}=2\frac{d-1}{d-2} \geq 2\), it holds that 
\(\frac{2}{1 - \epsilon_{0}} \leq \frac{2\mf{c}_{0}}{\mf{c}_{0} - 2(1-\epsilon_{0})} \leq \frac{2}{\epsilon_{0}} \)
for any \(0<\epsilon_{0}<2^{-100}\).  Thus, interpolation with an appropriately chosen parameter \(\theta\in[0,1]\)  and the definition \eqref{eq:def:Y-norm} of 
\(Y_{N}\), give that 
\[
\begin{aligned}[t]
& \|\UP_{e_{l}}h_{+}\|_{L_{e_{l}}^{(\frac{2}{\epsilon_{0}}, \frac{2}{1 - \epsilon_{0}}, \frac{2\mf{c}_{0}}{\mf{c}_{0} - 2(1-\epsilon_{0})} )} (\R)}     \hspace{-0.7em} \leq 
\|\UP_{e_{l}}h_{+}\|_{L_{e_{l}}^{(\frac{2}{\epsilon_{0}}, \frac{2}{1 - \epsilon_{0}}, \frac{2}{1-\epsilon_{0}} )} (\R) }^{1-\theta}
\|\UP_{e_{l}}h_{+}\|_{L_{e_{l}}^{(\frac{2}{\epsilon_{0}}, \frac{2}{1 - \epsilon_{0}}, \frac{2}{\epsilon_{0}} )} (\R) }^{\theta}    
\\
 &\qquad\lesssim  N_+^{\frac{1-\theta}{2}} \|\UP_{e_{l}}h_{+}\|_{Y_{N_+}(\R)}^{1-\theta}
N_+^{\frac{\theta}{2}} \|\UP_{e_{l}}h_{+}\|_{Y_{N_+}(\R)}^{\theta}
= N_{+}^{-\frac{1}{2}} \|\UP_{e_{l}}h_{+}\|_{Y_{N_{+}}(\R) } 
\end{aligned}
\]
and the proof is finished. 
\end{proof}

The above lemma is a crucial ingredient for showing that product of
three functions taken from the spaces \(Y^{S}\) or \(X^{\xs}\) lies in the
space \(X^{*,\sigma_{*}}\), which we use to control the non-homogeneity
\(h\) in \eqref{eq:non-homogeneous-schroedinger}. The specific
regularity of the product depends on whether the functions in play
belong to \(Y^{S}\) or to \(X^{\xs}\) and on the respective regularity
exponents \(S\) and \(\xs\). Since the three-fold product of functions is
estimated in the norm \(X^{*,\sigma_{*}}\), defined by duality in
\eqref{eq:def:X-norm-dual}, the proof of
\Cref{prop:trilinear-estimate} naturally reduces to showing bounds on
a \(4\)-linear integral form, formulated in \Cref{lem:fourlinear}. The
\(4\)-linear estimate \eqref{eq:4linear-form} of \Cref{lem:fourlinear}
in turn can be reduced using the Cauchy-Schwarz inequality to the
bilinear \(L^{2}\) estimate of \Cref{lem:bilinear-2-2-bound}.

\begin{proposition}\label{prop:trilinear-estimate}
Fix any 
\[\eqnum\label{eq:trilinear-exponent-assumptions}
    0 < S_1 \leq S_2 \leq S_3 < \xs_c < \xs 
    \,,
\]
and choose any \( 0 < \epsilon \lesssim_{S_{j},\xs}\!1\) and
\(0 < \epsilon_{0} \lesssim_{\epsilon, S_j,\xs} \!1\).  Then for any
$z_j$ and \(v_j\), $j\in\{1,2,3\}$, the following estimates hold:
\[\eqnum\label{eq:YYY}
\begin{aligned}
&\big\|z_{1}z_{2}z_{3}\big\|_{X^{*,\sigma_{*}}(\R)}\lesssim\|z_{1}\|_{Y^{S_{1}+\epsilon  }(\R)}\|z_{2}\|_{Y^{S_{2}+\epsilon}(\R)}\|z_{3}\|_{Y^{S_{3}+\epsilon}(\R)}
\\
&\qquad \text{with } \sigma_{*}\eqd S_{1}+ \min\Big(S_{2},\frac{1}{2}\Big)+\min\Big(S_{3},\frac{1}{2}\Big) \,.
\end{aligned}
\]
\[\eqnum\label{eq:YYX}
\begin{aligned}
& \big\|z_{1} z_{2}v_{3}\big\|_{X^{*,\sigma_{*}}(\R)} \lesssim\|z_{1}\|_{Y^{S_{1}+\epsilon  }(\R)}\|z_{2}\|_{Y^{S_{2}+\epsilon}(\R)}\|v_{3}\|_{X^{\xs+\epsilon}(\R)}
\\
& \qquad \text{with } \sigma_{*}\eqd S_{1}+\min\Big(S_{2},\frac{1}{2}\Big)+\frac{1}{2}\,.
\end{aligned}
\]
\[\eqnum\label{eq:YXX}
\begin{aligned}
& \big\|z_{1}v_{2}v_{3}\big\|_{X^{*,\sigma_{*}}(\R)} \lesssim
\|z_{1}\|_{Y^{S_1+\epsilon}(\R)}\|v_{2}\|_{X^{\xs+\epsilon}(\R)}\|v_{3}\|_{X^{\xs+\epsilon}(\R)}
\\
& \qquad\text{with } \sigma_{*}\eqd\min\big(S_1+\xs,\,S_1+1\big)\,.
\end{aligned}
\]
\[\eqnum\label{eq:XXX}
\begin{aligned}
& \big\|v_{1}v_{2}v_{3}\big\|_{X^{*,\sigma_{*}}(\R)} \lesssim \|v_{1}\|_{X^{\xs+\epsilon}(\R)}\|v_{2}\|_{X^{\xs+\epsilon}(\R)}\|v_{3}\|_{X^{\xs+\epsilon}(\R)}
\\ & \qquad\text{with } \sigma_{*}\eqd \xs+2(\xs-\xs_{c})\,.
\end{aligned}
\]
The implicit constants are allowed to depend on \(\epsilon\), \(\epsilon_{0}\), $\xs$, and $S_{j}$. 
\end{proposition}

We deduce  \Cref{prop:trilinear-estimate} by duality from  \Cref{lem:fourlinear} below, which is proved at the end of the section. 

Before proceeding, we introduce some notation. For $\sigma_{*}$ as in each
case of \Cref{prop:trilinear-estimate} (detailed below) we choose
\(v_* \in X^{-\sigma_*}(\R)\). Next, for any function
$h \in \{v_*, z_j, v_j\}$, $j = 1, 2, 3$, and $N \in 2^\N$ we define
quantities
\[\eqnum\label{eq:srnta}
Z_{N}( h) \eqd
\begin{cases}
\|v_*\|_{X_N(\R)} & \text{if } h = v_* \,,
\\
\|v_j\|_{X_N(\R)} & \text{if } h = v_j \,,
\\
\|z_j\|_{Y_N(\R)} & \text{if } h = z_j \,,
\end{cases}
\qquad 
\alpha(h) \eqd
\begin{cases}
-\sigma_{*} & \text{if } h = v_* \,,
\\
\xs & \text{if } h = v_j \,,
\\
S_j & \text{if } h = z_j \,,
\end{cases}    
\]
and we denote by $\ms{S}(h_1, \cdots, h_4)$ the set of permutations of a quadruple $(h_1, \cdots, h_4)$. 

\begin{lemma}\label{lem:fourlinear}
Fix exponents as in \eqref{eq:trilinear-exponent-assumptions}, and
choose any $0 < \epsilon \lesssim 1$ and
$0 < \epsilon_0 \lesssim_{\epsilon, S_j, \xs} 1$. Then for any
\(N_{j}\in2^{\N}\), \(j\in\{1,2,3,4\}\), we have that
\[\eqnum\label{eq:4-linear-orthogonality}
\int_{\R\times \R^{d}}\LP_{N_{1}}h_{1}\LP_{N_{2}}h_{2}\LP_{N_{3}}h_{3}\LP_{N_{4}}h_{4}\dd t\dd x = 0
\]
unless \(N_{j}\lesssim\sum_{j'\neq j} N_{j'}\) for all \(j\in\{1,2,3,4\}\). Furthermore, 
\[\eqnum\label{eq:4linear-form}
\Big|\int_{\R\times \R^{d}}\LP_{N_{1}}h_{1}\LP_{N_{2}}h_{2}\LP_{N_{3}}h_{3}\LP_{N_{4}}h_{4}\dd t\dd x\Big|\lesssim (N_{1} N_{2}N_{3}N_{4})^{\epsilon}
\prod_{j = 1}^4 N_j^{\alpha(h_j)} Z_{N_{j}}\big(\LP_{N_{j}} h_j\big)
\]
 provided that one of the following conditions is satisfied:

\begin{description}
\item[Case $zzz$] \((h_1, \cdots, h_4) \in \ms{S}(v_*, z_1, z_2, z_3) \) and \(\sigma_{*}\) as in \eqref{eq:YYY}.
\item[Case $zzv$] \( (h_1, \cdots, h_4) \in \ms{S}(v_*, z_1, z_2, v_3)\) and \(\sigma_{*}\) as in \eqref{eq:YYX}.
\item[Case $zvv$] \( (h_1, \cdots, h_4) \in  \ms{S}(v_*, z_1, v_2, v_3) \) and \(\sigma_{*}\) as in \eqref{eq:YXX}.
\item[Case $vvv$] \((h_1, \cdots, h_4) \in  \ms{S}(v_*, v_1, v_2, v_3) \) and \(\sigma_{*}\) as in \eqref{eq:XXX}.
\end{description}

\end{lemma}

\begin{proof}[Proof of \Cref{prop:trilinear-estimate}]
Choose any $h_j \in \{z_j, v_j\}$, $j = \{1, 2, 3\}$. The definition of $X^{*,\sigma_{*}}$ and $X_N$ in \eqref{eq:def:X-norm-dual},  \(\Id=\sum_{N\in2^{\N}}\LP_{N}\),  the triangle inequality, and $\Big(\sum |a_j|^2\Big)^{\sfrac{1}{2}} \leq \sum |a_j|$ imply that 
\[
\big\|h_{1}h_{2}h_{3}\big\|_{X^{*, \sigma_{*}}(\R)}
\begin{aligned}[t]
& = \Big(\sum_{N \in 2^\N } N^{2\sigma_{*}}\big\|\LP_{N}\big(h_{1}h_{2}h_{3}\big)\big\|_{X^{*}_{N}(\R)}^{2}\Big)^{\sfrac{1}{2}}
\\
& \leq \sum_{N,N_{1},N_{2},N_{3} \in 2^\N}N^{\sigma_{*}}
\big\|\LP_{N}\big(P_{N_{1}} h_{1}P_{N_{2}}h_{2}P_{N_{3}}h_{3}\big)\big\|_{X^{*}_{N}(\R)}
\\
& \leq \sum_{N,N_{1},N_{2},N_{3} \in 2^\N}N^{\sigma_{*}}\sup_{v_{*}}
\Big|\int_{\R\times \R^{d}}\LP_{N}v_{*} \LP_{N_{1}}h_{1}\LP_{N_{2}}h_{2}\LP_{N_{3}}h_{3}\dd t\dd x\Big| \,,
\end{aligned}
\]
where the upper bound in the last expression is taken over all \(v_{*}\in
C^{\infty}_{c}(\R\times\R^{d})\) with \(\|\LP_N v_{*}\|_{X_{N}}\leq 1\).  Note that the
assumptions of \Cref{lem:fourlinear} on \(S_j\), \(\xs\), \(\epsilon\), \(\epsilon_0\), and $\sigma_*$
coincide with those of \Cref{prop:trilinear-estimate}, and therefore \eqref{eq:4linear-form}
with $\epsilon$ replaced by $\frac{\epsilon}{8}$ yields
\[
\Big|\int_{\R\times \R^{d}}\LP_{N}v_{*} \LP_{N_{1}}h_{1}\LP_{N_{2}}h_{2}\LP_{N_{3}}h_{3}\dd t\dd x\Big| 
\begin{aligned}[t]
&\lesssim  (N_1N_2N_3N)^{\frac{\epsilon}{8}} \prod_{j = 1}^3 N_j^{ \alpha(h_j)} Z_{N_{j}}\big(\LP_{N_{j}} h_j\big)
\\
& \hspace{-2em} =  (N_1N_2N_3)^{-\frac{3\epsilon}{8}} N^{\frac{\epsilon}{8}} \prod_{j = 1}^3 N_j^{\alpha(h_j) + \epsilon/2 } Z_{N_{j}}\big(\LP_{N_{j}} h_j\big) \,.
\end{aligned}
\]
We may assume that $N \lesssim \max_j N_j$ since otherwise, by \eqref{eq:4-linear-orthogonality}, the left-hand side vanishes. We have that $(N_1N_2N_3)^{-\frac{3\epsilon}{8}} N^{\frac{\epsilon}{8}} \lesssim N^{-\frac{\epsilon}{4}}$, and consequently
\[
\begin{aligned}
& \sum_{N,N_{1},N_{2},N_{3} \in 2^\N}  (N_1N_2N_3)^{-\frac{3\epsilon}{8}} N^{\frac{\epsilon}{8}} \prod_{j = 1}^3 N_j^{\alpha(h_j) + \epsilon} Z_{N_{j}}\big(\LP_{N_{j}} h_j\big) 
\\
& \hspace{10em} \lesssim \sum_{N\in 2^\N} N^{-\frac{\epsilon}{4}}  \prod_{j = 1}^3 \Big(\sum_{N_{j} \in 2^\N}  N_j^{\alpha(h_j)  + \epsilon/2}  Z_{N_{j}}\big(\LP_{N_{j}} h_j\big)\Big)\,.
\end{aligned}
\]
Using the Cauchy-Schwarz inequality and the definitions of $Z_{N_{j}}(h_{j})$, $\alpha(h_j)$,  $X^{\xs}$, and $Y^S$ we have  
\[
\begin{aligned}
\sum_{N_{j} \in 2^N} N^{\alpha(h_{j}) + \epsilon /2} Z_{N_{j}}(\LP_{N_{j}} h_{j})
& \lesssim \Big(\sum_{N\in 2^\N} N_{j}^{-\varepsilon}\Big)^{\frac{1}{2}}\Big(\sum_{N \in 2^N} N_{j}^{2\alpha(h) + 2\epsilon} Z_{N_{j}}\big(\LP_{N_{j}} h_{j}\big)^2 \Big)^{\frac{1}{2}}
\\
& \lesssim_{\epsilon}  \begin{cases}
\|z_j\|_{Y^{S_j + \epsilon}(\R)} & \text{if } h = z_j \,,
\\
\|v_j\|_{X^{\xs+ \epsilon}(\R) } & \text{if } h = v_j \,.
\end{cases}
\end{aligned}
\]
\end{proof}

\begin{proof}[Proof of \Cref{lem:fourlinear}]
Note that the right-hand sides of \eqref{eq:4-linear-orthogonality} and \eqref{eq:4linear-form}, as well as the subsequent conditions are the same if we exchange \(P_{N_{j}}h_{j}\) with \(P_{N_{j'}}h_{j'}\) \(j,j'\in\{1,2,3,4\}\), and therefore we assume without loss of generality, that \(N_{1}\geq N_{2}\geq N_{3}\geq N_{4}\).

First, we prove an orthogonality observation
\eqref{eq:4-linear-orthogonality}. For a contradiction suppose that
\(N_{1}\lesssim N_{2}+N_{3}+N_{4}\) does not hold, and in particular, suppose
that \(N_{1}\geq 2^{5} N_{2}\). By \eqref{eq:LP-bump-geometry} we have
\(\spt \big(\Fourier(\LP_{N_{1}}h_{1})\big)\subset B_{N_{1}}\setminus B_{2^{-2}N_{1}} \), since
\(N_{1}\geq 2^{5}N_2 \geq 2^{5} > 1\). On the other hand, since \(\Fourier(h_{2}h_{3}h_{4})=\FT{h_{2}}*\FT{h_{3}}*\FT{h_{4}}\), it holds that 
\[
\spt \big( \Fourier(h_{2}h_{3}h_{4})\big)\subset \spt \big( \FT{h_{2}}\big)+\spt \big( \FT{h_{3}}\big)+\spt \big( \FT{h_{4}}\big),
\]
and since $N_2 \geq N_3 \geq N_4$, \(N_{1}\geq 2^{5}N_{2}\), and  \(\spt\Big(\Fourier\big(\LP_{N}h\big)\Big)\subset B_{N} \) by \eqref{eq:LP-bump-geometry}, we have 
\[\eqnum\label{eq:spest}
\spt \Big(\Fourier\big(\LP_{N_{2}}h_{2}\LP_{N_{3}}h_{3}\LP_{N_{4}}h_{4}\big) \Big)\subset B_{N_{2}}+B_{N_{3}}+B_{N_{4}}\subset B_{2^{3} N_{2}}\subset B_{2^{-2}N_{1}} \,,
\]
Thus \eqref{eq:4-linear-orthogonality} holds by Plancherel's identity.

We henceforth suppose that \(N_{1}\approx N_{2}\) and we proceed to prove bound \eqref{eq:4linear-form}.
To estimate the left-hand side of \eqref{eq:4linear-form} we apply the Cauchy-Schwarz inequality pairing one function with high frequency ($N_1$ or $N_2$) with a function with a lower frequency ($N_3$ or $N_4$). This gives the bound
\[
\begin{aligned}
&  \Big|\int_{\R\times \R^{d}}\LP_{N_{1}}h_{1}\LP_{N_{2}}h_{2}\LP_{N_{3}}h_{3}\LP_{N_{4}}h_{4}\dd t\dd x\Big| \\
 & \hspace{5em} \leq  \min\Big(
\begin{aligned}[t]
& \|\LP_{N_{1}} h_{1}\LP_{N_{4}}h_{4} \|_{L^{2}_{t} L^{2}_{x} ( \R \times \R^{d} )}
\|\LP_{N_{2}} h_{2}\LP_{N_{3}}h_{3}\|_{L^{2}_{t} L^{2}_{x} (\R  \times \R^{d}  )} ,
\\
& \qquad\|\LP_{N_{1}} h_{1}\LP_{N_{3}}h_{3} \|_{L^{2}_{t} L^{2}_{x} ( \R \times \R^{d} )}
\|\LP_{N_{2}} h_{2}\LP_{N_{3}}h_{4}\|_{L^{2}_{t} L^{2}_{x} (\R  \times \R^{d}  )} \Big).
\end{aligned}
\end{aligned}
\]
By definitions in \eqref{eq:srnta} and \Cref{lem:bilinear-2-2-bound} for any \(l<l'\in\{1,2,3,4\}\) it holds that 
\[
\begin{aligned}
\|\LP_{N_{l}} h_{l}\LP_{N_{l'}}h_{l'} \|_{L^{2}_{t} L^{2}_{x} ( \R \times \R^{d} )} &  \lesssim  N_l^{\epsilon -\sfrac{1}{2}-\alpha(h_{l})} N_{l'}^{\sfrac{1}{2}  - \alpha(h_{l'})+\beta(h_{l},h_{l'})} 
\\
& \hspace{5em}\times \prod_{j \in \{l, l'\}} N_j^{\alpha(h_j)}Z_{N_{j}}(\LP_{N_{j}} h_{j}) \,,\end{aligned}
\]
where
\[
\beta(h_{l},h_{l'}) \eqd
\begin{cases}
\xs_{c} & \text{if } h_l, h_{l'} \in \{v_*,v_1,v_{2},v_3\}\,,\\
0 & \text{otherwise}.
\end{cases}
\]
Hence, since $N_1 \approx N_2$, 
we obtain
\[
\begin{aligned}
&  \Big|\int_{\R\times \R^{d}}\LP_{N_{1}}h_{1}\LP_{N_{2}}h_{2}\LP_{N_{3}}h_{3}\LP_{N_{4}}h_{4}\dd t\dd x\Big|
\leq N_{1}^{2\epsilon-1-\alpha(h_{1})-\alpha(h_{2})}
N_{3}^{\sfrac{1}{2}-\alpha(h_{3})}N_{4}^{\sfrac{1}{2}-\alpha(h_{3})}
\\
& \hspace{5em}\times
\min\Big( N_{3}^{\beta(h_{2},h_{3})}N_{4}^{\beta(h_{1},h_{4})}, N_{3}^{\beta(h_{1},h_{3})}N_{4}^{\beta(h_{2},h_{4})}\Big)
 \prod_{j =1}^{4} N_j^{\alpha(h_j)}Z_{N_{j}}(\LP_{N_{j}} h_{j}) \,.
\end{aligned}
\]
By homogeneity of the required bound \eqref{eq:4linear-form} we can assume, without loss of generality, that  \(N_j^{\alpha(h_j)}Z_{N_{j}}(\LP_{N_{j}} h_{j}) = 1\), \(j\in\{1,2,3,4\}\). Thus, recalling that \(N_{1}\approx N_{2}\), we reduce \eqref{eq:4linear-form} to proving
\begin{multline}\eqnum\label{eq:sieww}
N_{3}^{\sfrac{1}{2}-\alpha(h_{3})-\epsilon}N_{4}^{\sfrac{1}{2}-\alpha(h_{4})-\epsilon}
\min\Big( N_{3}^{\beta(h_{2},h_{3})}N_{4}^{\beta(h_{1},h_{4})}, N_{3}^{\beta(h_{1},h_{3})}N_{4}^{\beta(h_{2},h_{4})}\Big)  
\\
\lesssim N_{1}^{1+\alpha(h_{1})+\alpha(h_{2})}
\end{multline}  
We first claim that \(\alpha(h_{1})+\alpha(h_{2})+1\geq0\). Indeed, \(\alpha(h)<0\) only if \(h=v_{*}\), in which case \(\alpha(h)=-\sigma_{*}\). Assuming, without loss of generality, that $h_1 = v_*$, in each of 
the cases zzz, zzv, or zvv,  we replace the minimum in \eqref{eq:YYY}, in \eqref{eq:YYX}, or in
\eqref{eq:YXX} by $\frac{1}{2}$ or by $S + 1$ to obtain that
\[
\sigma_{*}\leq  S_{1}+1 \leq  S_{j}+1\leq\xs+1, \qquad j\in\{1,2,3\};
\]
the claim follows. Finally, in the case vvv we have $h_2 = v$, and therefore $\alpha(h_{2}) = \xs$. Consequently, 
\[
\sigma_{*} =  \xs + 2(\xs - \xs_c) \leq  \xs+1,  
\]
and the claim follows.

Thus, since \(N_{1}\geq N_{3}\) it is sufficient to show \eqref{eq:sieww} for \(N_{1}=N_{3}\), that is,  to show
\[\eqnum\label{eq:ctble}
\begin{aligned}  
\hspace{10em}& \hspace{-10em}\min\Big( N_{3}^{\beta(h_{2},h_{3})}N_{4}^{\beta(h_{1},h_{4})}, N_{3}^{\beta(h_{1},h_{3})}N_{4}^{\beta(h_{2},h_{4})}\Big)
\\
& \lesssim  N_{3}^{\sfrac{1}{2}+\alpha(h_{1})+\alpha(h_{2})+\alpha(h_{3})}N_{4}^{-\sfrac{1}{2}+\alpha(h_{4})}
\end{aligned}
\]
After standard algebraic manipulations, \eqref{eq:ctble} follows if we show that for any $N_3 \geq N_4$ either 
\[
N_4^{\beta(h_{1},h_{4}) +\sfrac{1}{2}-\alpha(h_{4}) } \lesssim N_3^{\sfrac{1}{2}+\alpha(h_{1})+\alpha(h_{2})+\alpha(h_{3}) - \beta(h_{2},h_{3})}
\] 
or 
\[
N_4^{\beta(h_{2},h_{4}) +\sfrac{1}{2}-\alpha(h_{4}) } \lesssim N_3^{\sfrac{1}{2}+\alpha(h_{1})+\alpha(h_{2})+\alpha(h_{3}) - \beta(h_{1},h_{3})} \,.
\]
Since $N_3 \geq N_4$ are arbitrary, it suffices to 
show that either 
\[
   \max\{0, \beta(h_{1},h_{4}) +\sfrac{1}{2}-\alpha(h_{4}) \}  \leq \sfrac{1}{2}+\alpha(h_{1})+\alpha(h_{2})+\alpha(h_{3}) - \beta(h_{2},h_{3})
\] 
or 
\[
  \max\{0, \beta(h_{2},h_{4}) +\sfrac{1}{2}-\alpha(h_{4}) \}  \leq \sfrac{1}{2}+\alpha(h_{1})+\alpha(h_{2})+\alpha(h_{3}) - \beta(h_{1},h_{3})
\]
holds. The expressions above can be rewritten in more symmetric form as
\[ \eqnum\label{eq:nlcn1}
\beta(h_{2},h_{3}) +  \max\{\alpha(h_{4}) -  \sfrac{1}{2}, \beta(h_{1},h_{4}) \}  \leq \alpha(h_{1})+\alpha(h_{2})+\alpha(h_{3}) 
+\alpha(h_{4})
\]
or 
\[ \eqnum\label{eq:nlcn2}
  \beta(h_{1},h_{3}) +  \max\{\alpha(h_{4})  -\sfrac{1}{2}, \beta(h_{2},h_{4})\}  \leq \alpha(h_{1})+\alpha(h_{2})+\alpha(h_{3}) 
    +\alpha(h_{4})
\]
Next, we discuss each case separately. 

\noindent
\textbf{Case $zzz$}. Then, 
\[ 
\alpha(h_{1})+\alpha(h_{2})+\alpha(h_{3}) 
    +\alpha(h_{4}) = -\sigma_* + S_1 + S_2 + S_3    
\]
and $\beta(h_j, h_k) = 0$ for any $j \neq k$, and therefore \eqref{eq:nlcn1} is equivalent to 
\[ 
   \max\{\alpha(h_{4}) -  \sfrac{1}{2}, 0 \}  \leq 
   -\sigma_* + S_1 + S_2 + S_3 \,.
\]
Clearly the left hand side is largest if 
$\alpha(h_4) = S_3$, and consequently \eqref{eq:nlcn1} holds if 
\[ 
 \sigma_* \leq  S_1 + S_2 + S_3 - 
 \max\{S_3 -  \sfrac{1}{2}, 0 \} = 
 S_1 + S_2  + \min\{S_3, \sfrac{1}{2}\} \,,
\]
which holds by \eqref{eq:YYY}. 

\noindent
\textbf{Case $zzv$}. Then, 
\[ 
\alpha(h_{1})+\alpha(h_{2})+\alpha(h_{3}) 
    +\alpha(h_{4}) = -\sigma_* + S_1 + S_2 + \xs    
\]
and either $\beta(h_2, h_3) = 0$ or $\beta(h_1, h_3) = 0$. Suppose $\beta(h_2, h_3) = 0$, the case $\beta(h_1, h_3) = 0$
follows analogously by proving \eqref{eq:nlcn2} instead of \eqref{eq:nlcn1}. Then, \eqref{eq:nlcn1} 
is equivalent to 
\[ 
  \max\{\alpha(h_{4}) -  \sfrac{1}{2}, \beta(h_{1},h_{4}) \}  \leq -\sigma_* + S_1 + S_2 + \xs \,.
\] 
Since $\alpha(h_{4}) \leq \xs$ and $\beta(h_{1},h_{4}) \leq \xs_c$, then \eqref{eq:nlcn1} follows if we show that 
\[ 
  \max\{\xs -  \sfrac{1}{2}, \xs_c \}  \leq -\sigma_* + S_1 + S_2 + \xs \,,
\] 
which is equivalent to 
\[ 
  \sigma_*   \leq  S_1 + S_2 + \xs  + 
  \min\{\sfrac{1}{2}, \xs - \xs_c \} \,,
\] 
that follows from \eqref{eq:YYX}. 

\noindent
\textbf{Case $zvv$}. Then, 
\[ 
\alpha(h_{1})+\alpha(h_{2})+\alpha(h_{3}) 
    +\alpha(h_{4}) = -\sigma_* + S_1 + \xs + \xs 
    \,.
\]
We discuss three cases. If $h_4 = v$, then $z_1 \in \{h_1, h_2, h_3\}$ and consequently either 
$\beta(h_2, h_3) = 0$ or $\beta(h_1, h_3) = 0$. Suppose $\beta(h_2, h_3) = 0$, the case $\beta(h_1, h_3) = 0$
follows analogously by proving \eqref{eq:nlcn2} instead of \eqref{eq:nlcn1}. Then, \eqref{eq:nlcn1} 
is equivalent to 
\[ 
  \max\{\xs -  \sfrac{1}{2}, \beta(h_{1},h_{4}) \}  \leq -\sigma_* + S_1 + 2\xs \,.
\] 
Since $\beta(h_{1},h_{4}) \leq \xs_c$, then \eqref{eq:nlcn1} follows if we show that 
\[ 
  \max\{\xs -  \sfrac{1}{2}, \xs_c \}  \leq -\sigma_* + S_1 + 2\xs  \,,
\] 
which is equivalent to 
\[ 
  \sigma_*   \leq S_1 + \xs  + 
  \min\{ \sfrac{1}{2}, \xs -\xs_c\} \,,
\] 
that follows from \eqref{eq:YXX}.

If $h_4 = z_1$, then  $\beta(h_1, h_4) = 0$, 
$\beta(h_2, h_3) \leq \xs_c$, 
and $\alpha(h_{4}) = S_1$ implies that 
\eqref{eq:nlcn1} holds if 
\[ 
\xs_c + \max\{S_1 -  \sfrac{1}{2}, 0 \}  \leq -\sigma_* + S_1 + 2\xs \,.
\] 
After standard algebraic manipulations, 
the latter is equivalent to 
\[ 
\sigma_* \leq 2\xs - \xs_c + 
\min\{\sfrac{1}{2}, S_1\}  
\,,
\] 
which follows from \eqref{eq:YXX}. 

If $h_4 = v_*$, then $\alpha(h_{4}) -  \sfrac{1}{2} < 0$, and either $\beta(h_2, h_3) = 0$ or $\beta(h_1, h_3) = 0$. As above, we suppose $\beta(h_2, h_3) = 0$
and we note that \eqref{eq:nlcn1} follows if we show 
\[ 
   \xs_c \leq  -\sigma_* + S_1 + 2\xs, 
\]
or equivalently 
\[ 
    \sigma_* \leq  S_1 + 2\xs - \xs_c, 
\]
which follows from \eqref{eq:YXX}. 

\noindent
\textbf{Case $vvv$}. Then, 
\[ 
\alpha(h_{1})+\alpha(h_{2})+\alpha(h_{3}) 
    +\alpha(h_{4}) = -\sigma_* + 3\xs    
\]
and $\beta(h_j, h_k) \leq \xs_c$ for each $j \neq k$. Then, \eqref{eq:nlcn1} follows if we show
\[ 
\xs_c + \max\{\xs -  \sfrac{1}{2}, \xs_c \}  \leq -\sigma_* + 3\xs \,.
\] 
The latter is equivalent to 
\[ 
\sigma_*  \leq  2(\xs - \xs_c) + 
  \min\{\sfrac{1}{2} + \xs_c, \xs\}
  \,,
\] 
and \eqref{eq:nlcn1} follows from \eqref{eq:XXX}. 
\end{proof}

We conclude this section by a crucial estimate that justifies the regularity computation encoded by the definition \eqref{eq:def:multilinear-regularity} of \(\mu(k,S)\)

\begin{lemma}[Inductive property of $\mu(k,S)$]\label{lem:mu-inductive-property}
For any \(S>0\) and any \(k_{1},k_{2},k_{3}\in\N\setminus\{0\}\), with
\(k_{1}\leq k_{2}\leq k_{3}\), it holds that
\[\eqnum\label{eq:mu-inductive-property}
\mu(k_{1}+k_{2}+k_{3},S) \leq\mu(k_{1},S)+ \min\big(\mu(k_{2},S),\sfrac{1}{2}\big)+ \min\big(\mu(k_{3},S),\sfrac{1}{2}\big).
\]

Furthermore, setting \(k\eqd k_{1}+k_{2}+k_{3}\)
\(0 < \epsilon \lesssim_{k} 1\), and any
\(0 < \epsilon_{0} \lesssim_{k, S, \epsilon} 1\). Then for any functions
$z_j$, $j \in \{1, 2, 3\}$, it holds that
\[\eqnum\label{eq:trilinear-Y-mu-bound}
\|z_{1}z_{2}z_{3}\|_{X^{*,\mu(k,S)}}\lesssim \prod_{j = 1}^{3} \|z_{j}\|_{Y^{\mu(k_{j},S)+\epsilon}} \,,
\]
where the implicit constant depends on $k$, $\epsilon$, $\epsilon_0$, and $S$. 

\end{lemma}
\begin{proof}[Proof of \Cref{lem:mu-inductive-property}]
Without loss of generality, suppose \(k_{1}\leq k_{2}\leq k_{3}\). Let
\[
J\eqd\mu(k_{1},S)+ \min\big(\mu(k_{2},S),\sfrac{1}{2}\big)+ \min\big(\mu(k_{3},S),\sfrac{1}{2}\big).
\]
Trivially, we have $\mu(|k_{j}|,S) \geq S$, \(j\in\{1,2,3\}\). If $\mu(k_{3},S)\geq \mu(k_{2},S) \geq \frac{1}{2}$ then 
\[
 J \geq S + \frac{1}{2} + \frac{1}{2} = S + 1 \geq  \mu(k_{1}+k_{2}+k_{3},S)\,, 
\]
and \eqref{eq:mu-inductive-property} follows. If $\mu(k_{3},S) \geq \frac{1}{2} \geq \mu(k_{2},S)$, then
\[
J \geq S + S + \frac{1}{2} = 2S + \frac{1}{2} \geq  \mu(k_{1}+k_{2}+k_{3},S) \,,
\]
and \eqref{eq:mu-inductive-property} follows.  Finally, if $ \frac{1}{2} \geq \mu(k_{3},S)\geq \mu(k_{2},S)$ then
\[
 J = \mu(k_{1},S) + \mu(k_{2},S)  + \mu(k_{3},S)
\]
and in addition $\mu(k_{j},S) = k_{j}S$, \(j\in\{1,2,3\}\) since \(\min(S+1,2S+\frac{1}{2})>\frac{1}{2}\) and \(k_{1}\leq k_{2}\). Then we deduce that 
\[
 J =  \big(k_{1}+k_{2}+k_{3}\big)S\geq\mu(k_{1}+k_{2}+k_{3},S),
\]
and \eqref{eq:mu-inductive-property} follows.

Finally, bound \eqref{eq:trilinear-Y-mu-bound} follows from bound \eqref{eq:YYY} with \(S\) replaced by \(S+\epsilon\). This completes the proof.
\end{proof}

\section{The deterministic fixed point}\label{sec:deteministic-fixed-point}
The main goal of this section is to prove
\Cref{thm:intro-deterministic-controlled-well-posedness}.

Recall that solution \(u\) to \eqref{eq:NLS-deterministic} can be seen as fixed
point of the map \(u\mapsto \mbb{I}_{f}(|u|^{2}u)\) with
\(\mbb{I}_{f}\) given by \eqref{eq:duhamel}. Having assumed that \(u\) is a
function of the form given by
\eqref{eq:solution-decomposition--deterministic} for some \(M\in\N\), it
follows by direct substitution that \(u\) is a solution to
\eqref{eq:NLS-deterministic} if and only if the remainder term
\(u^{\#}_{M}\eqd u-z_{\leq M}\), with
\(z_{\leq M}\eqd\sum_{k\leq M}z_{k}\), is fixed point of the  map
\(v \mapsto \mc{J}_{z,M}(v)\eqd \mbb{I}_{0}\big(\Phi_{z_{\leq M}}[v]+[z,z,z]_{>M}\big)\), given by
\eqref{eq:remainder-iteration-map}.

The boundedness of the map \(\mbb{I}_{0}\) from \(X^{*,\xs_{*}}(\R)\) to
\(X^{*,\xs_{*}}(\R)\), for some \(\xs_{*}>\xs\), has already been
established in \Cref{prop:main-linear-estimate}. Thus, the proof of
\Cref{thm:intro-deterministic-controlled-well-posedness} relies on
estimates on \(\Phi_{z_{\leq M}}[v]\) and \(\big[z,z,z\big]_{>M}\) in the
\(X^{*,\xs_{*}}([0,T])\) norm for some \(\xs_{*}>\xs\) and for small enough intervals \(T>0\).

In \Cref{lem:F-bound} we establish bounds on the Lipschitz constant of
the nonlinear map
\(X^{\xs}(I)\ni v \mapsto \Phi_{z}[v]\in X^{*,\xs_*}\big(I\big)\)
for an interval \(I\subset\R\). This will allow us to prove
\Cref{thm:intro-deterministic-controlled-well-posedness} using the
uniqueness of a fixed point of \(v\mapsto \mc{I}(v)\) by using the Banach fixed
point theorem.

\begin{lemma}\label{lem:F-bound}
Fix \(0<S<\xs_{c}<\xs<\xs_{c}+\sfrac{1}{2}\). Choose
\(0<\epsilon\lesssim_{S,\xs}1\) and
\(0<\epsilon_{0}\lesssim_{\epsilon, S, \xs}1\). For any interval \(I\subseteq\R\), any 
\(z\in Y^{S}(I)\) and \(v_{1},v_{2}\in X^{\xs}(I)\) and any
\[ \eqnum\label{eq:sdoss}
\xs_{*} < \min\big( \xs+2(\xs-\xs_{c}),2S+\sfrac{1}{2},\,S+1, S + \xs\big)-\epsilon \,,
\]
it holds that 
\[\eqnum\label{eq:F-bound}
\big\|\Phi_{z}[v_{1}]\big\|_{X^{*,\xs_{*}}(I)}\lesssim\big(\|v_{1}\|_{X^{\xs}(I)}+\|z\|_{Y^{S}(I)}\big)^{2}\|v_{1}\|_{X^{\xs}(I)},
\]
\[\eqnum\label{eq:F-bound-Lip-1}
\Big\|\Phi_{z}[v_{1}]-\Phi_{z}[v_{2}]\Big\|_{X^{*, \xs_{*}}(I)} \lesssim \big(\|v_{1}\|_{X^{\xs}(I)}+\|v_{2}\|_{X^{\xs}(I)}+\|z\|_{Y^{S}(I)}\big)^{2}
\|v_{1}-v_{2}\|_{X^{\xs}(I)}\,,
\]
and
\[\eqnum\label{eq:F-bound-Lip-2}
\Big\|\Phi_{z}[v]-\Phi_{\tilde{z}}[v]\Big\|_{X^{*, \xs_{*}}(I)} \lesssim
\big(\|v\|_{X^{\xs}(I)}+\|z\|_{Y^{S}(I)}+\|\tilde{z}\|_{Y^{S}(I)}\big)^{2}
\|z-\tilde{z}\|_{Y^{S}(I)}.
\]

In particular, if \(\xs<\min\big(2S+\sfrac{1}{2},\,S+1\big)\), then for any
sufficiently small \(0< \epsilon \lesssim_{S, \xs} 1 \), \eqref{eq:sdoss} holds with
\(\xs_{*}=\xs+\epsilon\).

We remark that the implicit constants may depend on \(S\), \(\xs\), \(\xs_{*}\), \(\epsilon\), and \(\epsilon_{0}\) but not on \(I\), \(z\), \(v_{1}\) or \(v_{2}\).
\end{lemma}

\begin{proof} [Proof of \Cref{lem:F-bound}]
Without loss of generality, we assume that \(I=\R\). The general statement follows by replacing \(z\) with \(z \1_{I}\) and using that \(\|z\|_{Y^{S}(I)}=\|z\1_{I}\|_{Y^{S}(\R)}\). 

By \eqref{eq:remainder-cubic-nonlinearity} we have \(\Phi_{z}[0]=0\), and therefore \eqref{eq:F-bound} follows from \eqref{eq:F-bound-Lip-1} after setting \(v_{2}=0\). Thus we concentrate on proving \eqref{eq:F-bound-Lip-1}. Using the expression \eqref{eq:remainder-cubic-nonlinearity} for \(\Phi_{z}[v]\) it follows that 
\[
\Phi_{z}[v_{1}]-\Phi_{z}[v_{2}]=G_{1}\big[z,\bar{z},v_{1},\bar{v_{1}},v_{2},\bar{v_{2}}\big] (v_{1}-v_{2})+G_{2}\big[z,\bar{z},v_{1},\bar{v_{1}},v_{2},\bar{v_{2}}\big](\bar{v_{1}-v_{2}}) \,,
\]
where \(G_{j}[z,\bar{z},v_{1},\bar{v_{1}},v_{2},\bar{v_{2}}] \), \(j\in\{1,2\}\), are homogeneous polynomials of degree 2 in the variables \(z\), \(\bar{z}\), \(v_{1}\), \(\bar{v_{1}}\), \(v_{2}\), \(\bar{v_{2}}\). 
Assume $\varepsilon < \min\big(S,\xs-\xs_c\big)$ and observe that 
\[
\begin{aligned}
& \xs_{*}<\min\big(2S+\sfrac{1}{2},\,S+1\big)-\epsilon,
\\
& \xs_{*}<\min\big(S+\xs,\,S+1\big)-\epsilon,
\\
& \xs_{*}<\xs+2(\xs-\xs_{c})-\epsilon.
\end{aligned}
\]
Then, \eqref{eq:YYX}, \eqref{eq:YXX}, and \eqref{eq:XXX}  with \(\xs-\epsilon/3\) and $S_j - \epsilon/3$ replacing, respectively,  \(\xs\) and \(S_j\), $j = 1, 2, 3$ imply
\[
\begin{aligned}
\Big\|\Phi_{z}[v_{1}]-\Phi_{z}[v_{2}]\Big\|_{X^{*,\xs_*}(\R)}
& \leq 
\begin{aligned}[t]
& \Big\|G_{1}\big[z,\bar{z},v_{1},\bar{v_{1}},v_{2},\bar{v_{2}}\big] \big(v_{1}-v_{2}\big)\Big\|_{X^{*,\xs_*}(\R)}
\\
& \qquad +\Big\|G_{2}\big[z,\bar{z},v_{1},\bar{v_{1}},v_{2},\bar{v_{2}}\big] \big(\bar{v_{1}-v_{2}}\big)\Big\|_{X^{*,\xs_*}(\R)}.
\end{aligned}
\\
& \lesssim\big(\|v_{1}\|_{X^{\xs}(I)}+\|v_{1}\|_{X^{\xs}(I)}+\|z\|_{Y^{S}(I)}\big)^{2}\|v_{1}-v_{2}\|_{X^{\xs}(I)}\,,
\end{aligned}
\]
as desired.

Bound \eqref{eq:F-bound-Lip-2} follows analogously by representing
\[
\Phi_{z}[v]-\Phi_{\tilde{z}}[v]
=
\Psi_{1}\big[z,\bar{z},\tilde{z},\bar{\tilde{z}},v,\bar{v}\big] (z-\tilde{z})
+\Psi_{2}\big[z,\bar{z},\tilde{z},\bar{\tilde{z}},v,\bar{v}\big] \bar{(z-\tilde{z})}
\]
with \(\Psi_{1}\big[z,\bar{z},\tilde{z},\bar{\tilde{z}},v,\bar{v}\big]\), \(j\in\{1,2\}\), are homogeneous polynomials of degree  2 in the input variables.
\end{proof}

Next, for any \(t_{0}\geq0\) we define a generalization of
\(v\mapsto\mbb{I}_{v_0}\big(\Phi_{z}[v]+h\big)\) from \eqref{eq:duhamel} as
\[\eqnum\label{eq:def:delayed-iteration-map}
\mc{K}(v) := e^{i(t-t_{0})\Delta}v_{0}\mp i\int_{t_{0}}^{t}e^{i(t-s)\Delta}(\Phi_{z}[v] +h)\dd s \,.
\]
This generalization to \(t_{0}>0\) is important for the proof of the uniqueness of the solution.

In the following lemma we show that \(v\mapsto\mc{K}(v)\) is a contraction on \(X^{\xs}([t_{0},t_{1}))\) if the time interval $[t_{0},t_{1})$ is short enough, or if \(v_{0}\), $z$, and $h$ are small in appropriate norms.

\begin{lemma}\label{lem:iteration-map-bound}
Let \(I\) with \(I\ni t_{0}\) be a time interval. Fix
\(0<S<\xs_{c}<\xs<\xs_{c}+\sfrac{1}{2}\). Choose
\(0<\epsilon\lesssim_{S,\xs}1\) and
\(0<\epsilon_{0}\lesssim_{\epsilon, S, \xs}1\). Then there exist constants
\(C=C(\epsilon,\epsilon_{0},S,\xs)>0\) and
\(c=c(\epsilon,\epsilon_{0},S,\xs)>0\) such that for any
\(z\in Y^{S}(I)\), \(h\in X^{\xs + \epsilon}(I)\), and
\(v_{1},v_{2}\in X^{\xs}(I)\) the following bounds hold:
\[\eqnum\label{eq:iteration-map-bound}
\Big\|\mc{K}(v_1)\Big\|_{X^{\xs}(I)}  \leq C \big\langle |I|^{-1}\big\rangle^{-c} 
\begin{aligned}[t]
\Big(& \|v_{0}\|_{H_x^{\xs+\epsilon} (\R^{d})} + \|h\|_{X^{*,\xs+\epsilon}(I)}
\\
& \quad+ \big(\|v_{1}\|_{X^{\xs}(I)}+\|z\|_{Y^{S}(I)}\big)^{2}\|v_{1}\|_{X^{\xs}(I)} \Big)
\end{aligned}
\]
and 
\[\eqnum\label{eq:iteration-map-lipschitz}
\begin{aligned}
& \Big\|\mc{K}(v_1) - \mc{K}(v_2)\Big\|_{X^{\xs}(I)}
\\
&\qquad \leq C \big\langle |I|^{-1}\big\rangle^{-c} \big(\|v_{1}\|_{X^{\xs}(I)}+\|v_{2}\|_{X^{\xs}(I)}+\|z\|_{Y^{S}(I)}\big)^{2}\big\|v_{1}-v_{2}\big\|_{X^{\xs}(I)}\,.
\end{aligned}
\]

As a consequence:
\begin{itemize}
\item For any \(v_{0}\in H^{\xs+\epsilon}(\R^{d})\),
\(z\in Y^{S}(I)\), \(h\in X^{*,\xs+\epsilon}(I)\), and $\delta_0 > 0$ we define
\[\eqnum\label{eq:contraction-time-bound}
\begin{aligned}
& T_{\delta_{0}}\big(\|v_{0}\|_{H^{\xs+\epsilon}_{x}(\R^{d})},\|z\|_{Y^{S}(I)},\|h\|_{X^{*,\xs+\epsilon}(I)}\big)
\\
& \quad\eqd\frac{1}{2^{100}(C\delta_{0})^{\sfrac{1}{c}}}\big\langle\delta_0+\|v_{0}\|_{H^{\xs+\epsilon}_{x}(\R^{d})} + \|h\|_{X^{*,\xs+\epsilon}(\R)} +\|z\|_{Y^{S}(\R)}\big\rangle^{-\sfrac{3}{c}} \,.
\end{aligned}
\]
Then for any interval \(J=[t_{0},t_{1}]\subseteq I\) with
\[
|J|<T_{\delta_{0}}\big(\|v_{0}\|_{H^{\xs+\epsilon}_{x}(\R^{d})},\|z\|_{Y^{S}(I)},\|h\|_{X^{*,\xs+\epsilon}(I)}\big)
\]
the map $v \mapsto \mc{K}(v)$ is a contraction on the set
\[
\bar{\mathcal{B}}_{\delta_{0}}^{J}\eqd\Big\{v\in X^{\xs}(J)\st \|v\|_{ X^{\xs}(J)}\leq\delta_{0}\Big\}\,,
\]
and thus it admits a unique fixed point in \(\bar{\mathcal{B}}_{\delta_{0}}^{J}\).
\item If \(I=\R\), there exists \(\delta_{0}>0\) such that if
\[
\|z\|_{Y^{S}(\R)}\leq\frac{\delta_{0}}{2C}, \qquad \|v_{0}\|_{H_x^{\xs+\epsilon}(\R^{d})}\leq\frac{\delta_{0}}{2C},\qquad \|h\|_{X^{*,\xs+\epsilon}(\R)}\leq\frac{\delta_{0}}{2C},
\]
then $v \mapsto \mc{K}(v)$ is a contraction on the set
\[
\bar{\mathcal{B}}_{\delta_{0}}^{\R}\eqd\Big\{v\in X^{\xs}(\R)\st \|v\|_{ X^{\xs}(\R)}\leq\delta_{0}\Big\}\,,
\]
and thus admits a unique fixed point in \(\bar{\mathcal{B}}_{\delta_{0}}^{\R}\).
\end{itemize}
\end{lemma}

\begin{proof}[Proof of \Cref{lem:iteration-map-bound}]
Assume that \(\epsilon>0\) is chosen small, and in particular that it satisfies
\[
\varepsilon < \frac{1}{2}\min \big(S, 2(\xs-\xs_c), 2S+\sfrac{1}{2} -\xs, S+1-\xs \big).
\]
Fix \(z\in Y^{S}(I)\), \(h\in X^{\xs + \epsilon}(I)\), and \(v_{1},v_{2}\in X^{\xs}(I)\). Then, by \eqref{eq:non-homogeneous-bound-2}
\[
\begin{aligned}[t]
\Big\|\mathcal{K}(v_1)\Big\|_{X^{\xs}(I)}
\leq C \big\langle |I|^{-1}\big\rangle^{-c} \Big(\|v_{0}\|_{H^{\xs+\epsilon} (\R^{d})} + \|h\|_{X^{*,\xs+\epsilon}(I)}+ \big\|\Phi_{z}[v_{1}]\big\|_{X^{*,\xs+\epsilon}(I)}\Big)\,.
\end{aligned}
\]
Since \(\xs+\epsilon<\min\big(\xs+2(\xs-\xs_{c}),\,2S+\sfrac{1}{2},S+1,S +\xs\big)-\epsilon
\) for \(\epsilon>0\) small enough, \eqref{eq:iteration-map-bound} follows from 
\eqref{eq:F-bound} with $s_* = \xs + \epsilon$. Next, \eqref{eq:non-homogeneous-bound-2} implies
\[
\Big\|\mc{K}(v_1) - \mc{K}(v_2)\Big\|_{X^{\xs}(I)}
\leq C  \big\langle |I|^{-1}\big\rangle^{-c} \big\|\Phi_{z}[v_{1}]-\Phi_{z}[v_{2}]\big\|_{X^{*,\xs+\epsilon}(I)}.
\]
and \eqref{eq:F-bound-Lip-1} with $s_* = \xs + \epsilon$ yields \eqref{eq:iteration-map-lipschitz}.

Assuming, without loss of generality, that \(C>1>c>0\), direct computation allows us to deduce from \eqref{eq:iteration-map-bound} that \(\mc{K}\) maps the sets \(\bar{\mathcal{B}}_{\delta_{0}}^{J}\) and \(\bar{\mathcal{B}}_{\delta_{0}}^{\R}\) to themselves under the corresponding assumptions on \(\delta_{0}\), \(z\), and \(h\). Similarly, bound \eqref{eq:iteration-map-lipschitz} shows that \(\mc{K}\) is a contraction on these sets. The existence and uniqueness of the fixed point follows from the Banach contraction mapping principle.
\end{proof}

Finally, let us record the bound on \([z,z,z]_{>M}\) appearing in the
definition of the iteration map
\(\mc{J}_{z,M}(v) \eqd \mbb{I}_{0}\big(\Phi_{z_{\leq M}}[v]+[z,z,z]_{>M}\big)\).

\begin{lemma}\label{lem:z-remainder-term-bound}
Fix \(0<S<\xs_{c}\), \(0<\epsilon\lesssim_{S,\xs}1\), and
\(0<\epsilon_{0}\lesssim_{\epsilon, S, \xs}1\). Then it holds that 
\[\eqnum\label{eq:rz3M-Xs-bound}
\big\| [z,z,z]_{> M} \big\|_{ X^{*,\mu(M+1,S)} (\R)} \lesssim \Big(\max_{k\leq M}\|z_{k}\|_{Y^{\mu(k,S)+\epsilon}(\R)}\Big)^{3}
\]
and the map
\[
\vec{z}_{M}=(z_{k})_{k\leq M} \mapsto [z,z,z]_{> M}
\]
is Lipschitz on bounded sets with a Lipschitz constant bounded by
\[C\Big(\max_{k\leq M}\|z_{k}\|_{Y^{\mu(k,S)+\epsilon}(\R)}\Big)^{2}
\] for some
\(C>0\) with respect to distance introduced in \eqref{eq:initial-data-distance}.
\end{lemma}

\begin{proof}[Proof of \Cref{lem:z-remainder-term-bound}]
To prove \eqref{eq:rz3M-Xs-bound} we note that, according to \eqref{eq:def:z-sum} it holds that 
\[\eqnum\label{eq:z-remainder-term}
[z,z,z]_{>M}=\sum_{\mcl{\qquad \substack{\hspace{0.5em} k_{1}+k_{2}+k_{3}\geq M+1\\\hspace{-1em} k_{j}\leq M}}} z_{k_{1}}\bar{z_{k_{2}}} z_{k_{3}}\,,
\]
Since the sum above is finite, we estimate each summand separately in the space $X^{*, \mu(M+1,S)} (\R)$. Without loss of generality, assume that \(k_{1}\leq k_{2}\leq k_{3}\) and use \eqref{eq:trilinear-Y-mu-bound} to obtain the desired claim.

To prove Lipschitz regularity it is sufficient to note that
\[
[z,z,z]_{>M}-[\tilde{z},\tilde{z},\tilde{z}]_{>M}
=
\sum_{\mcl{\qquad \substack{\hspace{0.5em} k_{1}+k_{2}+k_{3}\geq M+1\\\hspace{-1em} k_{j}\leq M}}}
\Big(
(z_{k_{1}}-\tilde{z}_{k_{1}})\bar{z_{k_{2}}} z_{k_{3}}
+
\begin{aligned}[t]
&\tilde{z}_{k_{1}}\bar{(z_{k_{2}}-\tilde{z}_{k_{2}})} z_{k_{3}}
\\ 
&+
\tilde{z}_{k_{1}}\bar{\tilde{z}_{k_{2}}} (z_{k_{3}}-\tilde{z}_{k_{3}})
\Big)\,,
\end{aligned}
\]
and apply the same reasoning as above. 
\end{proof}

We are now ready to prove \Cref{thm:intro-deterministic-controlled-well-posedness}.

\begin{proof}[Proof of \Cref{thm:intro-deterministic-controlled-well-posedness}]
We suppose that \(\epsilon>0\) is small enough, and in particular it satisfies
\[ \eqnum\label{eq:usone}
\varepsilon < \frac{1}{3}\min \Big(S, 2(\xs-\xs_c), 2S+\frac{1}{2} -\xs , S+1-\xs, \mu(M+1,S) - \xs \Big).
\]
Let $C, c$ be constants from \Cref{lem:iteration-map-bound}. 

According to the discussion at the beginning of this section, \(u\) is a
solution to \eqref{eq:NLS-deterministic} on \([0,T)\) of the form
\eqref{eq:solution-decomposition}
if and only if
\(u^{\#}_{M}\) is a fixed point of the map \(v \mapsto
\mc{J}_{z,M}(v)\)
given by \eqref{eq:remainder-iteration-map}. 

\textbf{Local existence of solutions.} 
For $\vec{z}_M$ satisfying \eqref{eq:asoaz}, set \(v_{0}=0\), \(z=z_{\leq M}\),
\(h=[z,z,z]_{>M}\), and \(t_{0}=0\) and define $T := T_{\delta_0 = 1} > 0$,
where $T_{\delta_0}$ is as in \eqref{eq:contraction-time-bound}. Also, for
such a choice of $v_0$, $z$, $h$, and $t_0$, we have \(\mc{J}_{z,M}(v)\eqd\mc{K}(v)
\), with $\mc{K}$ defined in \eqref{eq:def:delayed-iteration-map}. Then 
the existence of a fixed point in \(X^{\xs}([0,T))\) follows from \Cref{lem:iteration-map-bound}, while the estimate \eqref{eq:time-of-existence} follows from 
\Cref{lem:iteration-map-bound},  \eqref{eq:rz3M-Xs-bound} and the bound
\begin{equation}\label{eq:stvonz}
\|z_{\leq M}\|_{Y^{S}(I)}\leq M \max_{k\in\{1,\ldots,M\}}\|z_{k}\|_{Y^{\mu(k,S)}(I)}.    
\end{equation}
We establish the bound on \(\|v\|_{C^{0}\big([0,T);H^{\xs}_{x}(\R^{d})\big)}\) once we prove \eqref{eq:v-classical-space} below.

\textbf{Global existence of solutions.}
The existence of fixed point of $\mc{J}_{z,M}$ in \(X^{\xs}([0,+\infty))\) again follows from
\Cref{lem:iteration-map-bound}. Indeed, \Cref{lem:iteration-map-bound},  \eqref{eq:rz3M-Xs-bound}, and \eqref{eq:stvonz}
show that
\(\|z_{\leq M}\|_{Y^{S}(I)}\) and \(\big\| [z,z,z]_{> M} \big\|_{ X^{*,\mu(M+1,S)}
  (\R)}\) can be made arbitrarily small as long as \(\delta_{0}\) in
\eqref{eq:multilinear-datum-smallness} is chosen small enough.

\textbf{Time-continuity of the solution} We show that if $v \in X^{\xs} (I)$ for some interval \(I\subset[0,T_{0})\), then \(\mc{J}_{z,M}(v) \in C^{0}\big(I,H^{\xs+\epsilon}_{x}(\R^{d})\big)\). 
Indeed, by \Cref{prop:linear-nonhomogeneous-continuity-bound}, \eqref{eq:rz3M-Xs-bound}, and \eqref{eq:F-bound} we
have
\[
\begin{aligned}
\|\mc{J}_{z,M}(v)\|_{C^0(I, H_{x}^{\xs +\varepsilon} (\R^d))}
&\lesssim 
 \big\|\Phi_{z_{\leq M}}[v]\big\|_{X^{\ast ,\xs +2\varepsilon} (I)}
 +\big\|[z,z,z]_{> M}\big\|_{X^{\ast,\xs +2\varepsilon}(I)} \,.
\\
&\lesssim 
 \big(1 + \|v\|_{X^{\xs} (I)}+M \max_{k\in\{1,\ldots,M\}}\|z_{k}\|_{Y^{\mu(k,S)+\epsilon}(I)}\big)^3,
 \end{aligned}
 \]
where we used that the assumptions of \Cref{lem:F-bound} are satisfied
since \eqref{eq:usone} implies that
\[
\xs + 2\varepsilon <\min\big(S + \xs, \xs+2(\xs-\xs_{c}),2S+\frac{1}{2},\,S+1\big)-\epsilon \,.
\]
This concludes the proof of our claim. 

\textbf{Uniqueness of the solution.} Next, we show the uniqueness of the fixed point without any a-priori assumption on the smallness of its \(X^{\xs}\) norm (cf. \Cref{lem:iteration-map-bound}).  
Let \(v_{j}\in X^{\xs}([0,T_{j}])\), \(j\in\{1,2\}\) be any 
two fixed points of \(\mc{J}_{z,M}\) and, without loss of generality, assume that $0 < T_1 \leq T_2$.  Define
\[\eqnum\label{eq:uniqueness-contradiction-time}
t_{0}\eqd\sup\big(t\in[0,T_{1}) \st v_{1}(s)= v_{2}(s) \text{ for all } s\in[0,t] \big)
\]
and for a contradiction, assume that \(t_{0}<T_{1}\). Note that $t_{0}$ is well defined since $v_1(0) = v_2(0)$, and therefore the supremum is taken over non-empty set. Since $v_{1},v_{2}\in C_t^{0}\big([0,T_{1}), H_x^{\xs+\epsilon}(\R^{d})\big)$, thanks to \eqref{eq:v-classical-space}, then \(v_{1}(t_{0})=v_{1}(t_{0})=:v_0\). Then, from the group property of the Schrödinger evolution for all \(t \in [t_{0}, T_{1}]\) we have that
\[
v_{j}(t)=e^{i(t-t_{0})\Delta}v_{0}\mp i\int_{t_{0}}^{t}e^{i(t-s)\Delta}(\Phi_{z_{\leq M}}[v_{j}] + [z,z,z]_{> M}) \, ds \qquad j = 1, 2\,,
\]
that is 
\(v_{j}\), $j = 1, 2$ are both fixed points of the map \(v\mapsto \mc{K}(v)\)  with \(z = z_{\leq M}\) and \(h = [z,z,z]_{> M}\).

By \eqref{eq:rz3M-Xs-bound} and \eqref{eq:F-bound}, there is \(R>0\) such that 
\[\eqnum\label{eq:ball-bounds-on-data-K}
\|z_{\leq M}\|_{Y^{S}([0,T_{1}))}\leq R,\quad \|[z,z,z]_{> M}\|_{X^{\xs}([0,T_{1}))}\leq R,\quad \|v_{0}\|_{H^{\xs+\epsilon}}\leq R.
\]
Estimate \eqref{eq:iteration-map-bound}  shows that there exists \(T\) with \(t_{0}<t_{0}+T<T_{1}\) for which
\[\|v_{1}\|_{X^{\xs}([t_{0},t_{0}+T])}\leq 1
\qquad \text{and } \quad \|v_{2}\|_{X^{\xs}([t_{0},T])}\leq 1
\] 
and
the map \(v\mapsto\mc{K}(v)\) has a unique fixed point on
\(\bar{\mc{B}}_{1}^{[t_{0},t_{0}+T]}\big\{v\st\|v\|_{X^{\xs}([t_{0},t_{0}+T])}\leq 1 \big\}\), a contradiction to the definition of $t_0$. Hence, the desired uniqueness is proved. 

\textbf{Blow-up criterion.} By the uniqueness property we showed above, the upper bound \(T_{\mrm{max}}\) of the existence and uniqueness time of solutions is well defined.
To prove \eqref{eq:blowup}, suppose that \(T_{\mrm{max}}<T_{0}\), and,  for a contradiction, fix a sufficiently large $R$ sufficiently large so that 
\begin{equation} \label{eq:Hs-blowup-contradiction}
\limsup_{t \to T_{\mrm{max}}} \|u^{\#}_{M}(t)\|_{H^{\xs}(\R^d)} \leq R < \infty\,,
\end{equation}
and such that \eqref{eq:ball-bounds-on-data-K} holds with $v_0 = 0$. 
Let $T =
T_{\delta_0 = 1}(R, R, R)$ be as in \eqref{eq:contraction-time-bound} and
fix $t_0 < T_{\mrm{max}}$ such that $t_0 + T > T_{\mrm{max}}$.  Then
\Cref{lem:iteration-map-bound} guarantees that it is possible to
extend the solution to at least \(\big[0,\min(t_{0}+T,T_{0}) \big)\) by
finding a fixed point to the map \(v\mapsto \mc{K}(v)\) on\(\big[t_{0},\min(t_{0}+T,T_{0})\big]\). In addition,
\[
\|u^{\#}_{M}\|_{X^{\xs}\big([0,\min(t_{0}+T,T_{0}) )\big)}
+\|u^{\#}_{M}\|_{C^{0}\big([0,\min(t_{0}+T,T_{0}) );H^{\xs}_{x}(\R^{d})\big)}
<\infty \,,
\]
contradicting that \(T_{\max}\) is the maximal existence time for the
solution. The second equality in \eqref{eq:blowup} then follows from
\eqref{eq:v-classical-space}.

\textbf{Time-continuity and scattering of the multilinear data}
Since \(z_{1}=e^{it\Delta}f\), \eqref{eq:sobolev-isometry-of-linear-evolution} follows from  \eqref{eq:continuity-bound} applied with \(v_{0}=f\) and \(h=0\).

Next,  we focus on \(z_{k}\) with \(2\leq k \leq M\) and show that \eqref{eq:time-continuity-sca} holds.  According to \eqref{eq:def:zk-deterministic}, it holds that that
\begin{equation}\label{eq:amlne}
z_{k}=\mp i \sum_{\mcl{\substack{ k_{1}+k_{2}+k_{3}=k\\ k_{1},k_{2},k_{3}\leq k }} } \mbb{I}_{0}(z_{k_{1}} \bar{z_{k_{2}}} z_{k_{3}})\,,
\end{equation}
Using \eqref{eq:continuity-bound}, $\mu(k,S)+\epsilon \leq \mu(k,S+\epsilon)
\leq \mu(k,S) + (M + 1)\epsilon$,  and \eqref{eq:trilinear-Y-mu-bound} with 
\(S+\epsilon\) instead of \(S\) we obtain that
\begin{equation}\label{eq:abfsc}
\begin{aligned}
\|z_{k}\|_{C^{0}\big([0,T_{0});H^{\mu(k,S)}_{x}(\R^{d})\big)} &\lesssim
\quad 
\sum_{\mcl{\substack{ k_{1}+k_{2}+k_{3}=k\\ k_{1},k_{2},k_{3}\leq k }} }
\Big\|z_{k_{1}} \bar{z_{k_{2}}} z_{k_{3}}\Big\|_{X^{*,\mu(k,S)+\epsilon}([0,T_{0}))}
\\
&\lesssim \prod_{j = 1}^{3} \|z_{k_{j}}\|_{Y^{\mu(k_{j},S)+(M + 1)\epsilon}([0,T_{0}))}\,.
\end{aligned}
\end{equation}
So \eqref{eq:time-continuity-sca} follows if we replace \(\epsilon\) with \(\sfrac{\epsilon}{(M+1)}\).

Clearly, $z_1 = e^{it\Delta} f$ scatters with $w_{1} = f$, i.e. \eqref{eq:zk-scattering} holds.  To show that \(z_{k}\), \(2\leq k\leq M\) scatters, we 
first observe that by \eqref{eq:abfsc} and \eqref{eq:limit-scattering-existence} the limit
\[
w_{k}\eqd \lim_{t\to+\infty} e^{-it\Delta} z_{k}=\mp i \sum_{\mcl{\substack{ k_{1}+k_{2}+k_{3}=k\\ k_{1},k_{2},k_{3}< k }} }
\lim_{t\to+\infty} e^{-it\Delta} \mbb{I}_{0}(z_{k_{1}} \bar{z_{k_{2}}} z_{k_{3}})
\]
exists, where we used \eqref{eq:amlne} in the second equality. 
In addition, \eqref{eq:limit-scattering-bound} implies
\[
\lim_{t\to+\infty} \Big\|z_{k}(t)-e^{it\Delta}w_{k}\Big\|_{H^{\mu(k,S)}_{x}(\R^{d})}
= \lim_{t\to+\infty} \Big\|e^{-it\Delta}z_{k}(t)-w_{k}\Big\|_{H^{\mu(k,S)}_{x}(\R^{d})}
= 0 \,,
\]
and scattering follows.

\textbf{Scattering of global solutions.}
Since we assume \eqref{eq:multilinear-datum-smallness}, then 
$u^{\#}_{M}$ is a fixed point of $\mc{J}_{z,M}$ in \(X^{\xs}([0,+\infty))\), and in particular 
$\|u^{\#}_{M}\|_{X^{\xs}([0,+\infty))} < \infty$.

Using \eqref{eq:F-bound} and \eqref{eq:rz3M-Xs-bound} we obtain for 
sufficiently small $< \epsilon \lesssim_{S,\xs, \xs_s} $ that 
\begin{multline*}
\Big\|\Phi_{z_{\leq  M}}[u^{\#}_{M}]+[z,z,z]_{>M}\Big\|_{X^{*,\xs+\epsilon}([0,+\infty))}
\\
\lesssim\Big(1 + \|u^{\#}_{M}\|_{X^{\xs}([0,+\infty))}+\max_{k\leq M}\|z_{k}\|_{Y^{\mu(k,S)+\epsilon}(\R)}\Big)^{3}<\infty.
\end{multline*}
Then \eqref{eq:limit-scattering-existence}  implies that 
\[
w^{\#}_{M}= \lim_{t\to +\infty} e^{-it\Delta}\mbb{I}_{0}\big(\Phi_{z_{\leq  M}}[u^{\#}_{M}]+[z,z,z]_{>M}\big)\,,
\]
is well defined as the limit in $H^{\xs}_x(\R^d)$, and  
\eqref{eq:remainder-scattering} follows from \eqref{eq:limit-scattering-bound}. 

\textbf{Continuous dependence on the multilinear data.} 
In `Local existence of solutions' we 
already argued that for any \(f\in H^{S+\epsilon}_{x}(\R^{d})\) with \(d(0,f)\leq R\)
there exists a unique solution \(u\in
C^{0}\big([0,T);H^{S+\epsilon}_{x}(\R^{d})\big)\) of the form
\eqref{eq:solution-decomposition--deterministic}, which satisfies
\(\|u^{\#}_{M}\|_{X^{\xs}([0,T))}<\infty \) and
\(u^{\#}_{M}=\mc{J}_{z,M}(u^{\#}_{M})\eqd \mbb{I}_{0}\big(\Phi_{z_{\leq
    M}}[u^{\#}_{M}]+[z,z,z]_{>M}\big)\). \Cref{prop:main-linear-estimate} and 
\eqref{eq:non-homogeneous-bound-2} show that \(\mbb{I}_{0}\) is a bounded
linear map from \(X^{*,\xs+\epsilon}([0,T))\) to \(X^{\xs}([0,T))\). By \Cref{lem:z-remainder-term-bound}, the map
\(\vec{z}_{M} \mapsto [z,z,z]_{>M} \) is Lipschitz continuous with respect to the distance introduced in \eqref{eq:initial-data-distance}. Finally, the map
\((\vec{z}_{M},v)\mapsto \Phi_{z_{\leq M}}[v]\) is Lipschitz continuous due to \eqref{eq:F-bound-Lip-1} and \eqref{eq:F-bound-Lip-2}.
\end{proof}

\section{Probabilistic estimates}\label{sec:multilinear-probabilistic-estimates}

In this section, we  establish probabilistic estimates for
the multilinear correction terms \(\rz_{k}\), defined in
\eqref{eq:def:zk-randomized}, and prove
\Cref{thm:random-multilinear-directional-bounds}.

To prove these estimates we have to keep track of how the random
variables \(\rz_{k}\) depend on the Wiener randomization of \(f\), denoted
by \(\rf\). We use ternary trees to encode the dependence of all
probabilistic terms on the initial datum \(\rf\).  We stress that trees
serve merely as a convenient notation, and we do not use any graph
theory or subtle properties of trees.

More precisely, we define the set \(\TT\) of ternary trees as
$\TT := \bigcup_{n \geq 1} \TT_{n}$, where for each
\(n\in \N\setminus{0}\) the set of trees \(\TT_{n}\) is given by induction as
follows:
\begin{itemize}
\item We say $\tau \in \TT_{1}$ if \(\tau = [ \bullet ]\).
\item We say $\tau \in \TT_{n}$ for $n \geq 1$ if \(\tau = [\tau_{1}, \tau_{2}, \tau_{3}]\) for some $\tau_{j} \in \TT_{n_{j}}$ and $n = n_{1} + n_{2} + n_{3}$ with $1 \leq n_{j} < n$.
\end{itemize}
The index \(n\) can be viewed as the number of leaves of a tree. 

Pictorially, our inductive construction of
\(\tau = [\tau_{1}, \tau_{2}, \tau_{3}]\) can be represented as follows
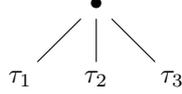
\begin{center}
\begin{tikzpicture}[level distance=1cm,
  level 1/.style={sibling distance=1cm},
  level 2/.style={sibling distance=1.5cm}]
  \node {$\bullet$}
    child {node {$\tau_1$}
    }
    child {node {$\tau_2$}
    }
    child {node {$\tau_3$}
    };
\end{tikzpicture}
\end{center}
where $\tau_j$, $j = 1, 2, 3$ are ternary trees.  If
$\tau \in \TT_{n}$, we set $|\tau|= n$ In particular, if
\(\tau = [\tau_{1}, \tau_{2}, \tau_{3}]\) then
\(|\tau| = |\tau_{1}| + |\tau_{2}| + |\tau_{3}|\). Under this convention,
$|\tau|$ corresponds to the number of nodes with no descendant, that we
refer to as ``leaves''. We omit the proof, as this fact does not
explicitly enter our discussion.

In the literature, one usually allows a node of a general ternary
tree to have zero, one, two, or three descendants. However, in the
present manuscript, we only consider trees in which each node has
either three descendants (children) or no descendants at all (a leaf). By
induction, one can show there are no ternary trees with an even number
of leaves, that is, \(\TT_{n}=\emptyset\) when \(n\in2\N\).  For example, we have
\[
\Big|\,\big[[ \bullet ] , [\bullet] , [\bullet]\big]\,\Big|=3\,,
\qquad
\Big|\, \Big[ [\bullet], \big[[ \bullet], [\bullet],[ \bullet]\big], [\bullet] \Big]\, \Big| = 5.
\]
Graphically, these trees can be represented as  
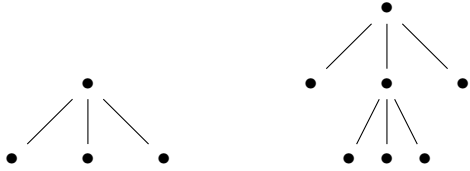
\begin{center}
\begin{tikzpicture}[level distance=1cm,
  level 1/.style={sibling distance=1cm},
  level 2/.style={sibling distance=1.5cm}]
  \node {$\bullet$}
    child {node {$\bullet$}
    }
    child {node {$\bullet$}
    }
    child {node {$\bullet$}
    };
\end{tikzpicture}
\qquad\qquad
\begin{tikzpicture}[level distance=1cm,
  level 1/.style={sibling distance=1cm},
  level 2/.style={sibling distance=.5cm}]
  \node {$\bullet$}
    child {node {$\bullet$}
    }
    child {node {$\bullet$}
     child{node {$\bullet$}}
     child{node {$\bullet$}}
     child{node {$\bullet$}}
    }
    child {node {$\bullet$}
    };
\end{tikzpicture}
\end{center}

To each tree \(\tau\in\TT_{n}\) with \(n\in2\N+1\) we inductively assign an
\(n\)-(real) linear tree operator \(R_{\tau}\) mapping \(n\)-tuples of
functions \((f_{1},\ldots,f_{n})\in L^{2}(\R^{d})\) to functions on \(\R\times\R^{d}\).
\begin{itemize}
\item We set \(R_{[\bullet]}[f](t,x) \eqd e^{it\Delta}f(x).\)
\item Inductively, we set
\[\eqnum\label{eq:def:R-inductive}
\begin{aligned}[c]
\hspace{5em}&\hspace{-5em} R_{[\tau_{1},\tau_{2},\tau_{3}]}[\mb{f}_{1}\oplus \mb{f}_{2}\oplus\mb{f}_{3}](t, \cdot)
\\
& \eqd \mp i \int_{0}^{t} e^{i (t - s) \Delta} \Big(R_{\tau_{1}} [\mb{f}_{1}] (s, \cdot)\,\bar{R_{\tau_{2}} [\mb{f}_{2}] (s, \cdot)}\,  R_{\tau_{3}} [\mb{f}_{3}] (s,\cdot)\Big) \dd s \,,
\end{aligned}
 \]
 where the choice of the sign is exactly the opposite to the sign on the right hand side of \eqref{eq:NLS-deterministic}. Here \(\mb{f}_{j}=(f_{j,1},\ldots,f_{j,|\tau_{j}|})\in \big(L^{2}(\R^{d})\big)^{|\tau_{j}|}\) are \(|\tau_{j}|\)-tuples of \(L^{2}(\R^{d})\) functions and 
\[
\mb{f}_{1}\oplus \mb{f}_{2}\oplus\mb{f}_{3}=\big(f_{1,1},\ldots,f_{1,|\tau_{1}|},f_{2,1},\ldots,f_{2,|\tau_{2}|},f_{3,1},\ldots,f_{3,|\tau_{3}|}^{3}\big)\,.
\]
\end{itemize}

To simplify notation, given \(\tau\in\TT\) and one function \(f \in L^{2}(\R^{d})\), we write
\[
R_{\tau}[f] := R_{\tau}\big[f,\ldots,f\big] \,,
\]
where the right-hand side contains $|\tau|$ copies of $f$.

The main result of this section shows that the functions
\(R_{\tau}[\rf]\) lie, almost surely, in the space
\(Y^{\mu(|\tau|,S)}\) with an appropriate regularity \(\mu(|\tau|,S)\) depending on
\(S\), the regularity of \(f\in H^{S}(\R^{d})\), and on \(|\tau|\), the order of
multilinearity of the operator \(R_{\tau}\).

\begin{proposition}\label{prop:tree-bound-probabilistic} 
For any $S \in \R$, $n \geq 1$, let \(\mu(k,S)\) be as defined in
\eqref{eq:def:multilinear-regularity}.

Fix \(\tau\in\TT\), \(S>0\), and let
\(0<\epsilon \lesssim_{|\tau|} 1\) be small. Then for every
\(0 < \epsilon_{0}\lesssim_{\epsilon, |\tau|, S} 1\) there exist constants
$C=C(\epsilon_{0},\epsilon,|\tau|,S) > 0$ independent of
$f\in L^{2}(\R^{d})$, such that for any $\lambda > 0$ it holds that
\[\eqnum\label{eq:tree-bound-probabilistic}
\mathbb{P} \Big(\big\{\omega \in \Omega : \big\| R_{\tau}[\rf] \big\|_{Y^{\mu(|\tau|, S)} (\R)} > \lambda\big\}\Big)
\leq C \exp\Bigg(- \frac{ \lambda^{\frac{2}{|\tau|}}}{C\|f\|_{H^{S +\epsilon}_{x} (\R^{d})}^{2}}\Bigg)\,,
\]
where \(\mf{f}\) is the unit scale randomization of \(f\) given by
\eqref{eq:initial-data-randomization}. In particular,
\[\eqnum\label{eq:R-tau-almost-sure}
\big\|R_{\tau}[\rf] \big\|_{Y^{\mu(|\tau|, S)} (\R)} < \infty \qquad \text{almost surely}
\]
for any \(f\in H^{S+\epsilon}_{x}(\R^{d})\).
\end{proposition}

First we use \Cref{prop:tree-bound-probabilistic} to obtain the proof of 
\Cref{thm:random-multilinear-directional-bounds}.

\begin{proof}[Proof of \Cref{thm:random-multilinear-directional-bounds}]
The set $\bigcup_{k \leq M}\TT_k$ is finite, and therefore by \eqref{eq:def:zk-randomized} there is a constant $\tilde{C} = \tilde{C}(M)$ such that for any integer $k \leq M$, we have 
\[
\begin{aligned}
\mathbb{P} \Big(\big\{\omega \in \Omega : \big\| \rz_{k}\big\|_{Y^{S} (\R)} > \lambda\big\}\Big)
& \leq
\mathbb{P} \Big(\bigcup_{\substack{\tau\in\TT_k}}\big\{\omega \in \Omega : \big\| R_{\tau}[\rf]\big\|_{Y^{\mu(k,S)} (\R)} > \tilde{C}^{-1}\lambda\big\}\Big) 
\\
& \leq \sum_{\substack{\tau\in\TT_k}} \mathbb{P} \Big(\big\{\omega \in \Omega : \big\| R_{\tau}[\rf]\big\|_{Y^{\mu(k,S)} (\R)}
> \tilde{C}^{-1}\lambda\big\}\Big) \,.
\end{aligned}
\]
From \eqref{eq:tree-bound-probabilistic} we obtain that
\[
\mathbb{P} \Big(\big\{\omega \in \Omega : \big\| \rz_{k}\big\|_{Y^{S} (\R)} > \lambda\big\}\Big)\lesssim \tilde{C}C
C \exp\Bigg(- \frac{ \lambda^{\frac{2}{k}}}{\tilde{C}^{2}C\|f\|_{H^{S +\epsilon}_{x} (\R^{d})}^{2}}\Bigg),
\]
as required. 
\end{proof}

Next, we proceed to the proof of
\Cref{prop:tree-bound-probabilistic}. We rely on the deterministic
estimate of \Cref{lem:R-estimate}, which establishes regularity for
the operators \(R_{\tau}\) when evaluated on functions with bounded
frequency support. Then, using \Cref{lem:Chebyshev-probab} below, we
show that the probabilistic bounds of
\Cref{prop:tree-bound-probabilistic} follow from estimates on 
sufficiently large moments of the random variable
\(\big\|R_{\tau}[\rf]\|_{Y^{\mu(|\tau|, S) } (\R)}\). Such moment bounds are
established with help of multi-parameter Wiener chaos estimates
(\Cref{lem:hypercontractivity}).

First, we state and prove or provide references for the required
lemmata. Finally, we prove \Cref{prop:tree-bound-probabilistic}.

\begin{lemma}\label{lem:R-estimate}
Let $\mu(n,S)$ be as in \eqref{eq:def:multilinear-regularity} and fix $\tau \in \TT$ and
\(S, R_0 >0\). Choose any \(0 < \epsilon \lesssim_{|\tau|} 1\) and 
\(0 < \epsilon_{0} \lesssim_{\epsilon, S, |\tau|} 1\).  For any tuple
\(\mb{f}=(f_{1},\ldots f_{|\tau|})\in\big(L^{2}(\R^{d})\big)^{|\tau|}\) of functions with
\(\diam\big(\spt(\FT{f}_{j}))\leq 2R_0\), \(j\in\{1,\ldots,|\tau|\}\), it holds that
\[ \eqnum\label{eq:R-estimate}
\Big\| R_{\tau} \big[ \mb{f} \big]\Big\|_{Y^{\mu(|\tau|,S)} (\R)}  \lesssim \prod_{j = 1}^{| \tau |} \| f_{j} \|_{H^{S+\epsilon}(\R^{d})}.
\]
The implicit constant may depend on \(|\tau|\), \(\epsilon\), \(\epsilon_{0}\), and \(R_0\).
\end{lemma}

The following lemma relates the regularity parameter \(\mu(|\tau|,S)\) and the trilinear bound \eqref{eq:YYY}. 

Our probabilistic estimates depend crucially on the Wiener chaos
estimates. If \(n=1\) and $g_k$ has a standard normal distribution, the
claim of \Cref{lem:hypercontractivity} reduces to the classical
Khintchine inequality for Gaussian sums. Here, we provide a
generalized statement.

\begin{lemma}\label{lem:hypercontractivity}
Let $( g_{k} )_{k\in\Z^{d}} $ be a collection of complex-valued i.i.d
random variables, such that
\begin{equation}
 \E \big[| g_{k} |^p\big] < \infty \qquad \text{for all} \quad  p \in [1, \infty) . 
\end{equation}
Let $c : (\Z^{d})^{n} \rightarrow \C$ be a function with finite support. Then, the
random variable
\[
H = \sum_{k_{1},\ldots,k_{n} \in  (\Z^{d})^{n}} c (k_{1},\ldots,k_{n}) g_{k_{1}} \ldots g_{k_{d}}
\]
satisfies the estimate
\begin{equation} \label{eq:mlce}
\Big( \E \big[| H |^p \big]\Big)^{1 / \gamma} \leq_{d,n, \gamma,g} \Big( \E\big[ | H |^2\big] \Big)^{1 / 2}\qquad \textrm{ for all} \quad  \gamma \in [2, \infty)\,,
 \end{equation}
where the implicit constant depends only on \(n\), \(\gamma\), and the
distribution of the variables \( (g_{k})_{k\in\N}\), but not on the
function $c$ nor on the size of its support.
\end{lemma}

The proof Lemma \ref{lem:hypercontractivity} reformulated as
\Cref{thm:hypercontractivity-main}, is provided in
\Cref{sec:multilinear-chaos-appendix}. The multilinear chaos
estimate \eqref{eq:mlce} is well known, see for example \cite[Proposition
2.4]{thomannGibbsMeasurePeriodic2010} if $(g_k)$ are 
Gaussian random variables and its proof relies on estimates for eigenfunctions
of the Ornstein-Uhlenbeck process. Our proof below is 
elementary and self-contained using combinatorial
techniques partly inspired by 
\cite{thomannGibbsMeasurePeriodic2010}

Finally, we reformulate \cite[Lemma
4.5]{tzvetkovConstructionGibbsMeasure2009}. We leave out the prof
proof: it coincides with that of \cite[Lemma
4.5]{tzvetkovConstructionGibbsMeasure2009} and follows from the
Chebyshev inequality and an appropriate optimization.

\begin{lemma}\label{lem:Chebyshev-probab}
Let $H$ be a random variable and suppose that there exist $K > 0$,
$\gamma_{0} \geq 1$, and $k \geq 1$ such that for any $\gamma \geq \gamma_{0}$ we have
\[
\big(\E |H|^{\gamma}\big)^{\sfrac{1}{\gamma}} \leq \gamma^{\frac{k}{2}} K .
\]
Then, there exist $c > 0$ and $C> 0$ depending on $\gamma_{0}$ and $k$, but
independent of $K$ and $\gamma$, such that for every $\lambda > 0$,
\[
\mathbb{P} \Big(\big\{\omega \in \Omega : |H | > \lambda\big\}\Big) \leq C \exp\bigg(-c \frac{ \lambda^{\frac{2}{k}}}{ K^{\frac{2}{k}}}  \bigg) \,.
\]
In particular, we have
\[
\mathbb{P} (\{\omega \in \Omega : |H | < \infty\}) = 1.
\]
\end{lemma}
From the above, we can deduce an estimate on the probability distribution of the Sobolev norms of the randomized initial data.
\begin{corollary}\label{cor:probability-of-wiener-randomization}
For some \(C,c>0\) it holds that
\[\eqnum\label{eq:probability-of-wiener-randomization}
\mathbb{P} \Big(\big\{\omega \in \Omega : \|\rf\|_{H^{S}_{x}(\R^{d})} > \lambda\big\}\Big) \leq C \exp\bigg(-c \frac{ \lambda^{2}}{\|f\|^{2}_{H^{S}_{x}(\R^{d})} }  \bigg) \,.
\]
\end{corollary}
\begin{proof}
For any \(\gamma\geq2\), the Minkowski inequality implies
\[
\big(\E\|\rf\|_{L^{2}_{x}(\R^{d})}^{\gamma}\big)^{\sfrac{1}{\gamma}}\leq\big\|\big(\E|\rf|^{\gamma}\big)^{\sfrac{1}{\gamma}}\big\|_{L^{2}_{x}(\R^{d})}
\]
According to the definition of the unit-scale Wiener randomization \eqref{eq:initial-data-randomization}, the pointwise value \(\rf(x)\) has the form required by \Cref{lem:hypercontractivity}, and 
therefore 
\[
\big(\E|\rf|^{\gamma}\big)^{\sfrac{1}{\gamma}}\lesssim (\gamma-1)^{\sfrac{1}{2}}\big(\E|\rf|^{2}\big)^{\sfrac{1}{2}} \,.
\]
Finally, using independence of the random variables \(g_{k}\) and Plancherel's identity, we obtain that
\[
\big\|\big(\E|\rf|^{2}\big)^{\sfrac{1}{2}}\big\|_{L^{2}_{x}(\R^{d})}^{2}
\begin{aligned}[t]
& =\sum_{k,k'\in\Z^{d}}\E (g_{k}\bar{g_{k'}})\int_{\R^{d}}\QP_{k}f(x)\bar{\QP_{k'}f(x)}\dd x
\\
&\lesssim\sum_{k\in\Z^{d}}\int_{\R^{d}}\big|\QP_{k}f(x)\big|^{2}\dd x
=\sum_{k\in\Z^{d}}\int_{\R^{d}}|\chi_{k}(\xi)|^{2}\big|\FT{f}(\xi)\big|^{2}\dd x
\\
&\lesssim\|f\|_{L^{2}_{x}(\R^{d})}^{2}\,,
\end{aligned}
\]
where 
in the last inequality we used that \(\sum_{k\in\Z^{d}}|\chi_{k}(\xi)|^{2}\lesssim1\), since \(|\chi_{k}(\xi)|\leq1\) and for any \(\xi\in\R^{d}\) there are finitely many \(k\in\Z^{d}\) for which \(\xi\in\spt \chi_{k}\). Thus,
\[
\big(\E\|\rf\|_{L^{2}_{x}(\R^{d})}^{\gamma}\big)^{\sfrac{1}{\gamma}} \lesssim (\gamma-1)^{\sfrac{1}{2}}\|f\|_{L^{2}_{x}(\R^{d})}^{2}
\]
and from \Cref{lem:Chebyshev-probab} follows that
\[
\mathbb{P} \Big(\big\{\omega \in \Omega : \|\rf\|_{L^{2}_{x}(\R^{d})} > \lambda\big\}\Big) \leq C \exp\bigg(-c \frac{ \lambda^{2}}{\|f\|_{L^{2}_{x}(\R^{d})}^{2} }  \bigg) \,.
\]
Since the Wiener randomization commutes with the operator \(\langle\Delta\rangle^{S}\), we can replace \(f\) with \(\langle\Delta\rangle^{S}f\) and the claim follows.
\end{proof}

Next, we prove \Cref{lem:R-estimate}, and finally we show \Cref{prop:tree-bound-probabilistic}.

\begin{proof}[Proof of \Cref{lem:R-estimate} ]
We prove the claim by induction on \(|\tau|\). First, we claim that for each $t \in \R$ and $\tau \in \TT$ we have 
\[\eqnum\label{eq:R-tau-k-fourier-scale}
\diam\Big(\spt \big(\Fourier( R_{\tau}[\mb{f}](t)) \big)\Big)\leq2R_0|\tau|.
\]
For the base step, if $| \tau |=1$, that is, \(\tau=[\bullet]\), then $R_{[\bullet]}[f] = e^{it\Delta} f$. Thus, for every \(t\in\R\), we have \(
\diam\big(\spt\big(\Fourier(R_{[\bullet]}[f](t))\big)\big) \leq 2R_0\), since \(\Fourier(e^{it\Delta}f)(\xi)=e^{-4\pi^{2}i|\xi|^{2}}\FT{f}(\xi)\) and $\diam\big(\spt\big(\hat{f}\big)\big) \leq 2R_0$. 

Next, fix $n > 1$ and assume that \eqref{eq:R-tau-k-fourier-scale}
holds for any $\tau \in \TT$ with $|\tau| < n$. Choose any $\tau \in \TT$ with $|\tau|
= n$ and let $\tau_j \in \TT$, $j \in\{1, 2, 3\}$, be such that $\tau = [\tau_1, \tau_2,
\tau_3]$. For any functions
\(h_{j}\in L^{2}(\R^{d})\), \(j\in\{1,2,3\}\) it holds that
\[
\spt \big(\FT{h_{1} \bar{h_{2}} h_{3}}\big) \subset \spt (\FT{h}_{1}) - \spt (\FT{h}_{2}) +   \spt (\FT{h}_{3})\,,
\]
and thus 
\[
\diam\Big(\spt \big(\FT{h_{1} \bar{h_{2}} h_{3}}\big)\Big)
\leq \diam\big(\spt (\FT{f}_{1})\big)
+ \diam\big(\spt (\FT{f}_{2})\big)
+ \diam\big(\spt (\FT{f}_{3})\big).
\]
Since $n = |\tau| = |\tau_1| + |\tau_2| + |\tau_3|$ and $|\tau_j| \geq 1$, we
obtain that $|\tau_j| < n$, $j \in\{1, 2, 3\}$, and we can use the induction hypothesis to deduce for any $s \in \R$ that
\[
\begin{aligned}
\hspace{10em}&\hspace{-10em}  \diam\Big(\spt \big(\Fourier(R_{\tau_{1}} [\mb{f}_{1}] (s, \cdot)\,\bar{R_{\tau_{2}} [\mb{f}_{2}] (s, \cdot)}\,  R_{\tau_{3}} [\mb{f}_{3}] (s,\cdot))\big)\Big) 
\\
&\leq 
 \sum_{j = 1}^3
 \diam\Big(\spt \big(\Fourier(R_{\tau_{j}} [\mb{f}_{j}] (s, \cdot))\big)\Big)  \leq \sum_{j = 1}^3 2R_0|\tau_j| = 2R_0|\tau|.
\end{aligned}
\]
The induction step follows since the multiplication by $e^{4\pi^{2}
i(t-s)|\xi|^2}$, or integration in time does not change the support in
Fourier space.

Next, let us inductively prove the bound
\[ \eqnum\label{eq:R-estimate-2}
\Big\| R_{\tau} \big[ \mb{f} \big]\Big\|_{Y^{\mu(|\tau|,S)} (\R)} \lesssim_{R_0, |\tau|, S} \prod_{j = 1}^{| \tau |} \| f_{j} \|_{H^{S+2|\tau|\epsilon}(\R^{d})} \,,
\]
from which \eqref{eq:R-estimate} follows, since $\epsilon \lesssim_{|\tau|} 1$ is
small. If $| \tau |=1$, that is, \(\tau=[\bullet]\), then $\mu(|\tau|, S) = S$. Also,
$R_{[\bullet]}[f] = e^{it\Delta} \QP_{k} f$ and \eqref{eq:R-estimate-2} follows
from \eqref{eq:non-homogeneous-bound-unit-scale-2} with $h = 0$.

Fix \(n\in2\N+1\), $n > 1$ and assume that \eqref{eq:R-estimate-2} holds
for all \(\tau\in\TT \) with \(|\tau|<n\). As above, let \(\tau_{j}\in\TT_{n_{j}}\) be such that $\tau = [\tau_1, \tau_2, \tau_3]$ with
$|\tau_j| < n$ for each $j \in \{1, 2, 3\}$.  For the rest of the proof, we
allow all our constants to depend on \(|\tau|, S\), and $R_0$.

Since $R_{\tau}[\mb{f}]$ has bounded support as shown in
\eqref{eq:R-tau-k-fourier-scale}, then from
\eqref{eq:non-homogeneous-bound-unit-scale-2} with $v_0 = 0$ and
\(\mu(|\tau|,S+\epsilon)\geq\mu(|\tau|,S)+\epsilon\) we obtain that
\[
\big\|R_{\tau}[\mb{f}]\big\|_{Y^{\mu(|\tau|,S)+\epsilon}(\R)}
\begin{aligned}[t]
& \lesssim\big\|R_{\tau}[\mb{f}]\big\|_{Y^{\mu(|\tau|,S+\epsilon)}(\R)}
\\
& \lesssim\Big\|R_{\tau_{1}} [\mb{f}_{1}] (s, \cdot)\,\bar{R_{\tau_{2}} [\mb{f}_{2}] (s, \cdot)}\,  R_{\tau_{3}} [\mb{f}_{3}] (s,\cdot) 
\Big\|_{X^{*,\mu(|\tau|,S+2\epsilon)} (\R)} \,.
\end{aligned}
\]
Consequently, 
\eqref{eq:trilinear-Y-mu-bound} yields
\[
\begin{aligned}
\hspace{5em} & \hspace{-5em}
\Big\|R_{\tau_{1}} [\mb{f}_{1}] (s, \cdot)\,\bar{R_{\tau_{2}} [\mb{f}_{2}] (s, \cdot)}\,  R_{\tau_{3}} [\mb{f}_{3}] (s,\cdot)  \Big\|_{X^{*,\mu(|\tau|,S+2\epsilon)} (\R)}
\\
& 
\lesssim \prod_{j=1}^{3} \Big\|R_{\tau_{j}} [\mb{f}_{j}] \Big\|_{Y^{\mu(|\tau_{j}|,S+2\epsilon)+\epsilon}(\R)}
\lesssim \prod_{j=1}^{3} \Big\|R_{\tau_{j}} [\mb{f}_{j}] \Big\|_{Y^{\mu(|\tau_{j}|,S+3\epsilon)}(\R)}.
\end{aligned}
\]
Using the inductive hypothesis with \(S+3\epsilon\) in place of \(S\) we obtain 
\[
\Big\|R_{\tau_{j}} [\mb{f}_{j}] \Big\|_{Y^{\mu(|\tau_{j}|,S+3\epsilon)}(\R)}
\lesssim \prod_{k=1}^{|\tau_{j}|}\|f_{j}\|_{H^{S+(2|\tau_j| + 3)\epsilon}}.
\]
Since $|\tau_j| \leq |\tau| - 2$, then $2|\tau_j| + 3 \leq 2|\tau|$, and therefore 
\[
\Big\|R_{\tau}[\mb{f}]\Big\|_{Y^{\mu(|\tau|,S)}(\R)} \lesssim\prod_{k=1}^{|\tau|}\|f_{k}\|_{H^{S+2|\tau|\epsilon}}\,,
\]
as desired
\end{proof}

\begin{proof}[Proof of \Cref{prop:tree-bound-probabilistic}]
Fix \(\tau\in\TT\).  Since \(\TT_{n}=\emptyset\) if \(n\in2\N\), we only consider \(|\tau|\in2\N+1\). According to \Cref{lem:Chebyshev-probab}, it suffices to prove that
\[\eqnum\label{eq:high-moment-random-bound}
\Big( \E\big\| R_{\tau} [\rf]\big\|_{Y^{\mu(|\tau|, S)} (\R)}^{\gamma } \Big)^{\frac{1}{\gamma}}\lesssim_{|\tau|}
\gamma^{\frac{|\tau|}{2}} \|f\|_{H^{S+\epsilon}(\R^{d})}^{|\tau|}
\]
for all \(\gamma\) large enough. Fix $\gamma > \frac{2}{\epsilon_{0}}$ and recall that  $\epsilon_{0}$ is a small, fixed constant appearing in the definition \eqref{eq:def:Y-norm} of the space \(Y\) depending on $S$, $|\tau|$, $\epsilon$.

Since $\gamma >\frac{2}{\epsilon_{0}}> 2$ and \(\epsilon_{0}<2^{-100}\), from the definition of $Y^S$ and Minkowski's inequality it follows that
\[
\Big( \E\big\| R_{\tau} [\rf]\big\|_{Y^{\mu(|\tau|, S)} (\R)}^{\gamma } \Big)^{\frac{1}{\gamma}}
\begin{aligned}[t]
& =  \Big( \E \big( \sum_{N \in 2^{\N}} N^{2 \mu(|\tau|, S)}  \big\| \LP_{N} R_{\tau} [\rf]\big\|_{Y_{N}(\R)}^{2} \big)^{\frac{\gamma}{2}} \Big)^{\frac{1}{\gamma}}
\\ 
& \leq \Big( \sum_{N \in 2^{\N}} N^{2 \mu(|\tau|,S)} \big( \E \big\| \LP_{N} R_{\tau} [\rf]\big\|^{\gamma}_{Y_{N}(\R)} \big)^{\frac{2}{ \gamma}} \Big)^{\frac{1}{2}} .
\end{aligned}
\]
Minkowski's inequality also gives that
$(\E\|f\|_{L^{p}}^{\gamma})^{\sfrac{1}{\gamma}}\leq\big\|
\big(\E|f|^{\gamma}\big)^{\sfrac{1}{\gamma}}\big\|_{L^{p}}$ whenever $\gamma \geq p$, where the
$L^{p}$ norm is over $t, x_{1}$, or $x'$.  Having assumed that $\gamma$ is
larger than any $L^{p}$ integrability exponent of appearing in the
definition of the norms $Y_{N}(\R)$ (see \eqref{eq:def:Y-norm}), we
obtain that
\[\eqnum\label{eq:multi-Y-minkowski}
\Big( \E \big\| R_{\tau} [\rf]\big\|_{Y^{\mu(|\tau|,S)} (\R)}^{\gamma } \Big)^{\frac{1}{\gamma}}
\lesssim
\Big( \sum_{N \in 2^{\N}} N^{2\mu(|\tau|,S)} \Big\| \Big( \E \big| \LP_{N} R_{\tau} [\rf]\big|^{\gamma} \Big)^{\sfrac{1}{\gamma}} \Big\|_{Y_{N} (\R)}^{2} \Big)^{\frac{1}{2}}.
\]
Henceforth, we fix $N \in 2^{\N}$ and we focus on bounding each term of
the sum on the right-hand side of \eqref{eq:multi-Y-minkowski} individually.

Recall that the function \(\rf\) is obtained via the randomization
procedure \(\rf \eqd\sum_{k\in\Z^{d}}g_{k}\QP_{k} f\). We then use that the map
\(\mb{f}=(f_{1},\ldots,f_{|\tau|})\mapsto R_{\tau} [\mb{f}]\) is linear in the odd entries
and anti-linear in even entries \(f_{j}\) to obtain that
\[\eqnum\label{eq:mleor}
\LP_{N} R_{\tau} [\rf](t,x)
\begin{aligned}[t]
& = \LP_N\Big(\sum_{\mrl{\hspace{-1em}\mb{k}\in(\Z^{d})^{|\tau|}}}  g_{\mb{k}} R_{\tau} \big[ \QP_{k_{1}}f,\ldots,\QP_{k_{|\tau|}}f\big](t,x)\Big)
\\
& = \LP_N\Big(\sum_{\mrl{\hspace{-1em}\mb{k}\in(\Z^{d})^{|\tau|}}}  g_{\mb{k}} R_{\tau}^{\mb{k}} \big[f\big](t,x)\Big) \,,
\end{aligned}
\]
where \(\mb{k}=(k_{1},\ldots,k_{|\tau|})\), \(g_{\mb{k}}\eqd g_{k_{1}}\bar{g_{k_{2}}}\ldots  g_{k_{|\tau|-2}}\bar{g_{k_{|\tau|-1}}}g_{k_{|\tau|}}\), and we use the notation \(R_{\tau}^{\mb{k}}[f]\eqd R_{\tau} \big[ \QP_{k_{1}}f,\ldots,\QP_{k_{|\tau|}}f\big](t,x)\). We claim that the following crucial estimate
\[\eqnum\label{eq:Y-Wiener-bilinearizations}
\Big\| \Big( \E \big| \LP_{N} R_{\tau} [\rf]\big|^{\gamma} \Big)^{\sfrac{1}{\gamma}} \Big\|_{Y_{N} (\R)}^{2}
\lesssim_{|\tau|}
\begin{aligned}[t]
&  (\gamma-1)^{|\tau|} 
\\ 
& \hspace{-7em} \times \sum_{\mrl{\hspace{-1em}\mb{k},\mb{l}\in(\Z^{d})^{|\tau|}}}\;\big|\E\big(g_{\mb{k}}\bar{g_{\mb{l}}}\big)\big|
\Big\|   \LP_{N} R_{\tau}^{\mb{k}} [f]\Big\|_{Y_{N} (\R)}  \Big\| \LP_{N} R_{\tau}^{\mb{l}} [f]\Big\|_{Y_{N} (\R)}
\end{aligned}
\]
holds.

Assuming \eqref{eq:Y-Wiener-bilinearizations} holds, let us prove
\eqref{eq:tree-bound-probabilistic} first. Applying bound
\eqref{eq:Y-Wiener-bilinearizations} to \eqref{eq:multi-Y-minkowski}
we obtain that
\[
\begin{aligned}[t]
\Big( \E \big\| R_{\tau} [\rf]\big\|_{Y^{\mu(|\tau|,S)} (\R)}^{\gamma } \Big)^{\frac{1}{\gamma}}
  \lesssim_{|\tau|} (\gamma-1)^{\frac{|\tau|}{2}} \bigg(&
\sum_{\mrl{\hspace{-0.5em}\mb{k},\mb{l}\in(\Z^{d})^{|\tau|}}}\;\big|\E\big(g_{\mb{k}}\bar{g_{\mb{l}}}\big)\big|
\\ & \hspace{-8em}\times 
\sum_{N\in2^{\N}} N^{2\mu(|\tau|,S)}
\Big\|\LP_{N} R_{\tau}^{\mb{k}} [f]\Big\|_{Y_{N} (\R)}  \Big\| \LP_{N} R_{\tau}^{\mb{l}} [f]\Big\|_{Y_{N} (\R)} \bigg)^{\frac{1}{2}}\,.
\end{aligned}
\]
Using the Cauchy-Schwarz inequality in \(N\) and the definition \eqref{eq:def:Y-norm} of the norm $Y^{\sigma}(\R)$, it follows from the bound \eqref{eq:R-estimate} of \Cref{lem:R-estimate} that 
\[ \begin{aligned}
& \sum_{N\in2^{\N}} N^{2\mu(|\tau|,S)}
\Big\|\LP_{N} R_{\tau}^{\mb{k}} [f]\Big\|_{Y_{N} (\R)}  \Big\| \LP_{N} R_{\tau}^{\mb{l}} [f]\Big\|_{Y_{N} (\R)}
\\& \qquad
\begin{aligned}
 & \leq \Big(\sum_{N\in2^{\N}} N^{2\mu(|\tau|,S)} \Big\|\LP_{N} R_{\tau}^{\mb{k}} [f]\Big\|^2_{Y_{N} (\R)} \Big)^{\frac{1}{2}}  \Big(\sum_{N\in2^{\N}} N^{2\mu(|\tau|,S)} \Big\| \LP_{N} R_{\tau}^{\mb{l}} [f]\Big\|^2_{Y_{N} (\R)} \Big)^{\frac{1}{2}}
\\
& = \big\|R_{\tau}^{\mb{k}} [f]\big\|_{Y^{\mu(|\tau|,S)} (\R)} \big\|R_{\tau}^{\mb{l}} [f]\big\|_{Y^{\mu(|\tau|,S)} (\R)}
\\
&\lesssim \prod_{j = 1}^{| \tau |} \|\QP_{k_{j}} f \|_{H^{S+\epsilon}(\R^{d})} \prod_{j = 1}^{| \tau |} \|\QP_{l_{j}} f \|_{H^{S+\epsilon}(\R^{d})}  \,,
\end{aligned}
\end{aligned} \]
and hence, 
\[\eqnum\label{eq:eonrt}
\begin{aligned}
& \Big( \E \big\| R_{\tau} [\rf]\big\|_{Y^{\mu(|\tau|,S)} (\R)}^{\gamma } \Big)^{\frac{1}{\gamma}}
\lesssim_{|\tau|}  (\gamma-1)^{\frac{|\tau|}{2}} 
\\ &
\qquad\times  \bigg( \sum_{\mrl{\hspace{-0.5em}\mb{k},\mb{l}\in(\Z^{d})^{|\tau|}}} \; \big|\E\big(g_{\mb{k}}\bar{g_{\mb{l}}}\big)\big| \Big(\prod_{j = 1}^{| \tau |} \|\QP_{k_{j}} f \|_{H^{S+\epsilon}(\R^{d})}\Big)
\Big(\prod_{j = 1}^{| \tau |} \|\QP_{l_{j}} f \|_{H^{S+\epsilon}(\R^{d})}\Big)
\bigg)^{\frac{1}{2}}\,.
\end{aligned}
\]
Since \((g_k)_{k\in\Z^{d}}\) satisfy \(\big|\E\big(g_{\mb{k}}\bar{g_{\mb{l}}}\big)\big|\lesssim_{|\tau|} \E \big(|g_{0}|^{2|\tau|}\big)\lesssim_{|\tau|}1\) and \(\big|\E\big(g_{\mb{k}}\bar{g_{\mb{l}}}\big)\big|=0\) if any fixed \(m\in\Z^{d}\) appears an odd number of times in sequence 
$(k_1, \cdots, k_{|\tau|}, l_1, \cdots, l_{|\tau|})$ of elements in \(\Z^{d}\). Then \eqref{eq:eonrt}  becomes 
\[
\Big( \E \big\| R_{\tau} [\rf]\big\|_{Y^{\mu(|\tau|,S)} (\R)}^{\gamma } \Big)^{\frac{1}{\gamma}}
\begin{aligned}[t]
& \lesssim_{|\tau|} (\gamma-1)^{\frac{|\tau|}{2}}  \bigg( \sum_{\mrl{\hspace{-0.7em}\mb{k}\in(\Z^{d})^{|\tau|}}}\hspace{2em} \prod_{j = 1}^{| \tau |} \|\QP_{k_{j}} f \|_{H^{S+\epsilon}(\R^{d})}^{2} \bigg)^{\frac{1}{2}}
\\
& \lesssim_{|\tau|} (\gamma-1)^{\frac{|\tau|}{2}} \bigg( \sum_{k\in\Z^{d}} \|\QP_{k} f \|_{H^{S+\epsilon}(\R^{d})}^{2} \bigg)^{\frac{|\tau|}{2}}\,.
\end{aligned}
\]
We claim that
\[ \eqnum\label{eq:eohsnr}
\sum_{k\in\Z^{d}} \|\QP_{k} f \|_{H^{S+\epsilon}(\R^{d})}^{2} \leq \| f \|_{H^{S+\epsilon}(\R^{d})}^{2} \,.    
\]
Indeed, since $0 \leq \psi(\xi) \leq 1$, then 
$\sum_{k\in\Z^{d}} |\psi(\xi-k)|^{2} \leq \sum_{k\in\Z^{d}} |\psi(\xi-k)|=1$, and therefore
\[
\sum_{k\in\Z^{d}} \|\QP_{k} f \|_{H^{S+\epsilon}(\R^{d})}^{2}
\begin{aligned}[t]
& =
\sum_{k\in\Z^{d}} \int_{\R^{d}}(1+|\xi|)^{2S+2\epsilon}|\psi(\xi-k)|^{2} |\hat{f}(\xi)|^{2}\dd \xi
\\
& \leq \int_{\R^{d}}(1+|\xi|)^{2S+2\epsilon}|\hat{f}(\xi)|^{2} d\xi 
\end{aligned}
\]
and the required bound \eqref{eq:high-moment-random-bound}
follows.

We conclude the proof by showing that \eqref{eq:Y-Wiener-bilinearizations} holds. By \eqref{eq:mleor} and \Cref{lem:hypercontractivity} applied point-wise in $(t,x)$, we have
\[\eqnum\label{eq:smszn}
\begin{aligned}
 \Big( \E \big| \LP_{N} R_{\tau}[\rf] (t, x) \big|^{\gamma} \Big)^{\sfrac{1}{\gamma}}
& \lesssim_{|\tau|} (\gamma - 1)^{\frac{|\tau|}{2}} \Big( \E \big| \LP_{N} R_{\tau}[\rf](t, x) \big|^{2}
\Big)^{\sfrac{1}{2}}
\\
& \hspace{-5em} = (\gamma - 1)^{\frac{|\tau|}{2}} \bigg(\sum_{\mrl{\hspace{-0.5em}\mb{k},\mb{l}\in(\Z^{d})^{|\tau|}}}\; \E\Big( g_{\mb{k}} \bar{g_{\mb{l}} }\Big)\; R_{\tau}^{\mb{k}} \big[ f\big](t,x) \bar{R_{\tau}^{\mb{l}} \big[ f\big](t,x)}\bigg)^{\sfrac{1}{2}}\,.
\end{aligned}
\]
Next, for any \(v\colon\R\times\R^{d}\to\C\) we define the norm \(Y_{N}^{\sfrac{1}{2}}(\R)\) by halving all integrability exponents in the expression \eqref{eq:def:Y-norm} for the norm \(Y_{N}(\R)\):
\begin{multline*}
 \|v\|_{Y_{N}^{\sfrac{1}{2}} (\R)} \eqd
 \|v\|_{L_{t}^{\frac{1}{\epsilon_{0}}} L_{x}^{\frac{1}{1 - \epsilon_{0}}}(\R \times \R^{d})}
+ \|v\|_{L_{t}^{\frac{1}{1 - \epsilon_{0}}} L_{x}^{\frac{d}{d - 2}  \frac{1}{1 - \epsilon_{0}}} (\R \times \R^{d})}
\\
\begin{aligned}
& + \sum_{l = 1}^{d} \Big( N^{- \frac{1}{2}} \|v\|_{L_{e_{l}}^{( \frac{1}{1 -\epsilon_{0}}, \frac{1}{\epsilon_{0}}, \frac{1}{\epsilon_{0}} )} (\R)}
+ N^{- \frac{1}{2}} \|v\|_{L_{e_{l}}^{( \frac{1}{1- \epsilon_{0}}, \frac{1}{\epsilon_{0}}, \frac{\mf{c}_{0} }{2(1-\epsilon_{0})})} (\R )} \Big)
\\
& + \sum_{l = 1}^{d} \Big( N^{\frac{1}{2}} \| \UP_{e_{l}} v\|_{L_{e_{l}}^{(\frac{1}{\epsilon_{0}}, \frac{1}{1 - \epsilon_{0}}, \frac{1}{1 -\epsilon_{0}} )} (\R)} + N^{\frac{1}{2}} \| \UP_{e_{l}} v\|_{L_{e_{l}}^{( \frac{1}{\epsilon_{0}}, \frac{1}{1 - \epsilon_{0}}, \frac{1}{1 - \epsilon_{0}} )} (\R)} \Big).
\end{aligned}
\end{multline*}

Note that \(\big\||v|^{\sfrac{1}{2}}\big\|_{Y_{N}(\R)}^{2}\approx_{\epsilon_0} \|v\|_{Y_{N}^{\sfrac{1}{2}}(\R)}\), and consequently after taking the \(Y_{N}(\R)\) norm on both sides of \eqref{eq:smszn} and using the triangle inequality for \(Y_{N}^{\sfrac{1}{2}}(\R)\) we have that 
\[\begin{aligned}
& \Big\| \Big( \E \big| \LP_{N} R_{\tau} [\rf]\big|^{\gamma} \Big)^{\sfrac{1}{\gamma}} \Big\|_{Y_{N} (\R)}^{2}
\\
& \hspace{5em}\lesssim_{|\tau|, \epsilon_0} (\gamma-1)^{|\tau|}
\sum_{\mrl{\hspace{-1em}\mb{k},\mb{l}\in(\Z^{d})^{|\tau|}}}
\;\big|\E\big(g_{\mb{k}}\bar{g_{\mb{l}}}\big)\big|
\Big\|   \LP_{N} R_{\tau}^{\mb{k}} [f] \LP_{N} R_{\tau}^{\mb{l}} [f]\Big\|_{Y_{N}^{\sfrac{1}{2}} (\R)}.
\end{aligned}
\]
Finally,  by Hölder's inequality one has
\(
\big\| v_{1}v_{2} \big\|_{Y_{N}^{1/2} (\R)}   \lesssim\big\| v_{1}\big\|_{Y_{N} (\R)}  \big\| v_{2}\big\|_{Y_{N} (\R)}  
\)
and \eqref{eq:Y-Wiener-bilinearizations} follows.
\end{proof}

\section{Conclusion - Proof of main theorems}\label{sec:conclusion}

In this section, we use \Cref{thm:intro-deterministic-controlled-well-posedness} and \Cref{thm:random-multilinear-directional-bounds} 
proved respectively in \Cref{sec:deteministic-fixed-point} 
and \Cref{sec:multilinear-probabilistic-estimates} to deduce the
remaining theorems stated in \Cref{sec:introduction}.

\begin{proof}[Proof of \Cref{thm:random-multilinear-classical-bounds}]
Replace $S$ in the assertion of the theorem by $S'$. 
Then the assumption \eqref{eq:asoaz} 
is satisfied almost surely with $S$ replaced by $\tilde{S}$, after an application of \Cref{thm:random-multilinear-directional-bounds} with $S = \tilde{S} + \epsilon$, where $\epsilon > 0$ is so small that $S + \epsilon \leq S'$ and $\tilde{S} + \epsilon < \mu(k, S)$. 

The first assertion follows from \eqref{eq:time-continuity-sca}
and \Cref{thm:random-multilinear-directional-bounds} and the second one follows from 
 the ``Time-continuity and scattering of the multilinear data'' claim of \Cref{thm:intro-deterministic-controlled-well-posedness}. 
\end{proof}

\begin{proof}[Proof of Theorems \ref{thm:random-local-wellposedness},
\ref{thm:random-scattering}, and
\ref{thm:random-remainder-regularity}]
We proved the assertions for $S$ replaced by $S' \in (S_{\min}, \xs_c)$. 
For the proof of \Cref{thm:random-remainder-regularity} we also fix $M$ such that 
\(\mu(M+1,S')>\xs_{c}\) and any $\xs < \mu(M+1,S')$. Next, we choose $S \in (S_{\mrm{min}}, S')$ such that \(\mu(M+1,S)>\xs_{c}\). 
Finally we choose \(0<\epsilon\lesssim_{M,S,S'} 1\) and  \(0<\epsilon_{0}\lesssim_{\epsilon,M,S,S'}1\) such that \(\epsilon<\frac{\min(S,1,S'-S)}{3M}\)  and the conditions of 
\Cref{thm:intro-deterministic-controlled-well-posedness} and \Cref{thm:random-multilinear-directional-bounds} are fulfilled.

According to \Cref{thm:random-multilinear-directional-bounds} with $S$
replaced by \(S+\epsilon\), and because \(\mu(k,S+\epsilon)\geq\mu(k,S)+\epsilon\), we obtain for all
\(k\leq M\) that
\[\eqnum\label{eq:swres}
\mbb{P}\Big(\|\rz_{k}\|_{Y^{\mu(k,S)+\epsilon}(\R)}>\lambda\Big)\leq
C \exp\Bigg(- \frac{ \lambda^{\frac{2}{k}}}{C\|f\|_{H^{S+2\epsilon}_{x} (\R^{d})}^{2}}\Bigg)\,.
\]
Thus, we focus only on a probability set \(\Omega_{0}\) with \(\mbb{P}(\Omega_{0})=1\) for which
\[
\max_{k\leq M}\|\rz_{k}\|_{Y^{\mu(k,S)+\epsilon}(\R)}<+\infty .
\]
Define 
\begin{equation*}
    R := \max\{\|\rf\|_{H^{S+\epsilon}_{x}}, \max_{k\in\{1,\ldots,M\}}  \|\rz_{k}\|_{Y^{\mu(k,S)+\epsilon}(\R)}\}
\end{equation*}
and the random time
\[ \eqnum\label{eq:srmtm}
T=
\begin{dcases}
\frac{1}{C}\big(1+\max_{k\leq M}\|\rz_{k}\|_{Y^{\mu(k,S)+\epsilon}(\R)}\big)^{-\sfrac{3}{c}} &\text{if } R\geq \delta_{0}\,,
\\
+ \infty &\text{if }R < \delta_{0} \,,
\end{dcases}
\]
with \(\delta_{0}\), \(c\), and \(C\) as in
\Cref{thm:intro-deterministic-controlled-well-posedness}. The local
existence of a random solution \(u=\rz_{\leq M}+u^{\#}_{M}\) on \([0,T)\)
with
\[
\|u^{\#}_{M}\|_{X^{\xs}([0,T))}+\|u^{\#}_{M}\|_{C^{0}\big([0,T);H^{\xs}_{x}(\R^{d})\big)}<\infty
\] is guaranteed by the local and global existence claims
\Cref{thm:intro-deterministic-controlled-well-posedness}. Such a
solution is unique among those satisfying \(u^{\#}_{M}\in
C^{0}\big([0,T);H^{\xs}_{x}(\R^{d})\) as claimed by the uniqueness part
of
\Cref{thm:intro-deterministic-controlled-well-posedness}. Furthermore,
since \(\xs<\mu(M+1,S')\) is arbitrary, the uniqueness claim of \Cref{thm:intro-deterministic-controlled-well-posedness} shows that
\[
\|u^{\#}_{M}\|_{C^{0}\big([0,T);H^{\tilde{\xs}}_{x}(\R^{d})\big)}<\infty
\]
for any \(\tilde{\xs}<\mu(M+1,S')\), as required by
\Cref{thm:random-remainder-regularity}.

The measurability of \(u\) follows from the fact that \(u^{\#}_{M}\)
depends continuously on the multilinear data \((\rf,\vec{\rz}_{M})\).

To establish almost-certain continuity of the full solution \(u\) we
recall that \Cref{thm:random-multilinear-classical-bounds} applied
with \(S+\epsilon\) in place of \(\epsilon\) shows that almost surely
\[
\|\rz_{k}\|_{C^{0}\big([0,\infty);H^{\mu(k,S)}_{x}(\R^{d})\big)}<\infty\,,
\]
while for \(k=1\) it holds that \(\|\rz_{1}\|_{C^{0}\big([0,\infty);H^{S'}_{x}(\R^{d})\big)}<\infty\) since, using \Cref{cor:probability-of-wiener-randomization}, we get
\[
\|\rz_{1}(t)\|_{H^{S'}_{x}(\R^{d})}=\|e^{it\Delta}\rf\|_{H^{S'}_{x}(\R^{d})}
=\|\rf\|_{H^{S'}_{x}(\R^{d})} <\infty\,.
\]
Since we chose \(S\) such \(\mu(k,S)>S'\) for \(k\geq2\), the triangle inequality implies
\[
\|u\|_{C^{0}\big([0,\infty);H^{S'}_{x}(\R^{d})\big)}\leq
\begin{aligned}[t]
\|\rz_{1}\|_{C^{0}\big([0,\infty);H^{S'}_{x}(\R^{d})\big)}&+ \sum_{k=2}^{M}\|\rz_{k}\|_{C^{0}\big([0,\infty);H^{S'}_{x}(\R^{d})\big)}
\\
&+\|u^{\#}_{M}\|_{C^{0}\big([0,\infty);H^{S'}_{x}(\R^{d})\big)}<\infty\,,
\end{aligned}
\]
as required. 

Next, since \( T \geq \frac{1}{C}\big(1+\max_{k\leq M}\|\rz_{k}\|_{Y^{\mu(k,S)+\epsilon}(\R)}\big)^{-\sfrac{3}{c}}\), according to the definition \eqref{eq:srmtm}. The probability estimates \eqref{eq:swres} imply for any $\lambda \geq \delta_{0}$ and $k \leq M$ that 
\[
    \mbb{P}\Big(\|\rz_{k}\|_{Y^{\mu(k,S)+\epsilon}(\R)}>\lambda\Big)\leq
C \exp\Bigg(- \frac{ \lambda^{\frac{2}{M}}}{C\|f\|_{H^{S+2\epsilon}_{x} (\R^{d})}^{2}}\Bigg)\,.
\]
Then \eqref{eq:random-time-probability-distribution} follows after standard algebraic manipulations. 

To prove \Cref{thm:random-scattering}, choose $\delta_0$ as in \Cref{thm:intro-deterministic-controlled-well-posedness} and note 
that by choosing appropriate $\epsilon, \epsilon_0, \xs_c$, and $M$ in   \Cref{thm:intro-deterministic-controlled-well-posedness}, then $\delta_0$ depends only on $S$ and $\xs_c$.  
Set
\[
\Omega_{\mrm{glob}} = \{\omega:
\|\rf\|_{H^{S+\epsilon}_{x}}\leq\delta_{0}, 
\|\rz_{k}\|_{Y^{\mu(k,S)+\epsilon}(\R)}\leq\delta_{0}\quad k\in\{1,\ldots,M\}\}
\]
and note that by \eqref{eq:srmtm}, $T = \infty$
on $\Omega_{\mrm{glob}}$. Then \eqref{eq:swres} implies
\[
\mbb{P}\Big(\|\rz_{k}\|_{Y^{\mu(k,S)+\epsilon}(\R)}\leq \delta_0 \Big)\geq
1 - C \exp\Bigg(- \frac{ \delta_0^{\frac{2}{k}}}{C\|f\|_{H^{S+2\epsilon}_{x} (\R^{d})}^{2}}\Bigg)\,,
\]
and therefore 
\[
\mbb{P}\Big(\max_{k = 1, \ldots, M}\|\rz_{k}\|_{Y^{\mu(k,S)+\epsilon}(\R)}\leq \delta_0 \Big)\geq
1 - CM\exp\Bigg(- \frac{ \delta_0^{\frac{2}{k}}}{C\|f\|_{H^{S+2\epsilon}_{x} (\R^{d})}^{2}}\Bigg)\,.
\]
Moreover, by \eqref{eq:probability-of-wiener-randomization} we have 
\[
 \mbb{P}\Big(  \|\rf\|_{H^{S+\epsilon}_{x}}\leq\delta_{0} \Big)
 \geq 1 - C\exp\Bigg(- \frac{\delta_0^{\frac{2}{k}}}{C\|f\|_{H^{S+2\epsilon}_{x} (\R^{d})}^{2}}\Bigg)
\]
and \eqref{eq:probabilistic-global-probability} follows.

Scattering for the random multilinear terms \(\rz_{k}\) follows from \Cref{thm:random-multilinear-classical-bounds} while scattering in \(H^{\xs}_{x}(\R^{d})\) of \(u^{\#}_{M}\) follows from \eqref{eq:remainder-scattering} since \(T=+\infty\) on 
$\Omega_{\mrm{glob}}$. 
\end{proof}

\appendix

\section{Multilinear chaos estimates}\label{sec:multilinear-chaos-appendix}

We prove the following multilinear chaos estimate. Compared to \Cref{lem:hypercontractivity} we index the random variables by \(\Z\) in place of \(\Z^{d}\), however this has no effect on the proof and merely simplifies the notation. 

\begin{theorem}\label{thm:hypercontractivity-main}
Let $( g_{k} )_{k\in\Z} $ be a collection of complex-valued
i.i.d random variables, such that
\begin{equation}\label{eq:g-finite-moments}
 \E \big[| g_{k} |^p\big] < \infty \qquad \text{for all} \quad  p \in [1, \infty) . 
\end{equation}
Let $c :  \Z^{n} \rightarrow \C$ be a function with finite support, and fix
$\alpha_{j} \in \N$ for $j \in \{ 1, \ldots, n \}$. Then, the random variable
\[
H = \sum_{\vec{k} \in \Z^{n}} c (\vec{k}) g_{k_{1}}^{\alpha_{1}} \ldots g_{k_{d}}^{\alpha_{d}}
\]
satisfies the estimate
\[
\Big( \E \big[| H |^p \big]\Big)^{1 / p} \leq C_{n, p, \vec{\alpha}} \Big( \E\big[ | H |^2\big] \Big)^{1 / 2}\qquad \textrm{ for all } \quad  p \in [2, \infty)\,,
\]
where the constant $C_{n,p, \vec{\alpha},g}$ depends only on \(n\),
\(p\), \(\vec{\alpha}\), and the distribution of the variables
\( (g_{k})_{k\in\N}\), but not on the function $c$ nor on the size of its support.
\end{theorem}

Let $\mu_{g}$ be the distribution (push-forward probability measure) on
$\C$ of any of the random variables $g_{n}$, and let
\(L_{g}\in(\N\setminus\{0\})\cup\{+\infty\}\) be the dimension of
\(\linspan\big(\{x^{\alpha}\}_{\alpha\in\N}\big)\) in
\(L^{2}(\C,\mu_{g})\). We use the notation
\(\N_{L} := \{n\in\N\st |n|<L\}\) and
\(\Z_{L} := \{n\in\Z\st |n|<L\}\).  Let
$(P^{\alpha})_{\alpha \in \N_{L_{g}}}$ be the a collection of
\(L^{2}(\C,\mu_{g})\)-orthonormal polynomials such that
\(x^{\alpha}\in\linspan(\{P^{\alpha'}\}_{\alpha'\leq \alpha})\). One can construct
$(P^{\alpha})_{\alpha \in \N_{L_{g}}}$ via the standard Gram-Schmidt algorithm
applied to \((x^{\alpha})_{\alpha\in\N_{L_{g}}}\). Since the random variables
\(g_{k}\) are complex-valued, by convention we set
\(P^{-\alpha}=\bar{P^{\alpha}}\). This way
$(P^{\alpha})_{\alpha \in \N_{L_{g}}}$ are orthonormal, as well as
$(P^{-\alpha})_{\alpha \in \N_{L_{g}}}$, however \(P^{\alpha}\) may fail to be orthogonal
to \(P^{\beta}\) when \(\beta<0<\alpha\). Finally, note tha \(P^{0}\) coincides with a
constant function \(\mu_{g}\)-almost everywhere, so for any
\(\alpha\neq0\), the polynomial \(P^{\alpha}\) is orthogonal to constants.

For any $\vec{\alpha} \in (\Z_{L_{g}}\setminus\{0\})^{n}$, and \(\vec{k}\in\Z^{n}\) we set
\[
P^{\vec{\alpha}}_{n} (g_{\vec{k}}) \eqd \prod_{j = 1}^{n} P^{\alpha_{j}} (g_{k_{j}})\,.
\]
For any \(n\in\N\setminus\{0\}\) let
\[
\Delta^{n}\eqd\{\vec{k}\in\N^{d}\st k_{1}<k_{2}<\ldots<k_{d}\}
\]
denote the strict \(n\)-simplex of vectors with integer coefficients.

We need a multi-orthogonality statement, for which we introduce the
following notation.  Given $q \in \N\setminus\{0\}$ and
$\vec{n} = (n_{1}, \ldots, n_{q}) \in (\N\setminus\{0\})_{+}^{q}$, and
\(\vec{k}_{r}\in\Delta^{n_{r}}\), for each $r\in\{1,\ldots,q\}$ we say that
\((\vec{k}_{1}, \ldots, \vec{k}_{q}) \in\ms{Z}^{\vec{n},q}\) if for some
$r_{*} \in \{1,\ldots,q\}$ and $j_{*} \in \{1,\ldots,n_{r_{*}}\}$ one has
\((\vec{k}_{r})_{j}\neq(\vec{k}_{r_{*}})_{j_{*}}\) for all
\((r,j)\neq(r_{*},j_{*})\), that is, if there exists an integer that
appears only once among the numbers
\( \big\{ (\vec{k}_{r})_{j}\st r\in\{1,\ldots,q\},\, j\in\{1,\ldots,n_{r}\}\big\}\). Note that if
\(q=1\) and \(\vec{k}\in\Delta^{n}\) then, automatically, \(\vec{k}\in\ms{Z}^{n,1}\).

\begin{lemma}\label{lem:P-multi-orthogonality}
Fix \(q\in\N\setminus\{0\}\), \(\vec{n} \in(\N\setminus\{0\})^q\), and
$\vec{\alpha}_{r} \in (\N_{L_{g}}\setminus\{0\})^{n}$ for each $r\in\{1, \ldots, q\}$.

There exists a constant \(C_{\vec{n}, (\vec{\alpha}),q,g}>0\) such that
\[\eqnum\label{eq:P-multi-boundedness}
\E\Big[\Big|\prod_{r=1}^{q}P^{\vec{\alpha}_{r}}_{n_{r}}(g_{\vec{k}_{r}})\Big|\Big]\leq C_{\vec{n}, (\vec{\alpha}),q,g}
\]
for any choice of \(\vec{k}_{r}\in\Delta^{n_{r}}\) for each
\(r\in\{1,\ldots,q\}\). Furthermore, if
\((\vec{k}_{r})_{r\in\{1,\ldots,q\}}\in\ms{Z}^{\vec{n},q}\) then
\[\eqnum\label{eq:P-multi-orthogonality}
\E\Big[\prod_{r=1}^{q}P^{\vec{\alpha}_{r}}_{n_{r}}(g_{\vec{k}_{r}})\Big]=0.
\]

Finally, in the specific case when \(q=2\), if  for any \(n_{1},n_{2}\in\N\setminus\{0\}\), \(\vec{k}\in\Delta^{n_{1}}\), \(\vec{l}\in\Delta^{n_{2}}\), and \(\vec{\alpha}\in(\N\setminus\{0\})^{n}\) and \(\vec{\beta}\in(\N\setminus\{0\})^{n}\), it holds that 
\[\eqnum\label{eq:P-2-generalized-orthonormality}
\E\big[P^{\vec{\alpha}}_{n_1}(X_{\vec{k}}) \bar{P^{\vec{\beta}}_{n_2}(X_{\vec{l}})}\big]=
\begin{dcases}
 1 &\text{if } n_{1}=n_{2},\,\vec{k}=\vec{l},\text{ and }{\vec{\alpha}=\vec{\beta}},
 \\
0 & \text{otherwise.}
\end{dcases}
\]
\end{lemma}

Note that if $q = 2$, then \eqref{eq:P-2-generalized-orthonormality} strengthens \eqref{eq:P-multi-boundedness} by providing a lower bound, and generalizes  \eqref{eq:P-multi-orthogonality} by establishing orthogonality when \(\vec{k}=\vec{l}\), \(n_{1}=n_{2}\) (and thus \((\vec{k},\vec{l}) \not \in \ms{Z}^{(n_{1},n_{2}),2}\)), but  \(\vec{\alpha}\ne\vec{\beta}\).

\begin{proof}[Proof of \Cref{lem:P-multi-orthogonality}]
First, by using the Hölder's inequality and \eqref{eq:g-finite-moments}, we obtain
\begin{align*}
 \Big| \E \big[ \prod_{r = 1}^{q} P^{\vec{\alpha}_{r}}_{n_{r}} (g_{\vec{k}_{r}}) \big] \Big| 
&=
\Big| \E \big[ \prod_{r = 1}^{q}\prod_{j = 1}^d
P^{(\alpha_{r})_{j}} (g_{(\vec{k}_{r})_{j}})
\big] \Big|
\\
&\leq 
 \prod_{r = 1}^{q}\prod_{j = 1}^{n_{r}}  \E \big[| P^{(\alpha_{r})_{j}} (g_{(\vec{n}_{r})_{j}}) |^{n_{+}}\big]^{\frac{1}{n_{+}}} 
\leq C_{d, q,(\vec{\alpha}),g} \,,
\end{align*}
where \(n_{+}\eqd\sum_{r=1}^{q}n_{r}\). In the last inequality, we used the assumption \eqref{eq:g-finite-moments} and that the degrees of the finitely many polynomials $P^{(\alpha_{r})_{j}}$ are bounded from above. 

To show \eqref{eq:P-multi-orthogonality} we note by the 
assumption that there exists some \(r_{*}\in\{1,\ldots,q\}\)
and \(j_{*} \in \{1,\ldots,d_{r_{*}}\}\) such that  \((\vec{n}_{r})_{j}\neq
(\vec{n}_{r_{*}})_{j_{*}}\) unless \((r,j)=(r_{*},j_{*})\). Then, by independence of the variables \((g_{n})_{n\in\Z}\), we have that 
\[
\E \big[ \prod_{r = 1}^{q} P^{\vec{\alpha_{r}}}_{n_{r}} (g_{\vec{k}_{r}}) \big] 
\begin{aligned}[t]
&=
\E \big[ \prod_{r = 1}^{q } \prod_{j=1}^{n_{r}} P^{(\alpha_{r})_{j}} (g_{(\vec{k}_{r})_{j}}) \big]
\\
& =
\E[P^{(\alpha_{r_{*}})_{j_{*}}}(g_{(\vec{k}_{r_{*}})_{j_{*}}})]
\,\E \Big[ \prod_{\nqquad \mathrlap{(r,j)\neq(r_{*},j_{*})}} P^{(\alpha_{r})_{j}} (g_{(\vec{k}_{r})_{j}}) \Big]
=0\,,
\end{aligned}
\]
where the last equality holds since \((\alpha_{r_{*}})_{j_{*}}\neq0\) and thus \(\E[P^{(\alpha_{r_{*}})_{j_{*}}}(g_{(\vec{k}_{r_{*}})_{j_{*}}})]=0\) by the $\mu_{g}$ orthogonality of $P^{\alpha_{j_{*}}}$ to constants. Hence, the proof of \eqref{eq:P-multi-orthogonality} follows.

To show \eqref{eq:P-2-generalized-orthonormality} note that if $n_{1} \neq n_{2}$ or $\vec{k} \neq \vec{l}$, then $(\vec{k}, \vec{l}) \in \ms{Z}^{ (n_{1}, n_{2}),2}$ and thus the orthogonality claim follows from \eqref{eq:P-multi-orthogonality}. On the other hand, if $n_{1} = n_{2}$ and  $\vec{k} = \vec{l}$, then
\[
\E\big[P^{\vec{\alpha}}_{n_{1}}(g_{\vec{k}}) \bar{P^{\vec{\beta}}_{n_{2}}(g_{\vec{l}})}\big]
= \prod_{j = 1}^{n_{1}} \E\big[P^{\alpha_{j}}(g_{k_{j}}) \bar{P^{\beta_{j}} (g_{k_{j}})}\big].
\]
If $\vec{\alpha} \neq \vec{\beta}$, that is if $\alpha_{j} \neq \beta_{j}$ for some $j =\{1, \ldots,
n_{1}\}$, then \(\E\big[P^{\alpha_{j}}(g_{k_{j}}) \bar{P^{\beta_{j}} (g_{k_{j}})}\big]=0\)
by orthogonality, thus showing
\eqref{eq:P-2-generalized-orthonormality}. If \(\vec{\alpha}=\vec{\beta}\), then, by
normalization, one has \(\E\big[P^{\alpha_{j}}(g_{k_{j}}) \bar{P^{\beta_{j}}
(g_{k_{j}})}\big]=1\) for each \(j\in\{1,\ldots,n\}\) and
\eqref{eq:P-2-generalized-orthonormality} follows.
\end{proof}

Next, we prove an Brascamp-Lieb -type deterministic estimate which is
essential for the proof of \Cref{thm:hypercontractivity-main} once independence has been accounted for. 

For $n \in \N\setminus\{0\}$ let
$\mc{E}^{n} := (\hat{e}_{1}, \hat{e}_{2}, \ldots, \hat{e}_{n}) \in \Z^{n}$ be
the standard basis of of $\Z^{n}$, that is, for any
$l \in \{1, \ldots, n\}$ one has
\[
(\hat{e}_{l})_{j} = \delta_{l, j} = \begin{cases}
1 & \text{if } j = l,\\
0 & \text{otherwise} .
\end{cases}
\]
We omit the dependence of $\hat{e}_{l}$ on \(n\) as it is clear from context.

\begin{definition}\label{def:pltk}
Let $n,q \in \N\setminus\{0\}$ and $N \in \{ 1, \ldots, q d \}$. We denote by
$\ms{P}^{q}_{N \rightarrow n}$ the collection of $q$-tuples of maps
$\vec{\pi} = \big( \pi_{r} \big)_{r \in \{ 1, \ldots, q \}}$ such that for each
$r \in \{ 1, \ldots, q \}$ the following conditions hold:
\begin{enumerate}
\item[(i)] The map $\pi_{r} : \Z^{N} \rightarrow \Z^{n}$ is $\Z$-linear.
\item[(ii)] For each $\hat{e} \in \mc{E}^{N}$ either $\pi_{r} (\hat{e}) = 0$ or
$\pi_{r} (\hat{e}) \in \mc{E}^{n}$.
\item[(iii)] For any \(\hat{e},\hat{e}'\in\mc{E}^{N}\) with 
$\pi_{r}(\hat{e}) = \pi_{r}(\hat{e}') \neq 0$ we have $\hat{e} = \hat{e}'$. Equivalently, 
$\hat{e} \mapsto
\pi_{r} (\hat{e})$ is injective on $\left\{ \hat{e} \in \mc{E}^{N} \st \pi_{r}
(\hat{e}) \neq 0 \right\}$.
\end{enumerate}
We say $\vec{\pi} \in \ms{P}^{q}_{N \rightarrow n, 2}$ if $\vec{\pi} \in
\ms{P}^{q}_{N \rightarrow n}$ and
\begin{enumerate}
\item[(iv)] for each $\hat{e} \in \mc{E}^{N}$ there are at least two distinct indices $r \in \{ 1, \ldots, q \}$ with $\pi_{r} (\hat{e}) \neq 0$.
\end{enumerate}
\end{definition}

\begin{theorem}\label{thm:hcdb}
Let $n\in \N\setminus\{0\}$ and \(q\in\N\) with $q \geq 2$. Then for any
$N \in \{ 1, \ldots, q n \}$, any
$\vec{\pi}=(\pi_{1}, \ldots, \pi_{q}) \in \ms{P}^{q}_{N \rightarrow n, 2}$, and any functions
$c_{r} : \Z^{n} \rightarrow \C$, $r \in \{ 1, \ldots, q \}$ it holds that
\[
\sum_{\vec{k} \in \Z^{N}} \Big( \prod_{r= 1}^{q} c_{r} (\pi_{r} (\vec{k})) \Big) \leq \prod_{r = 1}^{q} \| c_{r} \|_{l^2 ( \Z^{n} )} .
\]
\end{theorem}

\begin{proof}
We prove the claim by induction on the dimension \(N\) of the summation index \(\vec{k} \in \Z^{N}\).

For the base step \(N = 1\), let $\mc{S}_{1} \eqd \big\{r: \pi_{r}(\hat{e}_{1}) \neq 0\big\}$ and
denote by $\mc{S}_{1}^c$ the complement of $\mc{S}_{1}$ in $\{1, \ldots, q\}$. Then, $|\mc{S}_{1}| \geq 2$ by assumption (iv) of \Cref{def:pltk} and the
Hölder inequality implies that 
\[
\sum_{k \in \Z} \Big( \prod_{r = 1}^q c_{r} (\pi_{r} (k))\Big) 
\begin{aligned}[t]
& = \sum_{k \in \Z} \Big( \prod_{r \in \mc{S}_{1}} c_{r} (k)\Big) \times \Big( \prod_{r \in \mc{S}^c_{1}} c_{r} (0)\Big) 
\\
& \leq  \prod_{r \in \mc{S}_{1}} \Big(\sum_{k \in \Z} |c_{r} (k)|^{\mc{S}_{1}}\Big)^{\frac{1}{|\mc{S}_{1}|}}
\prod_{r \in \mc{S}^c_{1}} \|c_{r}\|_{l^2( \Z^{d} )}
\\
& \leq  \prod_{r \in \mc{S}_{1}} \Big(\sum_{k \in \Z} |c_{r} (k)|^{2}\Big)^{\frac{1}{2}}
\prod_{r \in \mc{S}^c_{1}} \|c_{r}\|_{l^2( \Z^{d} )}
\\
& =  \prod_{r = 1}^q \| c_{r} \|_{l^2( \Z^{d} )} \,, 
\end{aligned}
\]
as desired. 

Next, we assume that the assertion holds for \(N - 1\) for some \(N \geq 2\)
and let us prove the statement for $N$. Let $\mc{S}_{N} \eqd\big\{r:
\pi_{r}(\hat{e}_{N}) \neq 0\big\}$ and denote by $\mc{S}_{N}^c$ the complement of
$\mc{S}_{N}$ in $\{1, \ldots, q\}$. By assumption (iv) of \Cref{def:pltk}, it
holds that $|\mc{S}_{N}| \geq 2$. Thus, by the linearity of $\pi_{r}$, we have
\[
\sum_{\vec{k} \in \Z^{N}} \Big( \prod_{r = 1}^q c_{r} (\pi_{r} (\vec{k})) \Big)
=
\sum_{\vec{l} \in \Z^{N - 1}} \begin{aligned}[t]
& \Big( \sum_{k \in \Z} \prod_{r \in \mc{S}_{N}} c_{r} \big(\pi_{r} (\vec{l},0) + k \pi_{r}(\hat{e}_{N})\big) \Big)
\\
& \quad\times \Big( \prod_{r \in \mc{S}_{D}^c} c_{r} (\pi_{r} (\vec{l},0))\Big).
\end{aligned}
\]
For each \(r\in\{1,\ldots,q\}\) we define the maps $\tilde{\pi}_{r} : \Z^{N - 1} \rightarrow \Z^{n}$ by setting
\[
\tilde{\pi}_{r} (\vec{l}) = \pi_{r} \big(\vec{l}, 0\big) \,,
\]
and we define the index $j(r) \in \{1, \ldots, n\}$ so that $\pi_{r}(\hat{e}_{N})=\hat{e}_{j(r)} $. Then, for any $\vec{l} \in \Z^{N - 1}$ the Hölder inequality and $|\mc{S}_{N}| \geq 2$ imply that 
\[
\begin{aligned}[t]
& \sum_{k \in \Z} \prod_{r \in \mc{S}_{N}} c_{r} (\pi_{r} (\vec{l}, 0) + k \pi_{r}(\hat{e}_{N}))
= \sum_{k \in \Z} \prod_{r \in \mc{S}_{N}} c_{r} (\tilde{\pi}_{r} (\vec{l}) + k \pi_{r}(\hat{e}_{N}))
\\
& \qquad \leq \prod_{r \in \mc{S}_{N}}
\Big( \sum_{k \in \Z} | c_{r} (\tilde{\pi}_{r} (\vec{l}) + k \pi_{r}(\hat{e}_{N})) |^{|\mc{S}_{N}|} \Big)^{\frac{1}{|\mc{S}_{N}|}}
\leq \prod_{r \in\mc{S}_{N}} g_{r} (\tilde{\pi}_{r} (\vec{l})) \,.
\end{aligned}
\]
The functions $g_{r} : \Z^{n} \rightarrow [0, + \infty)$ are given
\[
g_{r} (\vec{m}) \eqd
\begin{dcases}
\Big( \sum_{k \in \Z} \Big| c_{r} \big(\vec{m}+k\hat{e}_{j(r)}\big) \Big|^2 \Big)^{\frac{1}{2}} & \text{if } (\vec{m})_{j(r)}=0,
\\
0 & \text{otherwise}.
\end{dcases}
\]
Note that $\big(\tilde{\pi}_{r}(\vec{l})\big)_{j(r)} =  \big(\pi_{r}((\vec{l}, 0)\big)_{j(r)} = 0$ by (iii) of \Cref{def:pltk}, and consequently 
\begin{equation*}
g_{r} (\tilde{\pi}_{r} (\vec{l})) = 
\Big( \sum_{k \in \Z} | c_{r} (\tilde{\pi}_{r} (\vec{l}) + k \pi_{r}(\hat{e}_{N})) |^{2}\Big)^{\frac{1}{2}} \,.
\end{equation*}
Therefore, 
\[ 
\sum_{\vec{k} \in \Z^{N}} \Big( \prod_{r = 1}^q c_{r} (\pi_{r} (\vec{k}))  \Big)
\leq \sum_{\vec{l} \in \Z^{N - 1}} \Big( \prod_{r \in \mc{S}_{N}} g_{r} (\tilde{\pi}_{r} (\vec{l})) \Big)
\Big( \prod_{r \in \mc{S}_{N}^c} c_{r} (\tilde{\pi}_{r} (\vec{l})) \Big) \,.
\]
Next, we verify that $(\tilde{\pi}_{1}, \ldots, \tilde{\pi}_{q}) \in \ms{P}^{q}_{N-1 \rightarrow n, 2}$.
First the $\Z$-linearity of $\tilde{\pi}_{r}$ required for (i) follows from the
$\Z$ linearity of $\pi_{r}$. To prove (ii), for any $\hat{e} \in \mc{E}^{N-1}$,
we have $(\hat{e}, 0) \in \mc{E}^{N}$, and therefore $\tilde{\pi}_{r}(\hat{e}) =
\pi_{r}(\hat{e}, 0) \in \mc{E}^{n} \cup \{0\}$, as desired. Also, 
if $\tilde{\pi}_{r}(\hat{e}) = \tilde{\pi}_{r}(\hat{e}') \neq 0$, then $\pi_{r}((\hat{e},0)) = \pi_{r}((\hat{e}',0)) \neq 0$. Since $\pi_{r}$ satisfies (iii) of \Cref{def:pltk}, then $(\hat{e},0) = (\hat{e}',0)$, and therefore $\hat{e} = \hat{e}'$ and (iii) follows. 
Finally, for any $\hat{e} \in \mc{E}^{N-1}$ we have $(\hat{e}, 0) \in \mc{E}^{N}$, and by assumption (iv)
for $(\pi_{1}, \ldots, \pi_{q})$, there are at least two indices $r \in \{1, \ldots, q\}$
such that $0 \neq \pi_{r}(\hat{e}, 0) = \tilde{\pi}_{r}(\hat{e})$, which establishes
(iv).

Since $(\tilde{\pi}_{1}, \ldots, \tilde{\pi}_{q}) \in \ms{P}^{q}_{N-1 \rightarrow n, 2}$, then by the
induction hypothesis we have
 \[
\sum_{\vec{k} \in \Z^{D}} \Big( \prod_{r = 1}^{q} c_{r} (\pi_{r} (\vec{k}))\Big)
\leq \Big( \prod_{r \in \mc{S}_{D}} \| g_{r} \|_{l^2( \Z^{d})} \Big)
\Big( \prod_{r \in \mc{S}_{D}^{c}} \| c_{r} \|_{l^2 ( \Z^{d} )} \Big) .
\]
Then, Fubini's Theorem yields
\[\begin{aligned}[t]
\| g_{r} \|_{l^2( \Z^{d } )}
& =
\Big\| \Big( \sum_{k \in \Z} | c_{r} (k_{1},\ldots, k_{j(r)-1},k, k_{j(r) + 1}, \ldots k_{d}) |^2 \Big)^{\frac{1}{2}} \Big\|_{l^2 ( \Z^{d - 1} )}
\\
&= \| c_{r}\|_{l^2 ( \Z^{d} )}\,,
\end{aligned}
\]
concluding our proof.
\end{proof}

The next lemma is close to the statement of \Cref{thm:hypercontractivity-main}, with the exception that in place of powers of our random variables we use the polynomials \(P^{\alpha}\) and the sum is taken over \(\Delta^{n}\) in place of \(\Z^{n}\)

\begin{lemma} \label{lem:H-strict-simplex-moments} Fix \(n\in\N\setminus\{0\}\) and
\(\vec{\alpha} \in (\N_{L_{g}}\setminus\{0\})^{n}\). Given any function $c : \Delta^{n} \to \C$ with finite support, then the random variable
\[
H = \sum_{\vec{k}\in\Delta^{d}} c (\vec{n}) P^{\vec{\alpha}}_{n} (g_{\vec{k}}) 
\]
satisfies
\begin{equation}\label{eq:H-L2}
\E \big[| H |^2\big] = \sum_{\vec{k}\in\Delta^{n}} | c(\vec{k}) |^2
\end{equation}
and for any $p \in \N$ there exists a constant \(C_{p,\vec{\alpha}}\) such that 
\begin{equation}\label{eq:H-Lp} 
\E \big[| H |^{2 p}\big] \leq C_{2p,d,\vec{\alpha}, g} \,\big( \E \big[| H |^{2}\big] \big)^{p}
= \Big( \sum_{\vec{k}\in\Delta^{n}} | c(\vec{k}) |^{2}\Big)^{p}\,.
\end{equation}
By interpolation, for any $q \in [2, \infty]$ it follows that 
\[
\big( \E\big[  | H |^{q}\big]\big)^{1 /{q}} \lesssim C_{q,d,\vec{\alpha}, g} \big( \E \big[\| H |^{2}\big] \big)^{1 / {2}} \,,
\]
where \(C_{q,n,\vec{\alpha}}\) depends only on \(q\),
\(\vec{\alpha}\), and the distribution of the variables \((g_{k})_{k\in\Z}\), 
but not on $c : \Delta^{n} \to \C$.
\end{lemma}

\begin{proof}
By the independence of $(g_{k})_{k\in\Z}$,  \eqref{eq:P-2-generalized-orthonormality} and the normalization of polynomials $P^{\alpha}$ we have 
\[
\begin{aligned}[t]
\E\big[ | H |^2\big] 
& = \sum_{\vec{k}, \vec{m} \in\Delta^{d}}  c (\vec{k}) \overline{c (\vec{m})}
\E \big[ P^{\vec{\alpha}}_{d} (g_{\vec{k}})  \bar{P^{\vec{\alpha}}_{d} (g_{\vec{m}})}\big]
 = \sum_{\vec{k}\in\Delta^{d}}  |c (\vec{k})|^2   \E [ |P^{\vec{\alpha}}_{d} (g_{\vec{k}})|^2]
\\
&= \sum_{\vec{k}\in\Delta^{d}}  |c (\vec{k})|^2\,,
\end{aligned}
\]
and \eqref{eq:H-L2} follows. To prove \eqref{eq:H-Lp}, fix $p \in \N$, \(p\geq2\) and \eqref{eq:P-multi-orthogonality} implies
\[
\begin{aligned}
\E | H |^{2 p} &= 
\sum_{\vec{k}_{1},  \cdots,  \vec{k}_{2p} \in \Delta^{d}}
\E \Big[ \prod_{r = 1}^p c (\vec{k}_{2 r - 1}) \overline{c(\vec{k}_{2 r})}
P^{\vec{\alpha}}_{n} (g_{\vec{k}_{2r-1}}) \bar{P^{\vec{\alpha}}_{n} (g_{\vec{k}_{2r}})}\Big]
\\
& \leq  \sum_{\vec{k}_{1},  \cdots,  \vec{k}_{2p} \in \Delta^{d}} \Big( \prod_{r = 1}^{2p} |c (\vec{k}_{r})| \Big)
\,\Big| \E \Big[ \prod_{r = 1}^{2 p} P^{(-1)^{r+1}\vec{\alpha}}_{n} (g_{\vec{k}_{r}})  \Big| 
\\
& =  \sum_{(\vec{k}_{1},  \cdots,  \vec{k}_{2p}) \not \in \ms{Z}^{2p,n}}
\Big( \prod_{r = 1}^{2p} |c (\vec{k}_{r})| \Big)  \E \big[ \prod_{r = 1}^{2 p} \big| P^{\vec{\alpha}}_{n} (g_{\vec{k}_{r}})\big|   
\big] \,.
\end{aligned}
\]
Next, we claim that
\[\eqnum\label{eq:bound-by-proj-summation}
\begin{aligned}
\E | H |^{2 p} &
\leq
\sum_{N=n}^{2p n}\sum_{\vec{\pi}\in\ms{P}^{2p}_{N\to n,2}}\sum_{\vec{m}\in\Delta^{N}}
\Big( \prod_{r = 1}^{2p} |c (\pi_{r}(\vec{m}))| \Big)\Big| \E \prod_{r = 1}^{2 p} P^{\vec{\alpha}}_{n} (g_{\pi_{r}(\vec{m})})  \Big|       \,,
\end{aligned}
\]
where $\ms{P}^{2p}_{N\to n,2}$ was introduced in \Cref{def:pltk}. Observe that \eqref{eq:bound-by-proj-summation} follows once we prove that  
for any fixed \((\vec{k}_{r})_{r\in\{1,\ldots,2p\}} \not \in \ms{Z}^{2p,n}\), there exists \(N\in\{1,\ldots,2 p n\}\), \(\vec{\pi}\in\ms{P}_{N\to n,2}^{2p}\), and 
\(\vec{m}\in\Delta^{N}\) such that
\[\eqnum\label{eq:pi-map-counting}
\vec{k}_{r}=\pi_{r}(\vec{m}),\quad r\in\{1,\ldots,2p\}.
\]
Let us provide details. Let \(N\eqd \Big|\big\{(\vec{k}_{r})_{j}\st r\in\{1,\ldots,2p\},\,j\in\{1,\ldots,n\}\big\}\Big| \leq 2pn
\) be the number of distinct integers appearing among the components of
the vectors \((\vec{k}_{r})_{r\in\{1,\ldots,2p\}}\). Next, define \(\vec{m}\in\Delta^{N}\) to be the vector obtained by arranging the elements of \(\big\{(\vec{k}_{r})_{j}\st r\in\{1,\ldots,2p\},\,j\in\{1,\ldots,n\}\big\}\) in strictly increasing order, that is, by
choosing the unique \(\vec{m}=(m_{1},\ldots,m_{N})\in\Delta^{N}\) such that
\[\eqnum\label{eq:x-vector-rearrangement}
\{m_{1},\ldots,m_{N}\}=\Big\{(\vec{k}_{r})_{j}\st r\in\{1,\ldots,2p\},\,j\in\{1,\ldots,n\}\Big\}.
\]
Finally, we choose \(\vec{\pi}=(\pi_{r})_{r\in\{1,\ldots,2p\}}\in\ms{P}_{N \to  n}^{2p}\) to satisfy \eqref{eq:pi-map-counting}. We define each \(\pi_{r}\), \(r\in\{1,\ldots,2p\}\), on the basis \(\hat{e}_{1},\ldots,\hat{e}_{N}\) and extend it by linearity. We set
\[\eqnum\label{eq:pi-map-construction}
\pi_{r}(\hat{e}_{\nu}) =
\begin{dcases}
\hat{e}_{j} & \text{if } m_{\nu}= (\vec{n}_{r})_{j}
\\
0 & \text{otherwise}.
\end{dcases}
\]
This definition is valid since for \(j\neq j'\) it holds that
\((\vec{k}_{r})_{j}\neq(\vec{k}_{r})_{j'}\), due to the fact that
\(\vec{k}_{r}\in\Delta^{n}\).

To prove \eqref{eq:pi-map-counting}, fix \(r\in\{1,\ldots,2p\}\)
and note that by \eqref{eq:x-vector-rearrangement},
for every \(j\in\{1,\ldots,n\}\)
there exists a unique \(\nu\in\{1,\ldots,N\}\) such that
\(m_{\nu}=(\vec{k}_{r})_{j}\). It follows from \eqref{eq:pi-map-construction}
that \(\big(\pi_{r}(m_{\nu}\hat{e}_{\nu}) \big)_{j}=m_{\nu} = (\vec{k}_{r})_{j}\), while for any \(\nu'\neq\nu\)
it holds that \(\big(\pi_{r}(m_{\nu'}\hat{e}_{\nu'}) \big)_{j}=0\), and \eqref{eq:pi-map-counting} follows. 

Let us check
that \(\vec{\pi}\in\ms{P}_{N\to n,2}^{2p}\): properties (i) and (ii) of
\Cref{def:pltk} follow from the construction of \(\pi_{r}\). To show \((iii)\), fix \(r\in\{1,\ldots,2p\}\) and suppose that 
\(\pi_{r}(\hat{e}_{\nu})=\pi_{r}(\hat{e}_{\nu'}) = \hat{e}_{j} \neq 0\) for some $j$ and \(\nu, \nu'\). Then, by \eqref{eq:pi-map-construction}, $m_\nu = (\vec{k}_{r})_{j} = m_{\nu'}$. Since $\vec{m} \in \Delta^{N}$, then $m_{\nu} = m_{\nu'}$ implies that $\nu = \nu'$ and (iii) follows. Finally, fix any $\hat{e}_\nu \in \mathcal{E}^{N}$. Since, by definition, $m_\nu$ appears among the coefficients of $(\vec{k}_{r})$ and $(\vec{k}_{1},  \cdots,  \vec{k}_{2p}) \not\in \ms{Z}^{2p,n}$, then there are at least two pairs $(r, j)$ and $(r',j')$ such that $m_\nu = (\vec{k}_{r})_{j} = (\vec{k}_{r'})_{j'}$. Then, $\pi_{r}(\hat{e}_\nu) = e_{j} \neq 0$ and $\pi_{r'}(\hat{e}_\nu) = e_{j'} \neq 0$ and (iv) is proved. 

Overall, we proved that \(\vec{\pi}\in\ms{P}_{N\to n,2}^{2p}\), and therefore \eqref{eq:bound-by-proj-summation} is established.

Then \eqref{eq:P-multi-boundedness} and \eqref{eq:bound-by-proj-summation} yield 
\[
\E \big[| H |^{2 p}\big]\leq
C_{2p,n,\vec{\alpha}} \sum_{N=n}^{2p n}\sum_{\vec{\pi}\in\ms{P}^{2p}_{N\to n,2}}\sum_{\vec{m}\in\Delta^{N}}
\Big( \prod_{r = 1}^{2p} |c (\pi_{r}(\vec{m}))| \Big)     \,.
\]

Finally, using \Cref{thm:hcdb} on the inner sum we obtain that
\[
\E \big[| H |^{2 p}\big]\leq C_{2p,n,\vec{\alpha}} \sum_{N=n}^{2p n}\sum_{\vec{\pi}\in\ms{P}^{2p}_{N\to n,2}}\|c\|_{l^{2}(\Z^{n})}^{2p}.
\]
Since, for fixed \(p,N,n\) the number of maps in \(\ms{P}^{2p}_{N\to n,2}\) is finite, we have
\[
\E\big[  | H |^{2 p}\big]\leq C_{2p,n,\vec{\alpha}}  |\ms{P}^{2p}_{N\to n,2}| \|c\|_{l^{2}(\Z^{n})}^{2p} \,,
\]
as desired. 
\end{proof}

Let us deduce \Cref{thm:hypercontractivity-main} from
\Cref{lem:H-strict-simplex-moments}.

\begin{proof}[Proof of \Cref{thm:hypercontractivity-main}]
Let \(L_{\alpha}\eqd \alpha_{1} + \ldots + \alpha_{d}\) and 
\(L_{\beta}\eqd \min\{L_{\alpha},  L_{g}\}\), where we recall that $L_{g}$ is the dimension of the space of polynomials in \(L^{2}(\C,\mu_{g})\). First, we provide the proof assuming that there exists a constant \(f_{0}\in\C\) and finitely supported functions \(f_{n',\vec{\beta}}:\Delta^{n'}\to\C\) with \(n'\in\{1,\ldots,n\}\) and \(\vec{\beta}\in(\N_{L_{\beta}}\setminus\{0\})^{n'}\), such that
\[\eqnum\label{eq:H-representation}
H=f_{0}+\sum_{n'=1}^{n}\sum_{\vec{\beta}\in(\N_{L_{\beta}}\setminus\{0\})^{n'}}\sum_{\vec{k}\in\Delta^{n'}} f_{n',\vec{\beta}}(\vec{k}) P^{\vec{\beta}}_{n'}(g_{\vec{k}}).
\]
Indeed, if \eqref{eq:H-representation} holds,  then \eqref{eq:P-2-generalized-orthonormality}
implies
\[
\E\big[H^{2}\big]=|f_{0}|^{2}+\sum_{n'=1}^{n}\sum_{\vec{\beta}\in(\N_{L_{\beta}}\setminus\{0\})^{n'}}
\sum_{\vec{k}\in\Delta^{n'}} |f_{n',\vec{\beta}}(\vec{k})|^{2} \,,
\]
where we used \eqref{eq:P-2-generalized-orthonormality} used that the random variables \(P^{\vec{\beta}}_{k}(g_{\vec{n}})\). Next, for any \(p\in[2,\infty)\)  the Minkowski inequality yields
\[
\E\big[H^{p}\big]^{1/p}\leq |f_{0}|+ \sum_{n'=1}^{n}\sum_{\vec{\beta}\in(\N_{L_{\beta}}\setminus\{0\})^{n'}}
\E\Big[\Big|\sum_{\vec{k}\in\Delta^{n'}} f_{n',\vec{\beta}}(\vec{k}) P^{\vec{\beta}}_{n'}(\vec{k})\Big|^{p}\Big]^{1/p}
\]
and from \Cref{lem:H-strict-simplex-moments} it follows that 
\[
\E\big[H^{p}\big]^{1/p}\lesssim |f_{0}|+ \sum_{n'=1}^{n}\sum_{\vec{\beta}\in(\N_{L_{\beta}}\setminus\{0\})^{n'}}
\Big(\sum_{\vec{k}\in\Delta^{n'}} |f_{n',\vec{\beta}}(\vec{k})|^{2}\Big)^{1/2},
\]
with the implicit constant dependents on $n$, \(p\), and \(L_{\beta}\), and
the distribution of $g$. Since the sums over \(k\) and  \(\vec{\beta}\) are finite, it follows from the Cauchy inequality that 
\(\E\big[H^{p}\big]^{1/p}\lesssim \E\big[H^{2}\big]^{1/2} \), as desired.

To prove that any \(H\) as in the statement of \Cref{thm:hypercontractivity-main} can be represented as \eqref{eq:H-representation} we first focus on 
the special case \(H=\prod_{j=1}^{n}g_{k_{j}}^{\alpha_{j}} \) for some $(k_{j})_{j\in\{1,\ldots,n\}}\in\Z$. 

By grouping together multiple occurrences of the same random variables $g_{k}$, for some $k$, we obtain 
\[
H=\prod_{j=1}^{\tilde{n}}g_{m_{j}}^{\gamma_{j}}=\prod_{j=1}^{n}g_{k_{j}}^{\alpha_{j}} ,
\]
for some $\tilde{n} \leq n$, some distinct $(m_{j})_{j \in\{1,\ldots,\tilde{n}\}}$ with
$\{m_{1}, \ldots m_{\tilde{n}}\} = \{k_{1}, \ldots, k_{n}\}$, and $\gamma_{1}+ \ldots + \gamma_{\tilde{n}} = \alpha_{1} + \ldots + \alpha_{n}$. In particular, $\gamma_{j} \leq L_{\alpha}$. Since $m_{1}, \ldots, m_{\tilde{n}}$ are distinct, by permuting the random variables, we can without loss of generality assume $m_{1} < m_{2} < \ldots < m_{\tilde{n}}$, and therefore $\vec{m} \in \Delta^{\tilde{n}}$. 
In addition, since the polynomials \((P^{\beta})_{\beta\in\{0,\ldots,L_{\beta}\}}\), span the
space of polynomials of degree at most $L_\alpha$ (as \(L^{2}(\C;\mu_{g})\)
functions), then for each \(j\in\{1,\ldots,\tilde{n}\}\) there exist complex
coefficients \((\lambda_{j,\beta}\in\C)_{\beta = 0}^{L_\beta}\) such that
\[
g^{\gamma_{j}}_{m}=\lambda_{j,0}+\sum_{\beta=1}^{L_{\beta}}\lambda_{j,\beta} P^{\beta}(g_{m}).
\]
for any \(m\in\Z\). Thus,
\[
H = \prod_{j=1}^{\tilde{n}}g_{m_{j}}^{\gamma_{j}}
\begin{aligned}[t]
& = \prod_{j=1}^{\tilde{n}}\Big(\lambda_{j,0}+\sum_{\beta_{j}=1}^{L_{\beta}}\lambda_{j,\beta_{j}} P^{\beta_{j}}(g_{m_{j}})\Big) 
\end{aligned}
\]
and \eqref{eq:H-representation} follows  for our special $H$ after expanding the product. The general case follows from linearity and the proof is completed.
\end{proof}

\newpage
{\sloppy \printbibliography}
\end{document}